\documentclass{amsart}

\usepackage[all]{xy}
\xyoption{poly}
\xyoption{arc}
\xyoption{curve}

\CompileMatrices

\begin{document}

\renewcommand{\theequation}{\thesection.\arabic{equation}}
\newcommand{\nc}{\newcommand}

\nc{\pr}{\noindent{\em Proof. }} \nc{\g}{\mathfrak g}
\nc{\n}{\mathfrak n} \nc{\opn}{\overline{\n}}\nc{\h}{\mathfrak h}
\renewcommand{\b}{\mathfrak b}
\nc{\Ug}{U(\g)} \nc{\Uh}{U(\h)} \nc{\Un}{U(\n)}
\nc{\Uopn}{U(\opn)}\nc{\Ub}{U(\b)} \nc{\p}{\mathfrak p}
\renewcommand{\l}{\mathfrak l}
\nc{\z}{\mathfrak z} \renewcommand{\h}{\mathfrak h}
\nc{\m}{\mathfrak m}
\renewcommand{\k}{\mathfrak k}
\nc{\opk}{\overline{\k}}
\nc{\opb}{\overline{\b}}
\nc{\e}{{\epsilon}}
\nc{\ke}{{\bf k}_\e}
\nc{\Hk}{{\rm Hk}^{\gr}(A,A_0,\e )}
\nc{\gr}{\bullet}
\nc{\ra}{\rightarrow}
\nc{\Alm}{A-{\rm mod}}
\nc{\DAl}{{D}^-(A)}
\nc{\HA}{{\rm Hom}_A}

\newtheorem{theorem}{Theorem}{}
\newtheorem{lemma}[theorem]{Lemma}{}
\newtheorem{corollary}[theorem]{Corollary}{}
\newtheorem{conjecture}[theorem]{Conjecture}{}
\newtheorem{proposition}[theorem]{Proposition}{}
\newtheorem{axiom}{Axiom}{}
\newtheorem{remark}{Remark}{}
\newtheorem{example}{Example}{}
\newtheorem{exercise}{Exercise}{}
\newtheorem{definition}{Definition}{}

\renewcommand{\thetheorem}{\thesection.\arabic{theorem}}

\renewcommand{\thelemma}{\thesection.\arabic{lemma}}

\renewcommand{\theproposition}{\thesection.\arabic{proposition}}

\renewcommand{\thecorollary}{\thesection.\arabic{corollary}}

\renewcommand{\theremark}{\thesection.\arabic{remark}}

\renewcommand{\thedefinition}{\thesection.\arabic{definition}}

\title{Conjugacy classes in Weyl groups and q-W algebras}

\author{A. Sevostyanov}

\address{ Institute of Pure and Applied Mathematics,
University of Aberdeen \\ Aberdeen AB24 3UE, United Kingdom \\e-mail: a.sevastyanov@abdn.ac.uk}

\begin{abstract}
We define noncommutative deformations $W_q^s(G)$ of algebras of functions on certain (finite coverings of) transversal slices to the set of conjugacy classes in an algebraic group $G$ which play the role of Slodowy slices in algebraic group theory. The algebras $W_q^s(G)$ called q-W algebras are labeled by (conjugacy classes of) elements $s$ of the Weyl group of $G$.  The algebra $W_q^s(G)$ is a quantization of a Poisson structure defined on the corresponding transversal slice in $G$ with the help of Poisson reduction of a Poisson bracket associated to a Poisson--Lie group $G^*$ dual to a quasitriangular Poisson--Lie group. The algebras $W_q^s(G)$ can be regarded as quantum group counterparts of W--algebras. However, in general they are not deformations of the usual W--algebras.

\end{abstract}


\maketitle

\section{Introduction}

\setcounter{equation}{0}

Let $\g$ be a complex semisimple Lie algebra, $G'$ the adjoint group of $\g$ and $e\in \g$ a nonzero nilpotent element in $\g$. By
the Jacobson--Morozov theorem there is an
$\mathfrak{sl}_2$--triple $(e,h,f)$ associated to $e$, i.e.
elements $f,h\in \g$ such that $[h,e]=2e$, $[h,f]=-2f$, $[e,f]=h$.
Fix such an $\mathfrak{sl}_2$--triple.

Let $\chi$ be the element of
$\g^*$ which corresponds to $e$ under the isomorphism $\g\simeq
\g^*$ induced by the Killing form. Under the action of ${\rm ad}~h$
we have a decomposition
\begin{equation}\label{deco}
\g=\oplus_{i\in \mathbb{Z}}\g(i),~{\rm where}~ \g(i)=\{x\in \g
\mid [h,x]=ix\}.
\end{equation}
The skew--symmetric bilinear form $\omega$ on $\g(-1)$ defined by
$\omega(x,y)=\chi([x,y])$ is non--degenerate. Fix an isotropic Lagrangian
subspace $l$ of $\g(-1)$ with respect to $\omega$.

Let
\begin{equation}\label{nopno}
\m=l\oplus \bigoplus_{i\leq -2}\g(i).
\end{equation}
Note that $\m$ is a nilpotent
Lie subalgebra of $\g$ and $\chi\in \g^*$ restricts to a character $\chi: \m \rightarrow \mathbb{C}$. Denote by $\mathbb{C}_\chi$ the corresponding one--dimensional $U(\m)$--module.

The associative algebra $W^e(\g)={\rm End}_{U(\g)}~(U(\g)\otimes_{U(\m)}\mathbb{C}_\chi)^{opp}$ is called the W--algebra associated to the nilpotent element $e$. The algebra $W^e(\g)$ was introduced in \cite{K} in case when $e$ is principal nilpotent and in \cite{Ly} when the grading (\ref{deco}) is even. In paper \cite{DB} the algebras $W^e(\mathfrak{sl}_n)$ are defined using cohomological BRST reduction, and the simple equivalent algebraic definition for arbitrary nilpotent element $e$ given above first appeared in \cite{Pr}. The equivalence of these two definitions follows, for instance, from a general property of homological Hecke--type algebras (see \cite{S2,S4}). An explicit computation establishing this equivalence can also be found in the Appendix in \cite{DS}.

If we denote by $z(f)$ the centralizer of $f$ in $\g$ then the algebra $W^e(\g)$ can be regarded as a noncommutative deformation of the algebra of regular functions on the Slodowy slice $s(e)=e+z(f)$ which is transversal to the set of adjoint orbits in $\g$ (see \cite{DB,GG,Pr}). Note also that the center $Z(U(\g))$ is naturally a subalgebra of the center of $W^e(\g)$, and for any character $\eta:Z(U(\g))\rightarrow \mathbb{C}$ the algebra $W^e(\g)/W^e(\g){\rm ker}~\eta$ can be regarded as a noncommutative deformation of the algebra of regular functions defined on a fiber of the adjoint quotient map $\delta_\g: s(e) \rightarrow \h/W$, where $\delta_\g: \g \rightarrow
\h/W$ is induced by the inclusion $\mathbb{C}[\h]^W\simeq
\mathbb{C}[\g]^{G'}\hookrightarrow \mathbb{C}[\g]$, $\h$ is a Cartan subalgebra in
$\g$ and $W$ is the Weyl group of the pair $(\g,\h)$. In particular, for singular fibers one obtains noncommutative deformations of the coordinate rings of the corresponding singularities (see \cite{Pr}).

W--algebras are primarily important
in the theory of Whittaker or, more generally, generalized
Gelfand--Graev representations (see
\cite{Ka,K}). Namely, to each nilpotent element $e$ in $\g$ one can associate the corresponding category of Gelfand--Graev representations which are $\g$--modules on which $x-\chi(x)$ acts locally nilpotently for each $x\in \m$. The category of generalized Gelfand--Graev representations associated to a nilpotent element $e\in \g$ is equivalent to the
category left modules over the W--algebra associated to $e$. This
remarkable result was proved by Kostant in case of regular
nilpotent $e\in \g$ (see \cite{K}) and by Skryabin in the general case
(see Appendix to \cite{Pr}). A more direct proof of Skryabin's
theorem was obtained in \cite{GG}.

In this paper we construct quantum group analogues of the algebras $W^e(\g)$. Namely, we define certain noncommutative deformations of algebras of regular functions defined on transversal slices to the set of conjugacy classes in an algebraic group $G$ associated to the Lie algebra $\g$. Such slices associated to (conjugacy classes of) Weyl group elements $s\in W$ were defined in \cite{S6} where we also introduced some natural Poisson structures on them. The algebras $W_q^s(G),~s\in W$ introduced in this paper are quantizations of those Poisson structures.

Technically, in order to define the algebras $W_q^s(G)$ one should first construct quantum group analogues of nilpotent Lie subalgebras $\m\subset \g$ and of their nontrivial characters $\chi$. Nilpotent subalgebras in $\g$ can be naturally described in terms of root vectors. It is well known that one can define analogues of root vectors in the standard Drinfeld-Jimbo quantum group $U_h(\g)$ in such a way that their ordered products form a Poincar\'{e}--Birkhoff-Witt basis of $U_h(\g)$ (see \cite{LS,L0,T}). The definition of the quantum group analogues of root vectors is given in terms of a certain braid group action on $U_h(\g)$ and depends on the choice of a normal ordering of the system of positive roots. Our first task is to associate to each Weyl group element $s\in W$ a system of simple roots, a normal ordering of the corresponding system of positive roots $\Delta_+$ and an ordered segment $\Delta_{\m_+}\subset \Delta_+$. The definition of the normal ordering relies on some distinguished normal orderings compatible with involutions in Weyl groups. Such normal orderings are defined in the Appendix A.  Next, for each element $s\in W$ we also define a new realization $U_h^s(\g)$ of the quantum group $U_h(\g)$ in such a way that the subalgebra $U_h^s(\m_+)$ generated by the quantum root vectors associated to roots from $\Delta_{\m_+}$ has a nontrivial character $\chi_h^s$. Note that in general position the subalgebras $U_h^s(\m_+)$ are not deformations of the algebras $U(\m)$. In case when $s$ is a Coxeter element the subalgebras $U_h^s(\m_+)$ were defined in \cite{S1}.

Then we recall that the standard quantum group $U_h(\g)$ contains a certain Hopf subalgebra defined over the ring $\mathbb{C}[q^{\frac{1}{2d}},q^{-\frac{1}{2d}}]$, $q=e^h$, $d\in \mathbb{N}$ such that its specialization at $q=1$ is the Poisson--Hopf algebra of regular functions on an algebraic Poisson--Lie group dual in the sense of Poisson--Lie groups to the quasitriangular Poisson Lie group associated to the standard (Drinfeld-Jimbo) bialgebra structure on $\g$ (see Section \ref{plgroupss}). Similarly, we define a certain Hopf subalgebra $\mathbb{C}_{\mathcal{A}'}[G^*]$ defined over the ring
$\mathcal{A}'=\mathbb{C}[q^{\frac{1}{2d}}, q^{-\frac{1}{2d}}, \frac{1}{[2]_{q_i}},\ldots ,\frac{1}{[r]_{q_i}}, \frac{1-q^{\frac{1}{2d}}}{1-q_i^{-2}}]_{i=1,\ldots, {\rm rank}(\g)}$, where $q_i=q^{d_i}$, $[n]_q={q^n - q^{-n} \over q-q^{-1} }$, $r$ and $d_i$ are some positive integers, such that its specialization at $q=1$ is the Poisson--Hopf algebra of regular functions on an algebraic Poisson--Lie group $G^*$ dual in the sense of Poisson Lie groups to the quasitriangular Poisson--Lie group associated to the nonstandard bialgebra structure on $\g$ with the r--matrix
$$
r^{s}_+=\sum_{\beta \in \Delta_+}(X_{\beta},X_{-\beta})^{-1} X_{\beta}\otimes X_{-\beta} + \frac 12 \left(({1+s \over 1-s }P_{{\h'}}+1)\otimes id\right) t_0,
$$
where $X_{\pm\beta}$ are root vectors of $\g$ corresponding to roots $\pm \beta$, $\beta \in \Delta_+$, $(\cdot,\cdot)$ is the normalized Killing form of $\g$, $t_0\in \h\otimes \h$ is the Cartan part of the Casimir element of $\g$ and $P_{{\h'}}$ is the orthogonal, with respect to the Killing form, projection operator onto the orthogonal complement $\h'$ to the set of fixpoints of $s$ in $\h$.

The algebra $\mathbb{C}_{\mathcal{A}'}[G^*]$ defined in Section \ref{qplproups} contains a subalgebra $\mathbb{C}_{\mathcal{A}'}[M_-]=\mathbb{C}_{\mathcal{A}'}[G^*]\cap U_h^s(\m_+)$ which can be equipped with a nontrivial character $\chi_q^s$. If we denote by $\mathbb{C}_{\chi_q^s}$ the corresponding rank one representation of $\mathbb{C}_{\mathcal{A}'}[M_-]$ then the associative algebra
$$
W_q^s(G)={\rm End}_{{\mathbb{C}}_{\mathcal{A}'}[G^*]}({\mathbb{C}}_{\mathcal{A}'}[G^*]\otimes_{{\mathbb{C}}_{\mathcal{A}'}[M_-]}\mathbb{C}_{\chi_q^{s}})^{opp}
$$
is a deformation of an algebra of functions defined on a finite covering $T_s'$ of a transversal slice $T_s$ to the set of conjugacy classes in an algebraic group $G$ associated to the Lie algebra $\g$.

The slice $T_s$ can be defined as follows.
Let $Z$ be the subgroup of $G$ corresponding to the Lie subalgebra $\z$ spanned by the root subspaces $\g_\alpha$ for roots $\alpha\in \Delta$ fixed by the action of $s$ and by the centralizer of $s$ in $\h$. Let $\Delta_+$ be a system of positive roots associated to the element $s\in W$, and denote by $N_+$ and $N_-$ the corresponding maximal unipotent subgroup of $G$ and the opposite maximal unipotent subgroup, respectively. Then $T_s=sZN_s$, where $N_s=\{ v \in N_+|svs^{-1}\in N_- \}$.

$T_s'$ can be equipped with a Poisson structure using Poisson reduction in the Poisson--Lie group $G^*$ (see \cite{S6}). The algebra $W_q^s(G)$ is a quantization of the Poisson structure defined on $T_s'$. The proof of these facts occupies Sections \ref{wpsred} and \ref{crossect}. Note that in general the algebras $W_q^s(G)$ are not deformations of the usual W--algebras $W^e(\g)$. In case when $s$ is a Coxeter element of $W$ the algebras $W_q^s(G)$ were introduced in \cite{S8,S7}.

For $\varepsilon\in \mathbb{C}$ one can also consider specializations $W_\varepsilon^s(G)={\rm End}_{{\mathbb{C}}_\varepsilon[G^*]}({\mathbb{C}}_\varepsilon[G^*]\otimes_{{\mathbb{C}}_\varepsilon[M_-]}\mathbb{C}_{\chi_\varepsilon^{s}})^{opp}$ of algebras $W_q^s(G)$ defined with the help of the specializations ${\mathbb{C}}_\varepsilon[G^*]={\mathbb{C}}_{\mathcal{A}'}[G^*]/
(q^{\frac{1}{2d}}-\varepsilon^{\frac{1}{2d}}){\mathbb{C}}_{\mathcal{A}'}[G^*]$, ${\mathbb{C}}_\varepsilon[M_-]={\mathbb{C}}_{\mathcal{A}'}[M_-]/
(q^{\frac{1}{2d}}-\varepsilon^{\frac{1}{2d}}){\mathbb{C}}_{\mathcal{A}'}[M_-]$ of the algebras ${\mathbb{C}}_{\mathcal{A}'}[G^*]$, ${\mathbb{C}}_{\mathcal{A}'}[M_-]$ and the natural specialization $\chi_\varepsilon^{s}$ of the character $\chi_q^{s}$.

 Next we discuss a homological realization of algebras $W_\varepsilon^s(G)$ similar to that suggested in \cite{DB} for the usual W--algebras. The latter one is based on the BRST cohomological  reduction procedure for Lie algebras. In case of arbitrary associative algebras an analogue of the BRST reduction technique was suggested in \cite{S2,S4}. We apply this technique in the situation considered in this paper. This yields a graded associative algebra ${\rm Hk}^\gr(({\mathbb{C}}_\varepsilon[G^*], {\mathbb{C}}_\varepsilon[M_-], \chi_\varepsilon^{s}))$ with trivial negatively graded components and such that ${\rm Hk}^0(({\mathbb{C}}_\varepsilon[G^*], {\mathbb{C}}_\varepsilon[M_-], \chi_\varepsilon^{s}))=W_\varepsilon^s(G)^{opp}$.
For every left ${\mathbb{C}}_\varepsilon[G^*]$--module $V$ and right ${\mathbb{C}}_\varepsilon[G^*]$--module $W$ the algebra ${\rm Hk}^\gr(({\mathbb{C}}_\varepsilon[G^*], {\mathbb{C}}_\varepsilon[M_-], \chi_\varepsilon^{s}))$ naturally acts in the spaces ${\rm Ext}^{\gr}_{{\mathbb{C}}_\varepsilon[M_-]}(\mathbb{C}_{\chi_\varepsilon^{s}},V)$ and ${\rm Tor}^{\gr}_{{\mathbb{C}}_\varepsilon[M_-]}(W,\mathbb{C}_{\chi_\varepsilon^{s}})$, from the right and from the left, respectively. Note that, as the example given in Section \ref{hreal} shows, in contrast to the Lie algebra case the positively graded components of the algebra ${\rm Hk}^\gr(({\mathbb{C}}_\varepsilon[G^*], {\mathbb{C}}_\varepsilon[M_-], \chi_\varepsilon^{s}))$ do not vanish even for generic $\varepsilon$.

Another interesting problem related to algebras $W_\varepsilon^s(G)$ is concerned with representation theory of quantum groups at roots of unity. Recall that this representation theory is similar to the representation theory of semisimple Lie algebras over a field of prime characteristic $p$. Let $\g_p$ be such an algebra. In \cite{Pr} it was shown that to each element $\xi\in \g_p^*$ one can associate a W-algebra in such a way that the corresponding reduced universal enveloping algebra $U(\g_p)_\xi$ is isomorphic to the algebra of matrices of size ${d(\xi)}$ with entries being elements of the W--algebra, where ${d(\xi)}=p^{\frac{1}{2}{\rm dim}~\mathcal{O}_\xi}$, and $\mathcal{O}_\xi$ is the coadjoint orbit of $\xi$.
Since each finite--dimensional representation $V$ of $\g_p$ is a representation of an algebra $U(\g_p)_\xi$ for some $\xi$ the last statement implies the Kac--Weisfeiler conjecture which states the dimension of $V$ is divisible by ${d(\xi)}$.

In \cite{DKP1} De Concini, Kac and Procesi formulated a similar conjecture for quantum groups at roots of unity. However, in case of quantum groups at roots of unity irreducible finite--dimensional representations are parameterized by conjugacy classes in algebraic groups. More precisely, in \cite{DKP1} it is shown that the center of the quantum group $U_\varepsilon(\g)$ at a primitive odd $m$--th root of unity $\varepsilon$ contains a large commutative subalgebra $Z_0$, and there is a covering $\pi: {\rm Spec}(Z_0)\rightarrow G^0$ of degree $2^{{\rm rank}~\g}$, where $G^0$ is the big cell in $G$. There is also an infinite--dimensional group $\mathcal{G}$ which acts by algebra automorphisms on a certain completion of $U_\varepsilon(\g)$. This action is called the quantum coadjoint action. The quantum coadjoint action preserves a completion of the subalgebra $Z_0$ and induces algebra isomorphisms $\widetilde{g}:U_\eta \rightarrow U_{\widetilde{g}\eta}$, $\widetilde{g}\in \mathcal{G}$, where for any homomorphism $\eta\in {\rm Spec}(Z_0)$ one defines a reduced quantum group $U_\eta$ by
$
U_\eta={U}_\varepsilon(\g)/I_\eta,
$
and $I_\eta$ is the ideal in ${U}_\varepsilon(\g)$ generated by elements $z-\eta(z)$, $z\in Z_0$. In particular, the sets of equivalence classes of irreducible finite--dimensional representations of algebras $U_\eta$ and $U_{\widetilde{g}\eta}$ are isomorphic. In \cite{DKP1} it is also shown that if $\mathcal{O}$ is a conjugacy class in $G$ then the intersection $\mathcal{O}^0=\mathcal{O}\bigcap G^0$ is a smooth connected variety, and the connected components of the variety $\pi^{-1}(\mathcal{O}^0)$ are $\mathcal{G}$--orbits in ${\rm Spec}(Z_0)$. Thus irreducible finite--dimensional representations of $U_\varepsilon(\g)$ can be parameterized by conjugacy classes in $G$, and every such a representation $V$ with central character $\psi$ such that $\psi\mid_{Z_0}=\eta\in {\rm Spec}(Z_0)$ corresponds to the conjugacy class $\mathcal{O}$ of an element $g\in G$ in the sense that $\eta \in \pi^{-1}(\mathcal{O}^0)$, where $\mathcal{O}^0=\mathcal{O}\cap G^0$.

Let $G$ be the simply connected algebraic with Lie algebra $\g$, and $T_s$ a transversal slice to the set of conjugacy classes in $G$ corresponding to element $s\in W$.
One can show that for simply connected $G$ and the for corresponding version of $U_\varepsilon(\g)$, under some technical assumptions on $m$, the algebra $U_\eta$ corresponding to $\eta\in \pi^{-1}(g)$, $g\in T_s$ is isomorphic to the algebra of matrices of size $m^{{\rm dim}~\m_+}$ with entries being elements of a reduced root of unity version of the W--algebra $W_q^s(G)$ (see \cite{S10}). Thus the dimension of every irreducible finite--dimensional representation of $U_\varepsilon(\g)$ corresponding to the conjugacy class of $g\in T_s$ is divisible by $m^{{\rm dim}~\m_+}$. Dimensional count in the proof of Theorem \ref{var} shows that $2~{\rm dim}~\m_++{\rm dim}~T_s={\rm dim}~G$, and hence if $g\in T_s$ is an element such that ${\rm dim}~T_s={\rm dim}~Z_G(g)$, where $Z_G(g)$ is the centralizer of $g$ in $G$, then the dimension of every irreducible finite--dimensional representation of $U_\varepsilon(\g)$ corresponding to the conjugacy class of $g$ is divisible by $m^{\frac{1}{2}{\rm dim}~\mathcal{O}}$, where $\mathcal{O}$ is the conjugacy class of $g$. This is exactly the De Concini--Kac--Procesi conjecture about dimensions of irreducible finite--dimensional representations of $U_\varepsilon(\g)$ which asserts that the above property should hold for any conjugacy class $\mathcal{O}$.

In \cite{DK} it is also shown that every irreducible finite--dimensional representation $V$ of $U_\varepsilon(\g)$ is induced from an irreducible finite--dimensional representation $V'$ of a subalgebra $U_\varepsilon(\g')\subset U_\varepsilon(\g)$, where $\g'\subset \g$ is the semisimple part of a Levi subalgebra in $\g$, and $V'$ corresponds to the conjugacy class of an exceptional element $g'\in G'$, where $G'\subset G$ is the algebraic subgroup corresponding to the Lie subalgebra $\g'\subset \g$ (we recall that an element $g'\in G'$ is called exceptional if the centralizer of the semisimple part of $g'$ in $G'$ has a finite center). Thus it suffices to study the structure of irreducible finite--dimensional representations of $U_\varepsilon(\g)$ which correspond to conjugacy classes of exceptional elements. In fact by the above observation it suffices to prove the De Concini--Kac--Procesi conjecture in case of exceptional elements. Note that the number of exceptional conjugacy classes in $G$ is finite.

Thus the De Concini--Kac--Procesi conjecture follows from the following statement about the structure of algebraic groups.
\begin{conjecture}\label{conj1}
Every exceptional element $g\in G$ is conjugate to an element in a slice $T_s$ such that ${\rm dim}~T_s={\rm dim}~Z_G(g)$.
\end{conjecture}

This conjecture was checked in case of simple Lie algebras of types $A_2$ and $B_2$ for all exceptional conjugacy classes and in case of arbitrary simple Lie algebras for certain unipotent classes which intersect slices $T_s$ associated to semi--Coxeter elements $s\in W$ corresponding to connected Carter graphs with maximal possible number of vertices in the Carter classification of conjugacy classes in Weyl groups (see \cite{C}). In all those cases the corresponding unipotent classes contain representatives defined by formula (\ref{defu}).

In paper \cite{L1} Lusztig constructed a map $\phi$ from the set of conjugacy classes in $W$ to the set of conjugacy classes of unipotent elements in $G$. If $s$ is an element of minimal length (with respect to a system of simple reflections) in the conjugacy class $\mathcal{O}_s$ of $s$ in $W$ then there exists a unique unipotent conjugacy class $\mathcal{O}_u$ in $G$ which intersects the Bruhat cell $B_+sB_+$ (here $B_+$ is the Borel subgroup of $G$ associated to the corresponding system of simple roots) and which has minimal dimension among of all unipotent classes intersecting $B_+sB_+$. Lusztig proves that $\phi(\mathcal{O}_s)=\mathcal{O}_u$ is a surjective map and in \cite{L4} he also shows that there is a right inverse $\psi$ to $\phi$ defined by $\psi(\mathcal{O}_u)=\mathcal{O}_w$, where $\phi(\mathcal{O}_w)=\mathcal{O}_u$, and the dimension of the subspace in $\h$ fixed by the action of $w$ is minimal among of all $s$ such that $\phi(\mathcal{O}_s)=\mathcal{O}_u$.

We expect that in case of unipotent classes Weyl group elements that appear in Conjecture \ref{conj1} should be defined with the help of the map $\psi$, and appropriate representatives in those Weyl group conjugacy classes correspond to suitable choices of orderings of Weyl group invariant subspaces in (\ref{hdec}) (see also the construction after formula (\ref{hdec})).

In \cite{L6,L1,L5} some varieties $V_s$ similar to the slices $T_s$ are defined in case of elliptic Weyl group elements, i.e. elements which do not belong to any parabolic subgroups of $W$. If $s$ is an elliptic element of minimal length (with respect to a system of simple reflections) in a conjugacy class $\mathcal{O}_s$ of $s$ in $W$ then $V_s=sN_s$, $N_s=\{ v \in N_+|svs^{-1}\in N_- \}$, where $N_+$ is the maximal unipotent subgroup of $G$ associated to the corresponding system of simple roots, and $N_-$ is the opposite maximal unipotent subgroup.
In \cite{L6,L5} an isomorphism similar to (\ref{cross}) is established for the varieties $V_s$. Note that the unipotent class $\phi(\mathcal{O}_s)=\mathcal{O}_u$ in $G$ intersects $V_s$, and in \cite{L6,L1,L5} it is also proved that ${\rm dim}~V_s={\rm dim}~Z_G(u)$, $u\in \mathcal{O}_u$.

We expect that the varieties $V_s$ should coincide with some of the $T_s$ in case of elliptic Weyl group conjugacy classes, and if $\psi(\mathcal{O}_u)=\mathcal{O}_s$ with elliptic $s$ one could fulfill Conjecture \ref{conj1} for $u\in \mathcal{O}_u$ by taking $T_s=V_s$.

In \cite{L2} for elliptic Weyl group conjugacy classes $\mathcal{O}_s$, some semisimple classes $\mathcal{O}_\zeta$, which intersect the varieties $V_s$ and enjoy the properties ${\rm dim}~V_s={\rm dim}~Z_G(\zeta)$, $\zeta\in \mathcal{O}_\zeta$, are constructed. We expect that the varieties $V_s$ could be used to prove Conjecture \ref{conj1} in case of those semisimple classes as well.

In conclusion we note that the results of \cite{L6,L1,L2,L4,L5} partially rely on computer calculations and on case--by--case analysis. Therefore the proof of Conjecture \ref{conj1} may require computer calculations and case--by--case analysis as well.

{\bf Acknowledgement}

The author is grateful to A. Berenstein and P. Etingof for useful
discussions.


\setcounter{equation}{0}
\setcounter{theorem}{0}

\section{Notation}\label{notation}

Fix the notation used throughout of the text.
Let $G$ be a
connected finite--dimensional complex simple Lie group, $
{\frak g}$ its Lie algebra. Fix a Cartan subalgebra ${\frak h}\subset {\frak
g}\ $and let $\Delta $ be the set of roots of $\left( {\frak g},{\frak h}
\right)$.  Let $\alpha_i,~i=1,\ldots, l,~~l=rank({\frak g})$ be a system of
simple roots, $\Delta_+=\{ \beta_1, \ldots ,\beta_N \}$
the set of positive roots.
Let $H_1,\ldots ,H_l$ be the set of simple root generators of $\frak h$.

Let $a_{ij}$ be the corresponding Cartan matrix,
and let $d_1,\ldots , d_l$ be coprime positive integers such that the matrix
$b_{ij}=d_ia_{ij}$ is symmetric. There exists a unique non--degenerate invariant
symmetric bilinear form $\left( ,\right) $ on ${\frak g}$ such that
$(H_i , H_j)=d_j^{-1}a_{ij}$. It induces an isomorphism of vector spaces
${\frak h}\simeq {\frak h}^*$ under which $\alpha_i \in {\frak h}^*$ corresponds
to $d_iH_i \in {\frak h}$. We denote by $\alpha^\vee$ the element of $\frak h$ that
corresponds to $\alpha \in {\frak h}^*$ under this isomorphism.
The induced bilinear form on ${\frak h}^*$ is given by
$(\alpha_i , \alpha_j)=b_{ij}$.

Let $W$ be the Weyl group of the root system $\Delta$. $W$ is the subgroup of $GL({\frak h})$
generated by the fundamental reflections $s_1,\ldots ,s_l$,
$$
s_i(h)=h-\alpha_i(h)H_i,~~h\in{\frak h}.
$$
The action of $W$ preserves the bilinear form $(,)$ on $\frak h$.
We denote a representative of $w\in W$ in $G$ by
the same letter. For $w\in W, g\in G$ we write $w(g)=wgw^{-1}$.
For any root $\alpha\in \Delta$ we also denote by $s_\alpha$ the corresponding reflection.

Let ${{\frak b}_+}$ be the positive Borel subalgebra and ${\frak b}_-$
the opposite Borel subalgebra; let ${\frak n}_+=[{{\frak b}_+},{{\frak b}_+}]$ and $%
{\frak n}_-=[{\frak b}_-,{\frak b}_-]$ be their
nilradicals. Let $H=\exp {\frak h},N_+=\exp {{\frak n}_+},
N_-=\exp {\frak n}_-,B_+=HN_+,B_-=HN_-$ be
the Cartan subgroup, the maximal unipotent subgroups and the Borel subgroups
of $G$ which correspond to the Lie subalgebras ${\frak h},{{\frak n}_+},%
{\frak n}_-,{\frak b}_+$ and ${\frak b}_-,$ respectively.

We identify $\frak g$ and its dual by means of the canonical invariant bilinear form.
Then the coadjoint
action of $G$ on ${\frak g}^*$ is naturally identified with the adjoint one. We also identify
${{\frak n}_+}^*\cong {\frak n}_-,~{{\frak b}_+}^*\cong {\frak b}_-$.

Let ${\frak g}_\beta$ be the root subspace corresponding to a root $\beta \in \Delta$,
${\frak g}_\beta=\{ x\in {\frak g}| [h,x]=\beta(h)x \mbox{ for every }h\in {\frak h}\}$.
${\frak g}_\beta\subset {\frak g}$ is a one--dimensional subspace.
It is well known that for $\alpha\neq -\beta$ the root subspaces ${\frak g}_\alpha$ and ${\frak g}_\beta$ are orthogonal with respect
to the canonical invariant bilinear form. Moreover ${\frak g}_\alpha$ and ${\frak g}_{-\alpha}$
are non--degenerately paired by this form.

Root vectors $X_{\alpha}\in {\frak g}_\alpha$ satisfy the following relations:
$$
[X_\alpha,X_{-\alpha}]=(X_\alpha,X_{-\alpha})\alpha^\vee.
$$

Note also that in this paper we denote by $\mathbb{N}$ the set of nonnegative integer numbers, $\mathbb{N}=\{0,1,\ldots \}$.


\section{Quantum groups}

\setcounter{equation}{0}
\setcounter{theorem}{0}

In this section we recall some basic facts about quantum groups.
We follow the notation of \cite{ChP}.

Let $h$ be an indeterminate, ${\Bbb C}[[h]]$ the ring of formal power series in $h$.
We shall consider ${\Bbb C}[[h]]$--modules equipped with the so--called $h$--adic
topology. For every such module $V$ this topology is characterized by requiring that
$\{ h^nV ~|~n\geq 0\}$ is a base of the neighborhoods of $0$ in $V$, and that translations
in $V$ are continuous. It is easy to see that, for modules equipped with this topology, every
${\Bbb C}[[h]]$--module map is automatically continuous.

A topological Hopf algebra over ${\Bbb C}[[h]]$ is a complete ${\Bbb C}[[h]]$--module $A$
equipped with a structure of ${\Bbb C}[[h]]$--Hopf algebra (see \cite{ChP}, Definition 4.3.1),
the algebraic tensor products entering the axioms of the Hopf algebra are replaced by their
completions in the $h$--adic topology.
We denote by $\mu , \imath , \Delta , \varepsilon , S$ the multiplication, the unit, the comultiplication,
the counit and the antipode of $A$, respectively.

The standard quantum group $U_h({\frak g})$ associated to a complex finite--dimensional simple Lie algebra
$\frak g$ is the algebra over ${\Bbb C}[[h]]$ topologically generated by elements
$H_i,~X_i^+,~X_i^-,~i=1,\ldots ,l$, and with the following defining relations:
\begin{equation}
\begin{array}{l}
[H_i,H_j]=0,~~ [H_i,X_j^\pm]=\pm a_{ij}X_j^\pm, \\
\\
X_i^+X_j^- -X_j^-X_i^+ = \delta _{i,j}{K_i -K_i^{-1} \over q_i -q_i^{-1}} , \label{qgrh} \\
\\
\mbox{where }K_i=e^{d_ihH_i},~~e^h=q,~~q_i=q^{d_i}=e^{d_ih},
\end{array}
\end{equation}
and the quantum Serre relations:
\begin{equation}\label{qserre}
\begin{array}{l}
\sum_{r=0}^{1-a_{ij}}(-1)^r
\left[ \begin{array}{c} 1-a_{ij} \\ r \end{array} \right]_{q_i}
(X_i^\pm )^{1-a_{ij}-r}X_j^\pm(X_i^\pm)^r =0 ,~ i \neq j ,\\ \\
\mbox{ where }\\
 \\
\left[ \begin{array}{c} m \\ n \end{array} \right]_q={[m]_q! \over [n]_q![n-m]_q!} ,~
[n]_q!=[n]_q\ldots [1]_q ,~ [n]_q={q^n - q^{-n} \over q-q^{-1} }.
\end{array}
\end{equation}
$U_h({\frak g})$ is a topological Hopf algebra over ${\Bbb C}[[h]]$ with comultiplication
defined by
$$
\begin{array}{l}
\Delta_h(H_i)=H_i\otimes 1+1\otimes H_i,\\
\\
\Delta_h(X_i^+)=X_i^+\otimes K_i+1\otimes X_i^+,
\end{array}
$$
$$
\Delta_h(X_i^-)=X_i^-\otimes 1 +K_i^{-1}\otimes X_i^-,
$$
antipode defined by
$$
S_h(H_i)=-H_i,~~S_h(X_i^+)=-X_i^+K_i^{-1},~~S_h(X_i^-)=-K_iX_i^-,
$$
and counit defined by
$$
\varepsilon_h(H_i)=\varepsilon_h(X_i^\pm)=0.
$$

We shall also use the weight--type generators
$$
Y_i=\sum_{j=1}^l d_i(a^{-1})_{ij}H_j,
$$
and the elements $L_i=e^{hY_i}$. They commute with the root vectors $X_i^\pm$ as follows:
\begin{equation}\label{weight-root}
L_iX_j^\pm L_i^{-1}=q_i^{\pm \delta_{ij}}X_j^\pm .
\end{equation}
We also obviously have
\begin{equation}\label{comml}
L_iL_j=L_jL_i.
\end{equation}

The Hopf algebra $U_h({\frak g})$ is a quantization of the standard bialgebra structure on $\frak g$, i.e.
$U_h({\frak g})/hU_h({\frak g})=U({\frak g}),~~ \Delta_h=\Delta~(\mbox{mod }h)$, where $\Delta$ is
the standard comultiplication on $U({\frak g})$, and
$$
{\Delta_h -\Delta_h^{opp} \over h}~(\mbox{mod }h)=\delta ,
$$
where
$\delta: {\frak g}\rightarrow {\frak g}\otimes {\frak g}$ is the standard cocycle on $\frak g$, and $\Delta^{opp}_h=\sigma \Delta_h$, $\sigma$ is the permutation in $U_h({\frak g})^{\otimes 2}$,
$\sigma (x\otimes y)=y\otimes x$.
Recall that
$$
\delta (x)=({\rm ad}_x\otimes 1+1\otimes {\rm ad}_x)2r_+,~~ r_+\in {\frak g}\otimes {\frak g},
$$
\begin{equation}\label{rcl}
r_+=\frac 12 \sum_{i=1}^lY_i \otimes H_i + \sum_{\beta \in \Delta_+}(X_{\beta},X_{-\beta})^{-1} X_{\beta}\otimes X_{-\beta}.
\end{equation}
Here $X_{\pm \beta}\in {\frak g}_{\pm \beta}$ are root vectors of $\frak g$.
The element $r_+\in {\frak g}\otimes {\frak g}$ is called a classical r--matrix.

$U_h({\frak g})$ is a quasitriangular Hopf algebra, i.e. there exists an invertible element
${\mathcal R}\in U_h({\frak g})\otimes U_h({\frak g})$, called a universal R--matrix, such that
\begin{equation}\label{quasitr}
\Delta^{opp}_h(a)={\mathcal R}\Delta_h(a){\mathcal R}^{-1}\mbox{ for all } a\in U_h({\frak g}),
\end{equation}
and
\begin{equation}\label{rmprop}
\begin{array}{l}
(\Delta_h \otimes id){\mathcal R}={\mathcal R}_{13}{\mathcal R}_{23},\\
\\
(id \otimes \Delta_h){\mathcal R}={\mathcal R}_{13}{\mathcal R}_{12},
\end{array}
\end{equation}
where ${\mathcal R}_{12}={\mathcal R}\otimes 1,~{\mathcal R}_{23}=1\otimes {\mathcal R},
~{\mathcal R}_{13}=(\sigma \otimes id){\mathcal R}_{23}$.

From (\ref{quasitr}) and (\ref{rmprop}) it follows that $\mathcal R$ satisfies the quantum Yang--Baxter
equation:
\begin{equation}\label{YB}
{\mathcal R}_{12}{\mathcal R}_{13}{\mathcal R}_{23}={\mathcal R}_{23}{\mathcal R}_{13}{\mathcal R}_{12}.
\end{equation}

For every quasitriangular Hopf algebra we also have (see Proposition 4.2.7 in \cite{ChP}):
$$
(S\otimes id){\mathcal R}=(id \otimes S^{-1}){\mathcal R}={\mathcal R}^{-1},
$$
and
\begin{equation}\label{S}
(S\otimes S){\mathcal R}={\mathcal R}.
\end{equation}

We shall explicitly describe the element ${\mathcal R}$.
First following \cite{ChP} we recall the construction of root vectors of $U_h({\frak g})$  in terms of  a braid group action on $U_h({\frak g})$. Let $m_{ij}$, $i\neq j$ be equal to $2,3,4,6$ if $a_{ij}a_{ji}$ is equal to $0,1,2,3$. The braid group $\mathcal{B}_\g$ associated to $\g$ has generators $T_i$, $i=1,\ldots, l$, and defining relations
$$
T_iT_jT_iT_j\ldots=T_jT_iT_jT_i\ldots
$$
for all $i\neq j$, where there are $m_{ij}$ $T$'s on each side of the equation.

There is an action of the braid group $\mathcal{B}_\g$ by algebra automorphisms of $U_h({\frak g})$ defined on the standard generators as follows:
\begin{eqnarray*}
T_i(X_i^+)=-X_i^-e^{hd_iH_i},~T_i(X_i^-)=-e^{-hd_iH_i}X_i^+,~T_i(H_j)=H_j-a_{ji}H_i, \\
\\
T_i(X_j^+)=\sum_{r=0}^{-a_{ij}}(-1)^{r-a_{ij}}q_i^{-r}
(X_i^+ )^{(-a_{ij}-r)}X_j^+(X_i^+)^{(r)},~i\neq j,\\
\\
T_i(X_j^-)=\sum_{r=0}^{-a_{ij}}(-1)^{r-a_{ij}}q_i^{r}
(X_i^-)^{(r)}X_j^-(X_i^-)^{(-a_{ij}-r)},~i\neq j,
\end{eqnarray*}
where
$$
(X_i^+)^{(r)}=\frac{(X_i^+)^{r}}{[r]_{q_i}!},~(X_i^-)^{(r)}=\frac{(X_i^-)^{r}}{[r]_{q_i}!},~r\geq 0,~i=1,\ldots,l.
$$

Recall that an ordering of a set of positive roots $\Delta_+$ is called normal if all simple roots are written in an arbitrary
order, and
for any three roots $\alpha,~\beta,~\gamma$ such that
$\gamma=\alpha+\beta$ we have either $\alpha<\gamma<\beta$ or $\beta<\gamma<\alpha$.

Any two normal orderings in $\Delta_+$ can be reduced to each other by the so--called elementary transpositions (see \cite{Z}, Theorem 1). The elementary transpositions for rank 2 root systems are inversions of the following normal orderings (or the inverse normal orderings):
\begin{equation}\label{rank2}
\begin{array}{lr}
\alpha,~\beta & A_1+A_1 \\
\\
\alpha,~\alpha+\beta,~\beta & A_2 \\
\\
\alpha,~\alpha+\beta,~\alpha+2\beta,~\beta & B_2 \\
\\
\alpha,~\alpha+\beta,~2\alpha+3\beta,~\alpha+2\beta,~\alpha+3\beta,~\beta & G_2
\end{array}
\end{equation}
where it is assumed that $(\alpha,\alpha)\geq (\beta,\beta)$. Moreover, any normal ordering in a rank 2 root system is one of orderings (\ref{rank2}) or one of the inverse orderings.

In general an elementary inversion of a normal ordering in a set of positive roots $\Delta_+$  is the inversion of an ordered segment of form (\ref{rank2}) (or of a segment with the inverse ordering) in the ordered set $\Delta_+$, where $\alpha-\beta\not\in \Delta$.

For any reduced decomposition $w_0=s_{i_1}\ldots s_{i_D}$ of the longest element $w_0$ of the Weyl group $W$ of $\g$ the set
$$
\beta_1=\alpha_{i_1},\beta_2=s_{i_1}\alpha_{i_2},\ldots,\beta_D=s_{i_1}\ldots s_{i_{D-1}}\alpha_{i_D}
$$
is a normal ordering in $\Delta_+$, and there is one--to--one correspondence between normal orderings of $\Delta_+$ and reduced decompositions of $w_0$ (see \cite{Z1}).

Now fix a reduced decomposition $w_0=s_{i_1}\ldots s_{i_D}$ of the longest element $w_0$ of the Weyl group $W$ of $\g$ and define the corresponding root vectors in $U_h({\frak g})$ by
\begin{equation}\label{rootvect}
X_{\beta_k}^\pm=T_{i_1}\ldots T_{i_{k-1}}X_{i_k}^\pm.
\end{equation}

\begin{proposition}\label{rootprop}
For $\beta =\sum_{i=1}^lm_i\alpha_i,~m_i\in {\Bbb N}$  $X_{\beta}^\pm $ is a polynomial in
the noncommutative variables $X_i^\pm$ homogeneous in each $X_i^\pm$ of degree $m_i$.
\end{proposition}

The root vectors $X_{\beta}^+$ satisfy the following relations:
\begin{equation}\label{qcom}
X_{\alpha}^+X_{\beta}^+ - q^{(\alpha,\beta)}X_{\beta}^+X_{\alpha}^+= \sum_{\alpha<\delta_1<\ldots<\delta_n<\beta}C(k_1,\ldots,k_n)
{(X_{\delta_1}^+)}^{(k_1)}{(X_{\delta_2}^+)}^{(k_2)}\ldots {(X_{\delta_n}^+)}^{(k_n)},
\end{equation}
where for $\alpha \in \Delta_+$ we put ${(X_{\alpha}^\pm)}^{(k)}=\frac{(X_\alpha^\pm)^{k}}{[k]_{q_\alpha}!}$, $k\geq 0$, $q_\alpha =q^{d_i}$ if the positive root $\alpha$ is Weyl group conjugate to the simple root $\alpha_i$, $C(k_1,\ldots,k_n)\in {\Bbb C}[q,q^{-1}]$.
They also commute with elements of the subalgebra $U_h({\frak h})$ as follows:

\begin{equation}\label{roots-cart}
[H_i,X_{\beta}^\pm]=\pm \beta(H_i)X_{\beta}^\pm,~i=1,\ldots ,l.
\end{equation}

Note that by construction
$$
\begin{array}{l}
X_\beta^+~(\mbox{mod }h)=X_\beta \in {\frak g}_\beta,\\
\\
X_\beta^-~(\mbox{mod }h)=X_{-\beta} \in {\frak g}_{-\beta}
\end{array}
$$
are root vectors of $\frak g$.

Let $U_h({\frak n}_+),U_h({\frak n}_-)$ and $U_h(\h)$ be the ${\Bbb C}[[h]]$--subalgebras of $U_h({\frak g})$ topologically
generated by the
$X_i^+$, by the $X_i^-$ and by the $H_i$, respectively.

Now using the root vectors $X_{\beta}^\pm$ we can construct a topological basis of
$U_h({\frak g})$.
Define for ${\bf r}=(r_1,\ldots ,r_D)\in {\Bbb N}^D$,
$$
(X^+)^{(\bf r)}=(X_{\beta_1}^+)^{(r_1)}\ldots (X_{\beta_D}^+)^{(r_D)},
$$
$$
(X^-)^{(\bf r)}=(X_{\beta_D}^-)^{(r_D)}\ldots (X_{\beta_1}^-)^{(r_1)},
$$
and for ${\bf s}=(s_1,\ldots s_l)\in {\Bbb N}^{l}$,
$$
H^{\bf s}=H_1^{s_1}\ldots H_l^{s_l}.
$$
\begin{proposition}{\bf (\cite{kh-t}, Proposition 3.3)}\label{PBW}
The elements $(X^+)^{(\bf r)}$, $(X^-)^{(\bf t)}$ and $H^{\bf s}$, for ${\bf r},~{\bf t}\in {\Bbb N}^N$,
${\bf s}\in {\Bbb N}^l$, form topological bases of $U_h({\frak n}_+),U_h({\frak n}_-)$ and $U_h({\frak h})$,
respectively, and the products $(X^+)^{(\bf r)}H^{\bf s}(X^-)^{(\bf t)}$ form a topological basis of
$U_h({\frak g})$. In particular, multiplication defines an isomorphism of ${\Bbb C}[[h]]$ modules:
$$
U_h({\frak n}_-)\otimes U_h({\frak h}) \otimes U_h({\frak n}_+)\rightarrow U_h({\frak g}).
$$
\end{proposition}

An explicit expression for $\mathcal R$ may be written by making use of the q--exponential
$$
exp_q(x)=\sum_{k=0}^\infty q^{\frac{1}{2}k(k+1)}{x^k \over [k]_q!}
$$
in terms of which the element $\mathcal R$ takes the form:
\begin{equation}\label{univr}
{\mathcal R}=exp\left[ h\sum_{i=1}^l(Y_i\otimes H_i)\right]\prod_{\beta}
exp_{q_{\beta}}[(1-q_{\beta}^{-2})X_{\beta}^+\otimes X_{\beta}^-],
\end{equation}
where
the product is over all the positive roots of $\frak g$, and the order of the terms is such that
the $\alpha$--term appears to the left of the $\beta$--term if $\alpha <\beta$ with respect to the normal
ordering of $\Delta_+$.

\begin{remark}
The r--matrix $r_+=\frac 12 h^{-1}({\mathcal R}-1\otimes 1)~~(\mbox{mod }h)$, which is the classical limit of $\mathcal R$,
coincides with the classical r--matrix (\ref{rcl}).
\end{remark}


\section{Realizations of quantum groups associated to Weyl group elements}\label{wqreal}

\setcounter{equation}{0}
\setcounter{theorem}{0}

The subalgebras of $U_h({\frak g})$ which possess nontrivial
characters are defined in terms of the new realizations $U_h^{s}(\g)$ of $U_h(\g)$ associated to Weyl group elements, and we start by defining these new realizations.

Let $s$ be an element of the Weyl group $W$ of the pair $(\g,\h)$, and $\h'$ the orthogonal complement, with respect to the Killing form, to the subspace of $\h$ fixed by the natural action of $s$ on $\h$. Let $\h'^*$ be the image of $\h'$ in $\h^*$ under the identification $\h^*\simeq \h$ induced by the canonical bilinear form on $\g$.
The restriction of the natural action of $s$ on $\h^*$ to the subspace $\h'^*$ has no fixed points. Therefore one can define the Cayley transform ${1+s \over 1-s }P_{{\h'}^*}$ of the restriction of $s$ to ${\h'}^*$, where $P_{{\h'}^*}$ is the orthogonal projection operator onto ${{\h'}^*}$ in $\h^*$, with respect to the Killing form.

Now we suggest a new realization of the quantum group $U_h({\frak g})$ associated to $s\in W$.
Let
$U_h^{s}({\frak g})$ be the associative algebra over ${\Bbb C}[[h]]$ topologically generated by elements
$e_i , f_i , H_i,~i=1, \ldots l$ subject to the relations:
\begin{equation}\label{sqgr}
\begin{array}{l}
[H_i,H_j]=0,~~ [H_i,e_j]=a_{ij}e_j, ~~ [H_i,f_j]=-a_{ij}f_j,\\
\\
e_i f_j -q^{ c_{ij}} f_j e_i = \delta _{i,j}{K_i -K_i^{-1} \over q_i -q_i^{-1}} , c_{ij}=\left( {1+s \over 1-s }P_{{\h'}^*}\alpha_i , \alpha_j \right),\\
 \\
K_i=e^{d_ihH_i}, \\
 \\
\sum_{r=0}^{1-a_{ij}}(-1)^r q^{r c_{ij}}
\left[ \begin{array}{c} 1-a_{ij} \\ r \end{array} \right]_{q_i}
(e_i )^{1-a_{ij}-r}e_j (e_i)^r =0 ,~ i \neq j , \\
 \\
\sum_{r=0}^{1-a_{ij}}(-1)^r q^{r c_{ij}}
\left[ \begin{array}{c} 1-a_{ij} \\ r \end{array} \right]_{q_i}
(f_i )^{1-a_{ij}-r}f_j (f_i)^r =0 ,~ i \neq j .
\end{array}
\end{equation}

\begin{theorem} \label{newreal}
For every solution $n_{ij}\in {\Bbb C},~i,j=1,\ldots ,l$ of equations
\begin{equation}\label{eqpi}
d_jn_{ij}-d_in_{ji}=c_{ij}
\end{equation}
there exists an algebra
isomorphism $\psi_{\{ n\}} : U_h^{s}({\frak g}) \rightarrow
U_h({\frak g})$ defined  by the formulas:
$$
\begin{array}{l}
\psi_{\{ n\}}(e_i)=X_i^+ \prod_{p=1}^lL_p^{n_{ip}},\\
 \\
\psi_{\{ n\}}(f_i)=\prod_{p=1}^lL_p^{-n_{ip}}X_i^- , \\
 \\
\psi_{\{ n\}}(H_i)=H_i .
\end{array}
$$
\end{theorem}
\begin{proof} The proof of this theorem is a direct verification of defining relations (\ref{sqgr}). The most nontrivial part is to verify deformed quantum Serre relations, i.e. the last two relations in (\ref{sqgr}).
The defining relations of $U_h({\frak g})$  imply the following relations for $\psi_{\{ n\}}(e_i)$,
$$
\sum_{k=0}^{1-a_{ij}}(-1)^k
\left[ \begin{array}{c} 1-a_{ij} \\ k \end{array} \right]_{q_i}
q^{k({d_j}n_{ij}-d_in_{ji})}\psi_{\{ n\}}(e_i)^{1-a_{ij}-k}\psi_{\{ n\}}(e_j)\psi_{\{ n\}}(e_i)^k =0 ,
$$
for any $i\neq j$.
Now using equation (\ref{eqpi}) we arrive to the quantum Serre relations for $e_i$ in (\ref{sqgr}).
\end{proof}

\begin{remark}
The general solution of equation (\ref{eqpi})
is given by
\begin{equation}\label{eq3}
n_{ij}=\frac 1{2d_j} (c_{ij} + {s_{ij}}),
\end{equation}
where $s_{ij}=s_{ji}$.
\end{remark}

We call the algebra $U_h^{s}({\frak g})$ the realization of the quantum group $U_h({\frak g})$ corresponding to the element $s\in W$.
\begin{remark}
Let $n_{ij}$ be a solution of the homogeneous system that corresponds to (\ref{eqpi}),
$$
d_in_{ji}-d_jn_{ij}=0.
$$
Then the map defined by
\begin{equation}\label{sautom}
\begin{array}{l}
X_i^+ \mapsto X_i^+ \prod_{p=1}^lL_p^{n_{ip}},\\
 \\
X_i^- \mapsto \prod_{p=1}^lL_p^{-n_{ip}}X_i^- , \\
 \\
H_i \mapsto H_i
\end{array}
\end{equation}
is an automorphism of $U_h({\frak g})$. Therefore for given element $s\in W$ the isomorphism $\psi_{\{ n\}}$
is defined uniquely up to automorphisms (\ref{sautom}) of $U_h({\frak g})$.
\end{remark}

Now we shall study the algebraic structure of $U_h^{s}({\frak g})$.
Denote by $U_h^{s}({\frak n}_\pm) $ the subalgebra in $U_h^{s}({\frak g})$ generated by
$e_i ~(f_i) ,i=1, \ldots, l$.
Let $U_h^{s}({\frak h})$ be the subalgebra in $U_h^{s}({\frak g})$ generated by $H_i,~i=1,\ldots ,l$.

We shall construct a Poincar\'{e}--Birkhoff-Witt basis for $U_h^{s}({\frak g})$.
It is convenient to introduce an operator $K\in {\rm End}~{\frak h}$ such that
\begin{equation}\label{Kdef}
KH_i=\sum_{j=1}^l{n_{ij} \over d_i}Y_j.
\end{equation}
In particular, we have
$$
{n_{ji} \over d_j}=(KH_j,H_i).
$$

Equation (\ref{eqpi}) is equivalent to the following equation for the operator $K$:
$$
K-K^* = {1+s \over 1-s}P_{{\h'}^*}.
$$

\begin{proposition}\label{rootss}
(i) For any solution of equation (\ref{eqpi}) and any normal ordering of the root system $\Delta_+$
the elements $e_{\beta}=\psi_{\{ n\}}^{-1}(X_{\beta}^+e^{hK\beta^\vee})$ and
$f_{\beta}=\psi_{\{ n\}}^{-1}(e^{-hK\beta^\vee}X_{\beta}^-),~\beta \in \Delta_+$
lie in the subalgebras $U_h^{s}({\frak n}_+)$ and $U_h^{s}({\frak n}_-)$, respectively.
For $\alpha \in \Delta_+$ we put ${(e_{\alpha})}^{(k)}=\frac{(e_\alpha)^{k}}{[k]_{q_\alpha}!}$, ${(f_{\alpha})}^{(k)}=\frac{(f_\alpha)^{k}}{[k]_{q_\alpha}!}$, $k\geq 0$.
The elements $e_{\beta},~\beta \in \Delta_+$ satisfy the following commutation relations
\begin{equation}\label{erel}
e_{\alpha}e_{\beta} - q^{(\alpha,\beta)+({1+s \over 1-s}P_{{\h'}^*}\alpha,\beta)}e_{\beta}e_{\alpha}= \sum_{\alpha<\delta_1<\ldots<\delta_n<\beta}C'(k_1,\ldots,k_n)
(e_{\delta_1})^{(k_1)}(e_{\delta_2})^{(k_2)}\ldots (e_{\delta_n})^{(k_n)},
\end{equation}
where $\alpha<\beta$, and $C'(k_1,\ldots,k_n)\in \mathbb{C}[q,q^{-1},q^{n_{ij}d_j},q^{-n_{ij}d_j}]_{i,j=1,\ldots ,l}$.

(ii) Moreover, the elements
$(e)^{(\bf r)}=(e_{\beta_1})^{(r_1)}\ldots (e_{\beta_D})^{(r_D)},~~(f)^{(\bf t)}=(f_{\beta_D})^{(t_D)}\ldots (f_{\beta_1})^{(t_1)}$
and $H^{\bf s}=H_1^{s_1}\ldots H_l^{s_l}$
for ${\bf r},~{\bf t}$, ${\bf s}\in {\Bbb N}^l$ form
topological bases of $U_h^{s}({\frak n}_+),~U_h^{s}({\frak n}_-)$ and $U_h^{s}({\frak h})$,
and the products $(f)^{(\bf t)}H^{\bf s}(e)^{(\bf r)}$ form a topological basis of
$U_h^{s}({\frak g})$. In particular, multiplication defines an isomorphism of ${\Bbb C}[[h]]$ modules$$
U_h^{s}({\frak n}_-)\otimes U_h^{s}({\frak h})\otimes U_h^{s}({\frak n}_+)\rightarrow U_h^{s}({\frak g}).
$$
\end{proposition}
\begin{proof}
Let $\beta=\sum_{i=1}^l m_i\alpha_i \in \Delta_+$ be a positive root,
$X_{\beta}^+\in U_h({\frak g})$ the corresponding root vector. Then $\beta^\vee=\sum_{i=1}^l m_id_iH_i$, and so
$K\beta^\vee=\sum_{i,j=1}^l m_in_{ij}Y_j$. Now the proof of the first statement follows immediately from
Proposition \ref{rootprop}, commutation relations (\ref{weight-root}), (\ref{qcom}) and the definition of the isomorphism
$\psi_{\{ n\}}$. The second assertion is a consequence of Proposition \ref{PBW}.
\end{proof}

The realizations $U_h^{s}({\frak g})$ of the quantum group $U_h({\frak g})$
are connected with quantizations of some nonstandard bialgebra structures on $\frak g$. At the quantum level
changing bialgebra structure corresponds to the so--called Drinfeld twist. We shall consider a particular class
of such twists described in the following proposition.
\begin{proposition}{\bf (\cite{ChP}, Proposition 4.2.13)}\label{twdef}
Let $(A,\mu , \imath , \Delta , \varepsilon , S)$ be a Hopf algebra over a commutative ring. Let $\mathcal F$ be an invertible element of $A\otimes A$
such that
\begin{equation}\label{twist}
\begin{array}{l}
{\mathcal F}_{12}(\Delta \otimes id)({\mathcal F})={\mathcal F}_{23}(id \otimes \Delta)({\mathcal F}),\\
\\
(\varepsilon \otimes id)({\mathcal F})=(id \otimes \varepsilon )({\mathcal F})=1.
\end{array}
\end{equation}
Then, $v=\mu (id\otimes S)({\mathcal F})$ is an invertible element of $A$ with
$$
v^{-1}=\mu (S\otimes id)({\mathcal F}^{-1}).
$$

Moreover, if we define $\Delta^{\mathcal F}:A\rightarrow A\otimes A$ and $S^{\mathcal F}:A\rightarrow A$ by
$$
\Delta^{\mathcal F}(a)={\mathcal F}\Delta(a){\mathcal F}^{-1},~~S^{\mathcal F}(a)=vS(a)v^{-1},
$$
then $(A,\mu , \imath , \Delta^{\mathcal F} , \varepsilon , S^{\mathcal F})$ is a Hopf algebra denoted by $A^{\mathcal F}$
and called the twist of $A$ by ${\mathcal F}$.
\end{proposition}

\begin{corollary}{\bf (\cite{ChP}, Corollary 4.2.15)}
Suppose that $A$ and ${\mathcal F}$ are as in Proposition \ref{twdef}, but assume in addition that $A$ is quasitriangular
with universal R--matrix $\mathcal R$. Then $A^{\mathcal F}$ is quasitriangular with universal R--matrix
\begin{equation}\label{rf}
{\mathcal R}^{\mathcal F}={\mathcal F}_{21}{\mathcal R}{\mathcal F}^{-1},
\end{equation}
where ${\mathcal F}_{21}=\sigma {\mathcal F}$.
\end{corollary}

Fix an element $s\in W$.
Consider the twist of the Hopf algebra $U_h({\frak g})$ by the element
\begin{equation}\label{Ftw}
{\mathcal F}=exp(-h\sum_{i,j=1}^l {n_{ji} \over d_j}Y_i\otimes Y_j) \in U_h({\frak h})\otimes U_h({\frak h}),
\end{equation}
where $n_{ij}$ is a solution of the corresponding equation (\ref{eqpi}).

This element satisfies conditions (\ref{twist}), and so $U_h({\frak g})^{\mathcal F}$ is a quasitriangular
Hopf algebra with the universal R--matrix ${\mathcal R}^{\mathcal F}={\mathcal F}_{21}{\mathcal R}{\mathcal F}^{-1}$,
where $\mathcal R$ is given by (\ref{univr}). We shall explicitly calculate the element ${\mathcal R}^{\mathcal F}$.
Substituting (\ref{univr}) and (\ref{Ftw}) into (\ref{rf}) and using (\ref{roots-cart}) we obtain
$$
\begin{array}{l}
{\mathcal R}^{\mathcal F}=exp\left[ h(\sum_{i=1}^l(Y_i\otimes H_i)+
\sum_{i,j=1}^l (-{n_{ij} \over d_i}+{n_{ji} \over d_j})Y_i\otimes Y_j) \right]\times \\
\prod_{\beta}
exp_{q_{\beta}}[(1-q_{\beta}^{-2})X_{\beta}^+e^{hK\beta^\vee} \otimes e^{-hK^*\beta^\vee}X_{\beta}^-],
\end{array}
$$
where $K$ is defined by (\ref{Kdef}).

Equip $U_h^{s}({\frak g})$ with the comultiplication given by :
$\Delta_{s}(x)=(\psi_{\{ n\}}^{-1}\otimes \psi_{\{ n\}}^{-1})\Delta_h^{\mathcal F}(\psi_{\{ n\}}(x))$.
Then $U_h^{s}({\frak g})$ becomes a quasitriangular Hopf algebra with the universal R--matrix
${\mathcal R}^{s}=\psi_{\{ n\}}^{-1}\otimes \psi_{\{ n\}}^{-1}{\mathcal R}^{\mathcal F}$. Using equation
(\ref{eqpi}) this R--matrix may be written as follows
\begin{equation}\label{rmatrspi}
\begin{array}{l}
{\mathcal R}^{s}=exp\left[ h(\sum_{i=1}^l(Y_i\otimes H_i)+
\sum_{i=1}^l {1+s \over 1-s }P_{{\h'}}H_i\otimes Y_i) \right]\times \\
\prod_{\beta}
exp_{q_{\beta}}[(1-q_{\beta}^{-2})e_{\beta} \otimes
e^{h{1+s \over 1-s}P_{{\h'}} \beta^\vee}f_{\beta}],
\end{array}
\end{equation}
where $P_{{\h'}}$ is the orthogonal projection operator onto $\h'$ in $\h$ with respect to the Killing form.

The element ${\mathcal R}^{s}$ may be also represented in the form
\begin{equation}\label{rmatrspi'}
\begin{array}{l}
{\mathcal R}^{s}=
\prod_{\beta}
exp_{q_{\beta}}[(1-q_{\beta}^{-2})e_{\beta}e^{-h({1+s \over 1-s}P_{{\h'}}+1)\beta^\vee}\otimes e^{h\beta^\vee}f_\beta]\times \\
exp\left[ h(\sum_{i=1}^l(Y_i\otimes H_i)+\sum_{i=1}^l {1+s \over 1-s }P_{{\h'}}H_i\otimes Y_i)\right] .
\end{array}
\end{equation}

The comultiplication $\Delta_{s}$ is given on generators by
$$
\begin{array}{l}
\Delta_{s}(H_i)=H_i\otimes 1+1\otimes H_i,\\
\\
\Delta_{s}(e_i)=e_i\otimes e^{hd_i({2 \over 1-s}P_{{\h'}}+P_{{\h'}^\perp})H_i}+1\otimes e_i,\\
\\
\Delta_{s}(f_i)=f_i\otimes e^{-hd_i{1+s \over 1-s}P_{{\h'}}H_i}+e^{-hd_iH_i}\otimes f_i,
\end{array}
$$
where $P_{{\h'}^\perp}$ is the orthogonal projection operator onto the orthogonal complement ${{\h'}^\perp}$ to $\h'$ in $\h$ with respect to the Killing form.

Finally, the new antipode $S_s(x)=\psi_{\{ n\}}^{-1}S_h^{\mathcal F}(\psi_{\{ n\}}(x))$ is given by
$$
S_s(e_i)=-e_ie^{-hd_i({2 \over 1-s}P_{{\h'}}+P_{{\h'}^\perp})H_i},~S_s(f_i)=-e^{hd_iH_i}f_ie^{hd_i{1+s \over 1-s}P_{{\h'}}H_i},~S_s(H_i)=-H_i.
$$

Note that the Hopf algebra $U_h^{s}({\frak g})$ is a quantization of the bialgebra structure on $\frak g$
defined by the cocycle
\begin{equation}\label{cocycles}
\delta (x)=({\rm ad}_x\otimes 1+1\otimes {\rm ad}_x)2r^{s}_+,~~ r^{s}_+\in {\frak g}\otimes {\frak g},
\end{equation}
where $r^{s}_+=r_+ + \frac 12 \sum_{i=1}^l {1+s \over 1-s }P_{{\h'}}H_i\otimes Y_i$, and $r_+$ is given by (\ref{rcl}).


\section{Normal orderings of root systems related to Weyl group elements}

\setcounter{equation}{0}
\setcounter{theorem}{0}

In this section we define certain normal orderings of root systems associated to Weyl group elements. The definition of subalgebras of $U_h(\g)$ having nontrivial characters will be given in terms of root vectors associated to such normal orderings.

Let  $W$ be the Weyul group of the pair $(\g,\h)$, $\Delta$ the root system of  $(\g,\h)$, $\Delta_+$ a system of positive roots in $\Delta$.
For every element $w\in W$ one can introduce the set $\Delta_w=\{\alpha \in \Delta_+: w(\alpha)\in -\Delta_+\}$, and the number of the elements in the set $\Delta_w$ is equal to the length $l(w)$ of the element $w$ with respect to the system of simple roots in $\Delta_+$.

Now recall that in the
classification theory of conjugacy classes in the Weyl group $W$ of the complex simple Lie algebra $\g$
the so-called primitive (or semi--Coxeter in another terminology) elements play a primary role. The primitive elements $w\in W$
are characterized by the property ${\rm det}(1-w)={\rm det}~a$, where $a$ is the Cartan matrix of $\g$.

Let $s$, as in the previous section, be an element of the Weyl group $W$ of the pair $(\g,\h)$.
According to the results of \cite{C} the element $s$  is primitive in the
Weyl group $W'$ of a regular semisimple Lie subalgebra $\g'\subset \g$ of the form
$$
\g'=\h'+\sum_{\alpha\in \Delta'}\g_\alpha,
$$
where $\Delta'$ is a root subsystem of the root system $\Delta$ of $\g$, $\g_\alpha$ is the root subspace of $\g$ corresponding to root $\alpha$,
and $\h'\subset \h$ is a Cartan subalgebra of $\g$ (it coincides with $\h'$ introduced in Section \ref{wqreal}).

Moreover, by Theorem C in \cite{C} $s$ can be represented as a product of two involutions,
\begin{equation}\label{inv}
s=s^1s^2,
\end{equation}
where $s^1=s_{\gamma_1}\ldots s_{\gamma_n}$, $s^2=s_{\gamma_{n+1}}\ldots s_{\gamma_{l'}}$, the roots in each of the sets $\gamma_1, \ldots, \gamma_n$ and ${\gamma_{n+1}}, \ldots, {\gamma_{l'}}$ are positive and mutually orthogonal, and
the roots $\gamma_1, \ldots, \gamma_{l'}$ form a linear basis of $\h'$, in particular $l'$ is the rank of $\g'$.

Let $\h_{\mathbb{R}}$ be the real form of $\h$, the real linear span of simple coroots in $\h$. The set of roots $\Delta$ is a subset of the dual space $\h_\mathbb{R}^*$.

The Weyl group element $s$ naturally acts on $\h_{\mathbb{R}}$ as an orthogonal transformation with respect to the scalar product induced by the Killing form of $\g$. Using the spectral theory of orthogonal transformations we can decompose $\h_{\mathbb{R}}$ into a direct orthogonal sum of $s$--invariant subspaces,
\begin{equation}\label{hdec}
\h_\mathbb{R}=\bigoplus_{i=0}^{K} \h_i,
\end{equation}
where we assume that $\h_0$ is the linear subspace of $\h_{\mathbb{R}}$ fixed by the action of $s$, and each of the other subspaces $\h_i\subset \h_\mathbb{R}$, $i=1,\ldots, K$, is either two--dimensional or one--dimensional and the Weyl group element $s$ acts on it as rotation with angle $\theta_i$, $0<\theta_i<\pi$ or as reflection with respect to the origin, respectively. Note that since $s$ has finite order $\theta_i=\frac{2\pi}{m_i}$, $m_i\in \mathbb{N}$. By Proposition 6.1 and the discussion after it in \cite{S10} the subspaces $\h_i$ can be chosen in such a way that each of them is invariant with respect to the involutions $s^1$ and $s^2$, and if $\h_i$ is one--dimensional one of the involutions acts on it in the trivial way.

We shall always assume that the subspaces $\h_i$ are chosen in this way.

Since the number of roots in the root system $\Delta$ is finite one can always choose elements $h_i\in \h_i$, $i=0,\ldots, K$, such that $h_i(\alpha)\neq 0$ for any root $\alpha \in \Delta$ which is not orthogonal to the $s$--invariant subspace $\h_i$ with respect to the natural pairing between $\h_{\mathbb{R}}$ and $\h_{\mathbb{R}}^*$.

Now we consider certain $s$--invariant subsets of roots $\overline{\Delta}_i$, $i=0,\ldots, K$, defined as follows
\begin{equation}\label{di}
{\overline{\Delta}}_i=\{ \alpha\in \Delta: h_j(\alpha)=0, j>i,~h_i(\alpha)\neq 0 \},
\end{equation}
where we formally assume that $h_{K+1}=0$.
Note that for some indexes $i$ the subsets ${\overline{\Delta}}_i$ are empty, and that the definition of these subsets depends on the order of terms in direct sum (\ref{hdec}).

 Now consider the nonempty $s$--invariant subsets of roots $\overline{\Delta}_{i_k}$, $k=0,\ldots, M$.
For convenience we assume that indexes $i_k$ are labeled in such a way that $i_j<i_k$ if and only if $j<k$.
According to this definition $\overline{\Delta}_{0}=\{\alpha \in \Delta: s\alpha=\alpha\}$ is the set of roots fixed by the action of $s$. Observe also that the root system $\Delta$ is the disjoint union of the subsets $\overline{\Delta}_{i_k}$,
$$
\Delta=\bigcup_{k=0}^{M}\overline{\Delta}_{i_k}.
$$

Now assume that
\begin{equation}\label{cond}
|h_{i_k}(\alpha)|>|\sum_{l\leq j<k}h_{i_j}(\alpha)|, ~{\rm for~any}~\alpha\in \overline{\Delta}_{i_k},~k=0,\ldots, M,~l<k.
\end{equation}
Condition (\ref{cond}) can be always fulfilled by suitable rescalings of the elements $h_{i_k}$.

Consider the element
\begin{equation}\label{hwb}
\bar{h}=\sum_{k=0}^{M}h_{i_k}\in \h_\mathbb{R}.
\end{equation}

From definition (\ref{di}) of the sets $\overline{\Delta}_i$ we obtain that for $\alpha \in \overline{\Delta}_{i_k}$
\begin{equation}\label{dech}
\bar{h}(\alpha)=\sum_{j\leq k}h_{i_j}(\alpha)=h_{i_k}(\alpha)+\sum_{j< k}h_{i_j}(\alpha)
\end{equation}
Now condition (\ref{cond}), the previous identity and the inequality $|x+y|\geq ||x|-|y||$ imply that for $\alpha \in \overline{\Delta}_{i_k}$ we have
$$
|\bar{h}(\alpha)|\geq ||h_{i_k}(\alpha)|-|\sum_{j< k}h_{i_j}(\alpha)||>0.
$$
Since $\Delta$ is the disjoint union of the subsets $\overline{\Delta}_{i_k}$, $\Delta=\bigcup_{k=0}^{M}\overline{\Delta}_{i_k}$, the last inequality ensures that  $\bar{h}$ belongs to a Weyl chamber of the root system $\Delta$, and one can define the subset of positive roots $\Delta_+$ and the set of simple positive roots $\Gamma$ with respect to that chamber. From condition (\ref{cond}) and formula (\ref{dech}) we also obtain that a root $\alpha \in \overline{\Delta}_{i_k}$ is positive if and only if
\begin{equation}\label{wc}
h_{i_k}(\alpha)>0.
\end{equation}
We denote by $(\overline{\Delta}_{i_k})_+$ the set of positive roots contained in $\overline{\Delta}_{i_k}$, $(\overline{\Delta}_{i_k})_+=\Delta_+\bigcap \overline{\Delta}_{i_k}$.

We shall also need a parabolic subalgebra $\p$ of $\g$ associated to the semisimple element $\bar{h}_0=\sum_{k=0,i_k>0}^{M}h_{i_k}\in \h_\mathbb{R}$ for $s\in W$. This subalgebra is defined with the help of the linear eigenspace decomposition of $\g$ with respect to the adjoint action of $\bar{h}_0$ on $\g$, $\g=\bigoplus_{m}(\g)_m$, $(\g)_m=\{ x\in \g \mid [\bar{h}_0,x]=mx\}$, $m \in \mathbb{R}$. By definition $\p=\bigoplus_{m\geq 0}(\g)_m$ is a parabolic subalgebra in $\g$, $\n=\bigoplus_{m>0}(\g)_m$ and $\l=\{x\in \g \mid [\bar{h}_0,x]=0\}$ are the nilradical and the Levi factor of $\p$, respectively. Note that we have natural inclusions of Lie algebras $\p\supset\b_+\supset\n$, where $\b_+$ is the Borel subalgebra of $\g$ corresponding to the system $\Gamma$ of simple roots, and $\overline{\Delta}_{0}$ is the root system of the reductive Lie algebra $\l$. We shall also need the corresponding parabolic subgroup $P$ of $G$, its Levi factor $L$ and the unipotent radical $N$.

\begin{proposition}\label{pord}
Let $s\in W$ be an element of the Weyl group $W$ of the pair $(\g,\h)$, $\Delta$ the root system of the pair $(\g,\h)$ and $\Delta_+$ the system of positive roots defined with the help of element (\ref{hwb}), $\Delta_+=\{\alpha \in \Delta|\bar{h}(\alpha)>0\}$.

Then the decomposition $s=s^1s^2$ is reduced in the sense that ${l}(s)={l}(s^2)+{l}(s^1)$, where ${l}(\cdot)$ is the length function in $W$ with respect to the system of simple roots in $\Delta_+$, and $\Delta_{s}=\Delta_{s^{2}}\bigcup s^2(\Delta_{s^{1}})$, $\Delta_{s^{-1}}=\Delta_{s^{1}}\bigcup s^1(\Delta_{s^{2}})$ (disjoint unions). Moreover, there is a normal ordering of the root system $\Delta_+$ of the following form
\begin{eqnarray}
\beta_1^1,\ldots, \beta_t^1,\beta_{t+1}^1, \ldots,\beta_{t+\frac{p-n}{2}}^1, \gamma_1,\beta_{t+\frac{p-n}{2}+2}^1, \ldots , \beta_{t+\frac{p-n}{2}+n_1}^1, \gamma_2, \nonumber \\
\beta_{t+\frac{p-n}{2}+n_1+2}^1 \ldots , \beta_{t+\frac{p-n}{2}+n_2}^1, \gamma_3,\ldots, \gamma_n, \beta_{t+p+1}^1,\ldots, \beta_{l(s^1)}^1,\ldots, \label{NO} \\
\beta_1^2,\ldots, \beta_q^2, \gamma_{n+1},\beta_{q+2}^2, \ldots , \beta_{q+m_1}^2, \gamma_{n+2}, \beta_{q+m_1+2}^2,\ldots , \beta_{q+m_2}^2, \gamma_{n+3},\ldots,  \nonumber \\
\gamma_{l'},\beta_{q+m_{l(s^2)}+1}^2, \ldots,\beta_{2q+2m_{l(s^2)}-(l'-n)}^2, \beta_{2q+2m_{l(s^2)}-(l'-n)+1}^2,\ldots, \beta_{l(s^2)}^2, \nonumber \\
\beta_1^0, \ldots, \beta_{D_0}^0, \nonumber
\end{eqnarray}
where
\begin{eqnarray*}
\{\beta_1^1,\ldots, \beta_t^1,\beta_{t+1}^1, \ldots,\beta_{t+\frac{p-n}{2}}^1, \gamma_1,\beta_{t+\frac{p-n}{2}+2}^1, \ldots , \beta_{t+\frac{p-n}{2}+n_1}^1, \gamma_2, \nonumber \\
\beta_{t+\frac{p-n}{2}+n_1+2}^1 \ldots , \beta_{t+\frac{p-n}{2}+n_2}^1, \gamma_3,\ldots, \gamma_n, \beta_{t+p+1}^1,\ldots, \beta_{l(s^1)}^1\}=\Delta_{s^1},
\end{eqnarray*}
\begin{eqnarray*}
\{\beta_{t+1}^1, \ldots,\beta_{t+\frac{p-n}{2}}^1, \gamma_1,\beta_{t+\frac{p-n}{2}+2}^1, \ldots , \beta_{t+\frac{p-n}{2}+n_1}^1, \gamma_2, \nonumber \\
\beta_{t+\frac{p-n}{2}+n_1+2}^1 \ldots , \beta_{t+\frac{p-n}{2}+n_2}^1, \gamma_3,\ldots, \gamma_n\}=\{\alpha\in \Delta_+|s^1(\alpha)=-\alpha\},
\end{eqnarray*}
\begin{eqnarray*}
\{\beta_1^2,\ldots, \beta_q^2, \gamma_{n+1},\beta_{q+2}^2, \ldots , \beta_{q+m_1}^2, \gamma_{n+2}, \beta_{q+m_1+2}^2,\ldots , \beta_{q+m_2}^2, \gamma_{n+3},\ldots,  \nonumber \\
\gamma_{l'},\beta_{q+m_{l(s^2)}+1}^2, \ldots,\beta_{2q+2m_{l(s^2)}-(l'-n)}^2, \beta_{2q+2m_{l(s^2)}-(l'-n)+1}^2,\ldots, \beta_{l(s^2)}^2\}=\Delta_{s^2},
\end{eqnarray*}
\begin{eqnarray*}
\{\gamma_{n+1},\beta_{q+2}^2, \ldots , \beta_{q+m_1}^2, \gamma_{n+2}, \beta_{q+m_1+2}^2,\ldots , \beta_{q+m_2}^2, \gamma_{n+3},\ldots,  \nonumber \\
\gamma_{l'},\beta_{q+m_{l(s^2)}+1}^2, \ldots,\beta_{2q+2m_{l(s^2)}-(l'-n)}^2\}=\{\alpha\in \Delta_+|s^2(\alpha)=-\alpha\},
\end{eqnarray*}
\begin{equation*}
\{\beta_1^0, \ldots, \beta_{D_0}^0\}=\Delta_0=\{\alpha\in \Delta_+|s(\alpha)=\alpha\},
\end{equation*}
and $s^1,s^2$ are the involutions entering decomposition (\ref{inv}), $s^1=s_{\gamma_1}\ldots s_{\gamma_n}$, $s^2=s_{\gamma_{n+1}}\ldots s_{\gamma_{l'}}$, the roots in each of the sets $\gamma_1, \ldots, \gamma_n$ and ${\gamma_{n+1}},\ldots, {\gamma_{l'}}$ are positive and mutually orthogonal.

The length of the ordered segment $\Delta_{\m_+}\subset \Delta$ in normal ordering (\ref{NO}),
\begin{eqnarray}
\Delta_{\m_+}=\gamma_1,\beta_{t+\frac{p-n}{2}+2}^1, \ldots , \beta_{t+\frac{p-n}{2}+n_1}^1, \gamma_2, \beta_{t+\frac{p-n}{2}+n_1+2}^1 \ldots , \beta_{t+\frac{p-n}{2}+n_2}^1, \nonumber \\
\gamma_3,\ldots, \gamma_n, \beta_{t+p+1}^1,\ldots, \beta_{l(s^1)}^1,\ldots, \beta_1^2,\ldots, \beta_q^2, \label{dn} \\
\gamma_{n+1},\beta_{q+2}^2, \ldots , \beta_{q+m_1}^2, \gamma_{n+2}, \beta_{q+m_1+2}^2,\ldots , \beta_{q+m_2}^2, \gamma_{n+3},\ldots, \gamma_{l'}, \nonumber
\end{eqnarray}
is equal to
\begin{equation}\label{dimm}
D-(\frac{l(s)-l'}{2}+D_0),
\end{equation}
where $D$ is the number of roots in $\Delta_+$, $l(s)$ is the length of $s$ and $D_0$ is the number of positive roots fixed by the action of $s$.

Moreover, for any two roots $\alpha, \beta\in \Delta_{\m_+}$ such that $\alpha<\beta$ the sum $\alpha+\beta$ cannot be represented as a linear combination $\sum_{k=1}^qc_k\gamma_{i_k}$, where $c_k\in \mathbb{N}$ and $\alpha<\gamma_{i_1}<\ldots <\gamma_{i_k}<\beta$.
\end{proposition}

\begin{proof}

First observe that each of the planes $\h_i$ is invariant not only with respect to the action of the Weyl group element $s$ but also with respect to the action of both involutions $s^1$ and $s^2$ which act as reflections in the plane $\h_i$.

Now consider the two--dimensional planes $\h_{i_k},i_k>0$ related to the definition of the set of positive roots.
The plane $\h_{i_k}$ is shown in Figure 1.

The vector $h_{i_k}$ is directed upwards at the picture, and the orthogonal projections of elements from $(\overline{\Delta}_{i_k})_+$ onto $\h_{i_k}$ are contained in the upper half--plane. The involutions $s^1$ and $s^2$ act in $\h_{i_k}$ as reflections with respect to the lines orthogonal to the vectors labeled by $v^1_k$ and $v^2_k$, respectively, at Figure 1, the angle between $v^1_k$ and $v^2_k$ being equal to $\pi-\theta_{i_k}/2$. The nonzero projections of the roots from the set $\{\gamma_1, \ldots \gamma_n\}\bigcap \overline{\Delta}_{i_k}$ onto the plane $\h_{i_k}$ have the same (or the opposite) direction as the vector $v^1_k$, and the nonzero projections of the roots from the set $\{\gamma_{n+1}, \ldots , \gamma_{l'}\}\bigcap \overline{\Delta}_{i_k}$ onto the plane $\h_{i_k}$ have the same (or the opposite) direction as the vector $v^2_k$.

For each of the involutions $s^1$ and $s^2$ we obviously have decompositions $\Delta_{s^{1,2}}=\bigcup_{k=0}^M{\overline{\Delta}_{i_k}^{1,2}}$, where ${\overline{\Delta}_{i_k}^{1,2}}={\overline{\Delta}_{i_k}}\bigcap \Delta_{s^{1,2}}$. In the plane $\h_{i_k}$, the elements from the sets ${\overline{\Delta}_{i_k}^{1,2}}$ are projected onto the interiors of the sectors labeled by ${\overline{\Delta}_{i_k}^{1,2}}$. Therefore the sets ${\overline{\Delta}_{i_k}^{1}}$ and ${\overline{\Delta}_{i_k}^{2}}$ have empty intersection. In case if $\h_{i_k}$ is an invariant line on which $s$ acts by multiplication by $-1$ one of the involution acts on it in the trivial way, and hence one of the sets ${\overline{\Delta}_{i_k}^{1}}$ and ${\overline{\Delta}_{i_k}^{2}}$ is empty. From the last two observations we deduce that the sets $\Delta_{s^{1}}$ and $\Delta_{s^{2}}$ have always empty intersection. In particular, by the results of \S 3 in \cite{Z1} the decomposition $s=s^1s^2$ is reduced in the sense that $l(s)=l(s^2)+l(s^1)$, and $\Delta_{s}=\Delta_{s^{2}}\bigcup s^2(\Delta_{s^{1}})$ (disjoint union).

$$
\xy/r10pc/: ="A",-(1,0)="B", "A",+(1,0)="C","A",-(0,1)="D","A",+(0.1,0.57)*{h_{i_k}},"A", {\ar+(0,+0.5)},"A", {\ar@{-}+(0,+1)}, "A";"B"**@{-},"A";"C"**@{-},"A";"D"**@{-},"A", {\ar+(0.6,0.13)},"A",+(0.65,0.18)*{v^2_k},"A",+(0.87,0.18)*{\overline{\Delta}_{i_k}^{2}},"A", {\ar@{-}+(0.9,0.41)}, "A",+(0.32,0)="D", +(-0.038,0.134)="E","D";"E" **\crv{(1.33,0.07)},"A",+(0.45,0.15)*{\psi_k},"A",+(0.49,0.05)*{\psi_k},"A", {\ar+(-0.6,0.23)},"A",+(-0.65,0.28)*{v^1_k},"A",+(-0.87,0.28)*{\overline{\Delta}_{i_k}^{1}},"A", {\ar@{-}+(-0.85,0.72)}, "A",+(-0.32,0)="F", +(0.066,0.21)="G","F";"G" **\crv{(0.67,0.07)},"A",+(-0.45,0.27)*{\varphi_k},"A",+(-0.49,0.08)*{\varphi_k}
\endxy
$$
\begin{center}
 Fig.1
\end{center}

Recall that by the results of \S 3 in \cite{Z1} for any element $w\in W$ one can always find a reduced decomposition $w_0=ww'$, where $w_0$ is the longest element of the Weyl group and $w'\in W$ is some element of the Weyl group. This implies, in particular, that any normal ordering of the root system $\Delta_+$ can be reduced by applying a number of elementary transpositions to one of the forms
\begin{equation}\label{no1}
\beta_1, \ldots, \beta_m, \ldots, \beta_D,
\end{equation}
\begin{equation}\label{no2}
\beta_D, \ldots, \beta_m, \ldots, \beta_1,
\end{equation}
where $\Delta_{w}=\{\beta_1, \ldots, \beta_m\}$.

Applying the last observation to the Weyl group element $s^1$ we obtain a normal ordering of the root system $\Delta_+$ of the form
\begin{equation}\label{o1}
\beta_1^1, \ldots, \beta_{l(s^1)}^1, \ldots, \beta_D,
\end{equation}
where $\Delta_{s^1}=\{\beta_1^1, \ldots, \beta_{l(s^1)}^1\}$.
Recalling that $\Delta_{s^{1}}\bigcap\Delta_{s^{2}}=\{\emptyset\}$ and using (\ref{no2}) one can reduce normal ordering (\ref{o1}) by applying a number of elementary transpositions to the form
\begin{equation}\label{o2}
\beta_1^1, \ldots, \beta_{l(s^1)}^1, \ldots,\beta_{1}^2,\ldots, \beta_{l(s^2)}^2,
\end{equation}
where $\Delta_{s^2}=\{\beta_1^2, \ldots, \beta_{l(s^2)}^2\}$.

Indeed, observe that the set $\Delta_{s^2}$ must contain at least one simple root since for any reduced decomposition $s^2=s_{i_1}\ldots s_{i_k}$ we have $\Delta_{s^2}=\{\alpha_{i_k},s_{i_k}\alpha_{i_{k-1}},\ldots,s_{i_k}\ldots s_{i_2}\alpha_{i_{1}}\}$, and $\alpha_{i_k}\in \Delta_{s^2}$. Since $\Delta_{s^1}\cap \Delta_{s^2}=\{\emptyset\}$ we can move all the simple roots in $\Delta_{s^2}$ and their linear combinations to the right in normal ordering (\ref{o1}) using elementary transpositions. Denote the ordered segment which consists of such linear combinations by $\Delta_{s^2}^c=\{{\beta_{1}^2}',\ldots, {\beta_{k}^2}'\}$. Thus we reduced (\ref{o1}) to the following form
$$
\beta_1^1, \ldots, \beta_{l(s^1)}^1, \ldots,{\beta_{1}^2}',\ldots, {\beta_{k}^2}'.
$$

Now assume that $\alpha \in \Delta_{s^2},~\alpha \not\in \Delta_{s^2}^c$ is a root such that there are no other roots to the right from it in the normal ordering above which belong to $\Delta_{s^2}$ except for those from $\Delta_{s^2}^c$. Then $\alpha$ can be moved to the right by a number of elementary transpositions. Indeed, since if $\alpha$ belongs to one of segments (\ref{rank2}) then by the choice of $\alpha$ either the entire segment belongs to $\Delta_{s^2}$ or the other end of the segment does not belong to $\Delta_{s^2}$. In the latter case we can apply the elementary transposition to the segment to move $\alpha$ to the right. Applying this procedure several times we obtain a new normal ordering of the form
$$
\beta_1^1, \ldots, \beta_{l(s^1)}^1, \ldots,\alpha,{\beta_{1}^2}',\ldots, {\beta_{k}^2}'.
$$
One can now proceed by induction to obtain normal ordering (\ref{o2}).

Now observe that since $\overline{\Delta}_{0}$ is the root system of the Levi factor $\l$ of the parabolic subalgebra $\p$ the elements of $\overline{\Delta}_{0}$ are linear combinations of roots from a subset $\Gamma_0\subset \Gamma$. Therefore noting that $\overline{\Delta}_{0}\bigcap(\Delta_{s^1}\bigcup \Delta_{s^2})=\{\emptyset\}$ and applying elementary transpositions to normal ordering (\ref{o1}) one can reduce it to a form similar to (\ref{o1}) in which the roots from the closed subset $(\overline{\Delta}_{0})_+=\overline{\Delta}_{0}\bigcap \Delta_+=\{\beta_1^0, \ldots, \beta_{D_0}^0\}$ form a segment. Moreover, we claim that one can reduce normal ordering (\ref{o2}) to the following form
\begin{equation}\label{o3}
\beta_1^1, \ldots, \beta_{l(s^1)}^1, \ldots,\beta_1^0, \ldots, \beta_{D_0}^0,\beta_{1}^2,\ldots, \beta_{l(s^2)}^2.
\end{equation}
In order to prove the last statement it suffices to verify that for any $\alpha \in (\overline{\Delta}_{0})_+$ and $\beta \in \Delta_{s^2}$ such that $\alpha +\beta=\gamma \in \Delta_+$ we have $\gamma \in \Delta_{s^2}$. Indeed, if $s^2\gamma \in \Delta_+$ then $s^2(\alpha+\beta)=\alpha+s^2\beta=s^2\gamma \in \Delta_+$, and hence $\alpha= s^2\gamma +(-s^2\beta)$ with $s^2\gamma, -s^2\beta\in \Delta_+$ and $s^2\gamma, -s^2\beta\not\in (\overline{\Delta}_{0})_+$. This is impossible since $\overline{\Delta}_{0}$ is the root system of the Levi factor $\l$ of the parabolic subalgebra $\p$.

Using the previous assertion and applying elementary transpositions one can also reduce normal ordering (\ref{o3}) to the form
\begin{equation}\label{o4}
\beta_1^1, \ldots, \beta_{l(s^1)}^1, \ldots,\beta_{1}^2,\ldots, \beta_{l(s^2)}^2,\beta_1^0, \ldots, \beta_{D_0}^0.
\end{equation}

Now we look at the segments $\beta_1^1, \ldots, \beta_{l(s^1)}^1$ and $\beta_{1}^2,\ldots, \beta_{l(s^2)}^2$ of normal ordering (\ref{o4}). We consider the case of the segment $\beta_1^1, \ldots, \beta_{l(s^1)}^1$ in detail. The other segment is treated in a similar way.

Recall that by Theorem A in \cite{Ric} every involution $w$ in the Weyl group $W$ is the longest element of the Weyl group of a Levi subalgebra in $\g$, and $w$ acts by multiplication by $-1$ at the Cartan subalgebra $\h_w\subset \h$ of the semisimple part $\m_w$ of that Levi subalgebra. By Lemma 5 in \cite{C} the involution $w$ can also be expressed as a product of ${\rm dim}~\h_w$ reflections from the Weyl group of the pair $(\m_w,\h_w)$, with respect to mutually orthogonal roots.
In case of the involution $s^1$, $s^1=s_{\gamma_1}\ldots s_{\gamma_n}$ is such an expression, and the roots ${\gamma_1},\ldots ,{\gamma_n}$ span the Cartan subalgebra $\h_{s^1}$.

Since $\m_{s^1}$ is the semisimple part of a Levi subalgebra, using elementary transpositions one can reduce the normal ordering of the segment $\beta_1^1, \ldots, \beta_{l(s^1)}^1$ to the form
\begin{equation}\label{o5}
\beta_1^1,\ldots, \beta_t^1,\beta_{t+1}^1, \ldots,\beta_{t+p}^1,\beta_{t+p+1}^1,\ldots, \beta_{l(s^1)}^1,
\end{equation}
where $\beta_{t+1}^1, \ldots,\beta_{t+p}^1$ is a normal ordering of the system $\Delta_+(\m_{s^1},\h_{s^1})$ of positive roots of the pair $(\m_{s^1},\h_{s^1})$. Now applying elementary transpositions we can reduce the ordering $\beta_{t+1}^1, \ldots,\beta_{t+p}^1$ to the form compatible with the decomposition $s^1=s_{\gamma_1}\ldots s_{\gamma_n}$ (see Appendix A).

Applying similar arguments to the involution $s^2$ and using the normal ordering of the positive roots of the pair $(\m_{s^2},\h_{s^2})$ inverse to that compatible with the decomposition $s^2=s_{\gamma_{n+1}}\ldots s_{\gamma_{l'}}$ we finally obtain the following normal ordering of the set $\Delta_+$
\begin{eqnarray}
\beta_1^1,\ldots, \beta_t^1,\beta_{t+1}^1, \ldots,\beta_{t+\frac{p-n}{2}}^1, \gamma_1,\beta_{t+\frac{p-n}{2}+2}^1, \ldots , \beta_{t+\frac{p-n}{2}+n_1}^1, \gamma_2, \nonumber \\
\beta_{t+\frac{p-n}{2}+n_1+2}^1 \ldots , \beta_{t+\frac{p-n}{2}+n_2}^1, \gamma_3,\ldots, \gamma_n, \beta_{t+p+1}^1,\ldots, \beta_{l(s^1)}^1,\ldots, \label{no}  \\
\beta_1^2,\ldots, \beta_q^2, \gamma_{n+1},\beta_{q+2}^2, \ldots , \beta_{q+m_1}^2, \gamma_{n+2}, \beta_{q+m_1+2}^2,\ldots , \beta_{q+m_2}^2, \gamma_{n+3},\ldots,  \nonumber \\
\gamma_{l'},\beta_{q+m_{l(s^2)}+1}^2, \ldots,\beta_{2q+2m_{l(s^2)}-(l'-n)}^2, \beta_{2q+2m_{l(s^2)}-(l'-n)+1}^2,\ldots, \beta_{l(s^2)}^2, \nonumber \\
\beta_1^0, \ldots, \beta_{N_0}^0, \nonumber
\end{eqnarray}
where
\begin{eqnarray*}
\gamma_{n+1},\beta_{q+2}^2, \ldots , \beta_{q+m_1}^2, \gamma_{n+2}, \beta_{q+m_1+2}^2,\ldots , \beta_{q+m_2}^2, \gamma_{n+3},\ldots,  \nonumber \\
\gamma_{l'},\beta_{q+m_{l(s^2)}+1}^2, \ldots,\beta_{2q+2m_{l(s^2)}-(l'-n)}^2
\end{eqnarray*}
is the normal ordering of the system of positive roots of the pair $(\m_{s^2},\h_{s^2})$ inverse to that compatible with the decomposition $s^2=s_{\gamma_{n+1}}\ldots s_{\gamma_{l'}}$. By construction ordering (\ref{no}) is the required ordering (\ref{NO}).

We claim that $t=l(s^1)-(t+p)$, i.e. there are equal numbers of roots on the left and on the right from the segment $\beta_{t+1}^1, \ldots,\beta_{t+p}^1$ in the segment $$\beta_1^1,\ldots, \beta_t^1,\beta_{t+1}^1, \ldots,\beta_{t+p}^1,\beta_{t+p+1}^1,\ldots, \beta_{l(s^1)}^1,$$ and
\begin{equation}\label{l1}
t=\frac{l(s^1)-p}{2}.
\end{equation}

Recall that by formula (3.5) in \cite{Z1}, given a reduced decomposition $w=s_{i_1}\ldots s_{i_m}$ of a Weyl group element $w$, one can represent $w$ as a product of reflections with respect to the roots from the set
$$
\Delta_{w^{-1}}=\{\beta_1=\alpha_{i_1},\beta_2=s_{i_1}\alpha_{i_2},\ldots,\beta_m=s_{i_1}\ldots s_{i_{m-1}}\alpha_{i_m}\},
$$
\begin{equation}\label{wr}
w=s_{i_1}\ldots s_{i_m}=s_{\beta_m}\ldots s_{\beta_1}
\end{equation}
Note that
$$
\beta_1=\alpha_{i_1},\beta_2=s_{i_1}\alpha_{i_2},\ldots,\beta_m=s_{i_1}\ldots s_{i_{m-1}}\alpha_{i_m}
$$
is the initial segment of a normal ordering of $\Delta_+$.

Applying this observation to the segment of the normal ordering (\ref{o5}) consisting of elements from the set $\Delta_{s^1}$ one can represent $(s^1)^{-1}=s^1$ as follows
\begin{equation}\label{s1}
s^1=s_{\beta_{l(s^1)}^1} \ldots s_{\beta_{t+p+1}^1} s_{\beta_{t+p}^1} \ldots s_{\beta_{t+1}^1}s_{\beta_t^1}\ldots s_{\beta_1^1},
\end{equation}
where if $s^1=s_{i_1}\ldots s_{i_{l(s^1)}}$ is the corresponding reduced decomposition of $s^1$ then
\begin{equation}\label{bt}
\beta_1^1=\alpha_{i_1},\beta_2^1=s_{i_1}\alpha_{i_2},\ldots,\beta_{l(s^1)}^1=s_{i_1}\ldots s_{i_{l(s^1)-1}}\alpha_{i_{l(s^1)}}.
\end{equation}

Since $\beta_{t+1}^1, \ldots,\beta_{t+p}^1$ is a normal ordering of the system of positive roots of the pair $(\m_{s^1},\h_{s^1})$ and $s^1$ is the longest element in the Weyl group of the pair $(\m_{s^1},\h_{s^1})$ we also have
\begin{equation}\label{s2}
s^1=s_{\beta_{t+p}^1} \ldots s_{\beta_{t+1}^1},
\end{equation}
and hence by (\ref{s1})
\begin{equation}\label{s3}
s^1=s_{\beta_{l(s^1)}^1}\ldots s_{\beta_{t+p+1}^1}s^1s_{\beta_t^1}\ldots s_{\beta_1^1}.
\end{equation}
From the last formula we deduce that
\begin{equation}\label{s4}
s_{\beta_1^1}\ldots s_{\beta_t^1}=(s_{\beta_t^1}\ldots s_{\beta_1^1})^{-1}(s^1)^{-1}s_{\beta_{l(s^1)}^1}\ldots s_{\beta_{t+p+1}^1}s^1s_{\beta_t^1}\ldots s_{\beta_1^1}.
\end{equation}

Now formula (\ref{wr}) implies that $s_{\beta_1^1}\ldots s_{\beta_t^1}=s_{i_t}\ldots s_{i_1}$, and
relations (\ref{bt}) combined with (\ref{wr}) yield
\begin{eqnarray*}
u(s^1)^{-1}s_{\beta_{l(s^1)}^1}\ldots s_{\beta_{t+p+1}^1}s^1u^{-1}=s_{u(s^1)^{-1}(\beta_{l(s^1)}^1)}\ldots s_{u(s^1)^{-1}(\beta_{t+p+1}^1)}= \\ =s_{s_{i_{t+p+1}}\ldots s_{i_{l(s^1)-1}}\alpha_{i_{l(s^1)}}}\ldots s_{i_{t+p+1}}=s_{i_{t+p+1}}\ldots s_{i_{l(s^1)}},
\end{eqnarray*}
where $u=s_{\beta_1^1}\ldots s_{\beta_t^1}=s_{i_t}\ldots s_{i_1}$.
Therefore from formula (\ref{s4}) we deduce that
\begin{equation}\label{s5}
s_{i_t}\ldots s_{i_1}=s_{i_{t+p+1}}\ldots s_{i_{l(s^1)}}.
\end{equation}

Since the decompositions in both sides of (\ref{s5}) are parts of reduced decompositions they are reduced as well, and we have $t=l(s^1)-(t+p)$. This is equivalent to formula (\ref{l1}).

Using a similar formula for the involution $s^2$ and recalling the definition of the orderings of positive roots of the pairs $(\m_{s^1},\h_{s^1})$, $(\m_{s^2},\h_{s^2})$ compatible with decompositions $s^1=s_{\gamma_1}\ldots s_{\gamma_n}$ and $s^2=s_{\gamma_{n+1}}\ldots s_{\gamma_{l'}}$ (see Appendix A) we deduce that the number of roots in the segment $\Delta_{\m_+}$ of normal ordering (\ref{no}),
\begin{eqnarray}
\Delta_{\m_+}=\gamma_1,\beta_{t+\frac{p-n}{2}+2}^1, \ldots , \beta_{t+\frac{p-n}{2}+n_1}^1, \gamma_2, \beta_{t+\frac{p-n}{2}+n_1+2}^1 \ldots , \beta_{t+\frac{p-n}{2}+n_2}^1, \nonumber \\
\gamma_3,\ldots, \gamma_n, \beta_{t+p+1}^1,\ldots, \beta_{l(s^1)}^1,\ldots, \beta_1^2,\ldots, \beta_q^2, \nonumber  \\
\gamma_{n+1},\beta_{q+2}^2, \ldots , \beta_{q+m_1}^2, \gamma_{n+2}, \beta_{q+m_1+2}^2,\ldots , \beta_{q+m_2}^2, \gamma_{n+3},\ldots, \gamma_{l'} \nonumber
\end{eqnarray}
is equal to $D-(\frac{l(s)-l'}{2}+D_0)$, where $l(s)=l(s^1)+l(s^2)$ is the length of $s$ and $D_0$ is the number of positive roots fixed by the action of $s$. This proves the second statement of the proposition.

Now let $\alpha, \beta\in \Delta_{\m_+}$, be any two roots such that $\alpha<\beta$. We shall show that the sum $\alpha+\beta$ cannot be represented as a linear combination $\sum_{k=1}^qc_k\gamma_{i_k}$, where $c_k\in \mathbb{N}$ and $\alpha<\gamma_{i_1}<\ldots <\gamma_{i_k}<\beta$.

Suppose that such a decomposition exists, $\alpha+\beta=\sum_{k=1}^qc_k\gamma_{i_k}$. Obviously at least one of the roots $\alpha, \beta$ must belong to the set $\Delta_+(\m_{s^1},\h_{s^1})\bigcap \Delta_{\m_+}$ or to the set $\Delta_+(\m_{s^2},\h_{s^2})\bigcap \Delta_{\m_+}$ for otherwise the set of roots $\gamma_{i_k}$ such that $\alpha<\gamma_{i_k}<\beta$ is empty.

Suppose that $\alpha \in \Delta_+(\m_{s^1},\h_{s^1})\bigcap \Delta_{\m_+}$. The other cases are considered in a similar way.

If $\beta \not \in \Delta_+(\m_{s^2},\h_{s^2})\bigcap \Delta_{\m_+}$ then $\alpha+\beta=\sum_{k=1}^qc_k\gamma_{i_k}$, and $\gamma_{i_k}\leq \gamma_{n}$. In particular, since $\alpha \in \h_{s^1}$ and $\gamma_{i_k}\in \h_{s^1}$ if $\gamma_{i_k}\leq \gamma_{n}$, we have $\beta=\sum_{k=1}^qc_k\gamma_{i_k}-\alpha \in \h_{s^1}$. This is impossible by the definition of the ordering of the set $\Delta_+(\m_{s^1},\h_{s^1})$ compatible with the decomposition $s^1=s_{\gamma_1}\ldots s_{\gamma_n}$.

If $\beta \in \Delta_+(\m_{s^2},\h_{s^2})\bigcap \Delta_{\m_+}$ then $\alpha+\beta=\sum_{k=1}^qc_k\gamma_{i_k}=\sum_{i_k\leq n}c_k\gamma_{i_k}+\sum_{i_k>n}c_k\gamma_{i_k}$. This implies
$$
\alpha-\sum_{i_k\leq n}c_k\gamma_{i_k}=\sum_{i_k>n}c_k\gamma_{i_k}-\beta.
$$
The l.h.s. of the last formula is an element of $\h_{s^1}$ and the r.h.s. is an element $\h_{s^2}$. Since $\h'=\h_{s^1}+\h_{s^2}$ is a direct vector space decomposition we infer that
$$
\alpha=\sum_{i_k\leq n,\alpha<\gamma_{i_k}}c_k\gamma_{i_k}
$$
and
$$
\beta=\sum_{i_k>n,\gamma_{i_k}<\beta}c_k\gamma_{i_k}.
$$
This is impossible by the definition of the orderings of the sets $\Delta_+(\m_{s^1},\h_{s^1})$ and $\Delta_+(\m_{s^2},\h_{s^2})$ compatible with the decompositions $s^1=s_{\gamma_1}\ldots s_{\gamma_n}$ and $s^2=s_{\gamma_{n+1}}\ldots s_{\gamma_{l'}}$, respectively.

Therefore the sum $\alpha+\beta$, $\alpha<\beta$, $\alpha,\beta\in \Delta_{\m_+}$ cannot be represented as a linear combination $\sum_{k=1}^qc_k\gamma_{i_k}$, where $c_k\in \mathbb{N}$ and $\alpha<\gamma_{i_1}<\ldots <\gamma_{i_k}<\beta$. This completes the proof of the proposition.

\end{proof}

We call the system of positive roots $\Delta_+$ ordered as in (\ref{NO}) the normally ordered system of positive roots associated to the (conjugacy class) of the Weyl group element $s\in W$.
We shall also need the circular ordering in the root system $\Delta$ corresponding to normal ordering (\ref{NO}) of the positive root system $\Delta_+$.

Let $\beta_{1}, \beta_{2}, \ldots, \beta_{D}$ be a normal ordering
of a positive root system $\Delta_+$. Then following \cite{KT3} one can introduce the corresponding circular normal ordering of the root system ${\Delta}$ where
the roots in ${\Delta}$ are located on a circle in
the following way
\begin{center}
\setlength{\unitlength}{0.6mm}
     \begin{picture}(180,120)(-40,0)
     \put(0,50){\makebox(0,0){$\beta_1$}}
     \put(4,69){\makebox(0,0){$\beta_2$}}
     \put(15,85){\circle*{1.5}}
     \put(31,96){\circle*{1.5}}
     \put(50,100){\circle*{1.5}}
     \put(69,96){\circle*{1.5}}
     \put(85,85){\circle*{1.5}}
     \put(96,69){\makebox(0,0){$\beta_D$}}
     \put(100,50){\makebox(0,0){$-\beta_1$}}
     \put(96,31){\makebox(0,0){$-\beta_2$}}
     \put(85,15){\circle*{1.5}}
     \put(69,4){\circle*{1.5}}
     \put(50,0){\circle*{1.5}}
     \put(31,4){\circle*{1.5}}
     \put(15,15){\circle*{1.5}}
     \put(4,31) {\makebox(0,0){-$\beta_D$}}
     \put(64,88){\vector (3,-2){10}}
     \put(36,12){\vector (-3,2){10}}
     \end{picture}
\end{center}
\vskip 0.5cm
\begin{center}
 Fig.2
\end{center}

Let $\alpha,\beta\in \Delta$. One says that the segment $[\alpha, \beta]$ of the circle
is minimal if it does not contain the opposite roots $-\alpha$ and $-\beta$ and the root $\beta$ follows after $\alpha$ on the circle above, the circle being oriented clockwise.
In that case one also says that $\alpha < \beta$ in the sense of the circular normal ordering,
\begin{equation}\label{noc}
\alpha < \beta \Leftrightarrow {\rm the ~segment}~ [\alpha, \beta]~{\rm  of~ the ~circle~
is~ minimal}.
\end{equation}

Later we shall need the following property of minimal segments which is a direct consequence of Proposition 3.3 in \cite{kh-t}.
\begin{lemma}\label{minsegm}
Let $[\alpha, \beta]$ be a minimal segment in a circular normal ordering of a root system $\Delta$. Then if $\alpha+\beta$ is a root we have
$$
\alpha<\alpha+\beta<\beta.
$$
\end{lemma}


\section{Nilpotent subalgebras and quantum groups}

\setcounter{equation}{0}
\setcounter{theorem}{0}

Now we can define the subalgebras of $U_h({\frak g})$ which resemble nilpotent subalgebras in
${\frak g}$ and possess nontrivial characters.
\begin{theorem}\label{qnil}
Let $s\in W$ be an element of the Weyl group $W$ of the pair $(\g,\h)$, $\Delta$ the root system of the pair $(\g,\h)$. Fix a decomposition (\ref{inv}) of $s$ and let $\Delta_+$ be the system of positive roots associated to $s$.
Let $U_h^{s}({\frak g})$ be the realization of the quantum group $U_h({\frak g})$ associated to $s$. Let $e_\beta\in U_h^{s}({\n_+})$, $\beta \in \Delta_+$ be the root vectors associated to the corresponding normal ordering (\ref{NO}) of $\Delta_+$.

Then elements $e_\beta\in U_h^{s}({\n_+})$, $\beta \in \Delta_{\m_+}$, where $\Delta_{\m_+}\subset \Delta$ is ordered segment (\ref{dn}), generate a subalgebra $U_h^{s}({\frak m}_+)\subset U_h^{s}({\frak g})$ such that $U_h^{s}({\frak m}_+)/hU_h^{s}({\frak m}_+)\simeq U({\m_+})$, where ${\m_+}$ is the Lie subalgebra of $\g$ generated by the root vectors $X_\alpha$, $\alpha\in \Delta_{\m_+}$.
The elements
$e^{\bf r}=e_{\beta_1}^{r_1}\ldots e_{\beta_D}^{r_D}$, $r_i\in \mathbb{N}$, $i=1,\ldots, D$ and $r_i$ can be strictly positive only if $\beta_i\in \Delta_{\m_+}$, form a
topological basis of $U_h^{s}({\frak m}_+)$.

Moreover the map $\chi_h^s:U_h^{s}({\frak m}_+)\rightarrow \mathbb{C}[[h]]$ defined on generators by
\begin{equation}\label{char}
\chi_h^s(e_\beta)=\left\{ \begin{array}{ll} 0 & \beta \not \in \{\gamma_1, \ldots, \gamma_{l'}\} \\ c_i & \beta=\gamma_i, c_i\in \mathbb{C}[[h]]
\end{array}
\right  .
\end{equation}
is a character of $U_h^{s}({\frak m}_+)$.
\end{theorem}

\begin{proof}
The first statement of the theorem follows straightforwardly from commutation relations (\ref{erel}) and Proposition \ref{rootss}.

In order to prove that the map $\chi_h^s:U_h^{s}({\frak m}_+)\rightarrow \mathbb{C}[[h]]$ defined by (\ref{char}) is a character of $U_h^{s}({\frak m}_+)$ we show that all relations (\ref{erel}) for $e_\alpha,~e_\beta$ with $\alpha,\beta \in \Delta_{\m_+}$, which are obviously defining relations in the subalgebra $U_h^{s}({\frak m}_+)$,  belong to the kernel of $\chi_h^s$. By definition the only generators of $U_h^{s}({\frak m}_+)$ on which $\chi_h^s$ does not vanish are $e_{\gamma_i}$, $i=1,\ldots,l'$. By the last statement in Proposition \ref{pord} for any two roots $\alpha, \beta\in \Delta_{\m_+}$ such that $\alpha<\beta$ the sum $\alpha+\beta$ cannot be represented as a linear combination $\sum_{k=1}^qc_k\gamma_{i_k}$, where $c_k\in \mathbb{N}$ and $\alpha<\gamma_{i_1}<\ldots <\gamma_{i_k}<\beta$. Hence for any two roots $\alpha, \beta\in \Delta_{\m_+}$ such that $\alpha<\beta$ the value of the map $\chi_h^s$ on the l.h.s. of the corresponding commutation relation (\ref{erel}) is equal to zero.

Therefore it suffices to prove that
$$
\chi_h^s(e_{\gamma_i}e_{\gamma_j} - q^{(\gamma_i,\gamma_j)+({1+s \over 1-s}P_{{\h'}^*}\gamma_i,\gamma_j)}e_{\gamma_j}e_{\gamma_j})=c_ic_j(1-q^{(\gamma_i,\gamma_j)+({1+s \over 1-s}P_{{\h'}^*}\gamma_i,\gamma_j)})=0,~i<j.
$$

The last identity holds provided $(\gamma_i,\gamma_j)+({1+s \over 1-s}P_{{\h'}}^*\gamma_i, \gamma_j)=0$ for $i<j$. As we shall see in the next lemma this is indeed the case.

Recall that $\gamma_1, \ldots , \gamma_{l'}$ form a basis of a subspace ${\h'}^*\subset \h^*$ on which $s$ acts without fixed points.
We shall study the matrix elements of the Cayley transform of the restriction of
$s$ to ${\h'}^*$ with respect to this basis.
\begin{lemma}\label{tmatrel}
Let $P_{{\h'}^*}$ be the orthogonal projection operator onto ${{\h'}^*}$ in $\h^*$, with respect to the Killing form.
Then the matrix elements of the operator ${1+s \over 1-s }P_{{\h'}^*}$ in the basis $\gamma_1, \ldots , \gamma_{l'}$ are of the form:
\begin{equation}\label{matrel}
\left( {1+s \over 1-s }P_{{\h'}^*}\gamma_i , \gamma_j \right)=
\varepsilon_{ij}(\gamma_i,\gamma_j),
\end{equation}
where
$$
\varepsilon_{ij} =\left\{ \begin{array}{ll}
-1 & i <j \\
0 & i=j \\
1 & i >j
\end{array}
\right  . .
$$
\end{lemma}
\begin{proof} (Compare with \cite{Bur}, Ch. V, \S 6, Ex. 3).
First we calculate the matrix of the Coxeter element $s$ with respect to the
basis $\gamma_1, \ldots , \gamma_{l'}$. We obtain this matrix in the form of the Gauss
decomposition of the operator $s:{\h'}^*\rightarrow {\h'}^*$.

Let $z_i=s \gamma_{i}$. Recall that $s_{\gamma_i}(\gamma_j)=\gamma_j-A_{ij}\gamma_i$, $A_{ij}=(\gamma_i^\vee,\gamma_j)$.
Using this definition the elements $z_i$ may be represented as:
$$
z_i=y_i -\sum_{k \geq i} A_{k i}y_{k},
$$
where
\begin{equation}\label{y}
y_{i}=s_{\gamma_1}\ldots s_{\gamma_{i-1}}\gamma_{i}.
\end{equation}
Using the matrix notation we can rewrite the last formula as follows:
\begin{equation}\label{2*}
\begin{array}{l}
z_{i}=
(I+V)_{k i}y_{k} , \\  \\ \mbox{ where } V_{k i}=
\left\{ \begin{array}{ll}
A_{k i} & k\geq i \\
0 & k < i
\end{array}
\right  .
\end{array}
\end{equation}

To calculate the matrix of the operator $s:{\h'}^*\rightarrow {\h'}^*$ with respect to the basis $\gamma_1, \ldots , \gamma_{l'}$ we have to express
the elements $y_{i}$ via $\gamma_1, \ldots , \gamma_{l'}$.
Applying the definition of reflections to (\ref{y}) we can pull out the element $\gamma_{i}$ to the right:
\[
y_{i}=\gamma_{i}-\sum_{k<i}A_{ki}y_{k}.
\]
Therefore
\[
\gamma_{i}=(I+U)_{k i}y_{k} ~, \mbox{ where } U_{ki}=
\left\{ \begin{array}{ll}
A_{ki} & k<i \\
0 & k \geq i
\end{array}
\right .
\]
Thus
\begin{equation}\label{1*}
y_{k}=(I+U)^{-1}_{jk}\gamma_{j}.
\end{equation}

Summarizing (\ref{1*}) and (\ref{2*}) we obtain:
\begin{equation}\label{**}
s \gamma_i=\left( (I+U)^{-1}(I-V) \right)_{ki}\gamma_k .
\end{equation}
This implies:
\begin{equation}\label{3*}
{1+s \over 1-s}P_{{\h'}^*}\gamma_i=\left( {2I+U-V \over U+V}\right)_{ki}\gamma_k .
\end{equation}

Observe that $(U+V)_{ki}=A_{ki}$ and $(2I+U-V)_{ij}=-A_{ij}\varepsilon_{ij}$.
Substituting these expressions into (\ref{3*}) we get :
\begin{eqnarray}
\left( {1+s \over 1-s }P_{{\h'}^*}\gamma_i , \gamma_j \right) =
-(A^{-1})_{kp}\varepsilon_{pi} A_{pi}(\gamma_j,\gamma_k)=\varepsilon_{ij}(\gamma_i,\gamma_j).
\end{eqnarray}
This completes the proof of the lemma, and thus Theorem \ref{qnil} is proved.
\end{proof}
\end{proof}

The matrix $A_{ij}$ is called the Carter matrix of $s$.
We shall also use the Lie subalgebra ${\m_-}$ of $\g$ generated by the root vectors $X_{-\alpha}$, $\alpha\in \Delta_{\m_+}$.


\section{Some specializations of the algebra $U_h^{s}({\frak g})$}\label{forms}

\setcounter{equation}{0}
\setcounter{theorem}{0}

In this section we introduce some forms of the quantum group $U_h^{s}({\frak g})$ which are similar to the rational form, the restricted integral form and to its specialization for the standard quantum group $U_h(\g)$. The motivations of the definitions given below will be clear in Section \ref{qplproups}. The results in this section are slight modifications of similar statements for $U_h(\g)$, and we refer to \cite{ChP}, Ch. 9 for the proofs.

We start with the observation that the numbers
\begin{equation}\label{rat}
p_{ij}=\left( {1+s \over 1-s }P_{{\h'}}Y_i,Y_j\right)+(Y_i,Y_j)
\end{equation}
are rational, $p_{ij}\in \mathbb{Q}$.

Indeed, let $\gamma_i^*$, $i=1,\ldots, l'$ be the basis of $\h'^*$ dual to $\gamma_i$, $i=1,\ldots, l'$ with respect to the restriction of the bilinear form $(\cdot,\cdot)$ to $\h'^*$. Since the numbers $(\gamma_i,\gamma_j)$ are integer each element $\gamma_i^*$ has the form $\gamma_i^*=\sum_{j=1}^{l'}m_{ij}\gamma_j$, where $m_{ij}\in \mathbb{Q}$. Now we have
\begin{eqnarray*}
\left( {1+s \over 1-s }P_{{\h'}}Y_i,Y_j\right)+(Y_i,Y_j)= \qquad \qquad \qquad \qquad \qquad \qquad \qquad \\ \qquad \qquad =\sum_{k,l,p,q=1}^{l'}\gamma_k(Y_i)\gamma_l(Y_j)\left( {1+s \over 1-s }P_{{\h'}^*}\gamma_p,\gamma_q\right)m_{kp}m_{lq}+(Y_i,Y_j).
\end{eqnarray*}
All the terms in the r.h.s. of the last identity are rational since $\gamma_i(Y_j)\in \mathbb{Z}$ for any $i=1,\ldots, l'$ and $j=1,\ldots,l$ because $Y_i$ are the fundamental weights, the numbers $\left( {1+s \over 1-s }P_{{\h'}^*}\gamma_p,\gamma_q\right)$ are integer by Lemma \ref{tmatrel}, the coefficients $m_{ij}$ are rational as we observed above, and the scalar products $(Y_i,Y_j)$ of the fundamental weights are rational. Therefore the numbers $p_{ij}$ are rational.

Denote by $d$ the smallest integer number divisible by all the denominators of the rational numbers $p_{ij}/2$, $i,j=1,\ldots, l$.

Let
$U_q^{s}({\frak g})$ be the $\mathbb{C}(q^{\frac{1}{2d}})$-subalgebra of $U_h^{s}({\frak g})$ generated by the elements $e_i , f_i , t_i^{\pm 1}=\exp(\pm\frac{h}{2d}H_i),~i=1, \ldots, l$.

The defining relations for the algebra $U_q^{s}({\frak g})$ are
\begin{equation}\label{sqgr1}
\begin{array}{l}
t_it_j=t_jt_i,~~t_it_i^{-1}=t_i^{-1}t_i=1,~~ t_ie_jt_i^{-1}=q^{\frac{a_{ij}}{2d}}e_j, ~~t_if_jt_i^{-1}=q^{-\frac{a_{ij}}{2d}}f_j,\\
\\
e_i f_j -q^{ c_{ij}} f_j e_i = \delta _{i,j}{K_i -K_i^{-1} \over q_i -q_i^{-1}} , c_{ij}=\left( {1+s \over 1-s }P_{{\h'}^*}\alpha_i , \alpha_j \right)\\
 \\
K_i=t_i^{2d d_i}, \\
 \\
\sum_{r=0}^{1-a_{ij}}(-1)^r q^{r c_{ij}}
\left[ \begin{array}{c} 1-a_{ij} \\ r \end{array} \right]_{q_i}
(e_i )^{1-a_{ij}-r}e_j (e_i)^r =0 ,~ i \neq j , \\
 \\
\sum_{r=0}^{1-a_{ij}}(-1)^r q^{r c_{ij}}
\left[ \begin{array}{c} 1-a_{ij} \\ r \end{array} \right]_{q_i}
(f_i )^{1-a_{ij}-r}f_j (f_i)^r =0 ,~ i \neq j .
\end{array}
\end{equation}
Note that by the choice of $d$ we have $q^{c_{ij}}\in \mathbb{C}[q^{\frac{1}{2d}},q^{-\frac{1}{2d}}]$.

The second form of $U_h^{s}({\frak g})$ is a subalgebra $U_\mathcal{A}^{s}(\g)$ in $U_q^s(\g)$ over the ring $\mathcal{A}=\mathbb{C}[q^{\frac{1}{2d}},q^{-\frac{1}{2d}},\frac{1}{[2]_{q_i}},\ldots ,\frac{1}{[r]_{q_i}}]$, where $i=1, \ldots ,l$, $r$ is the maximal number $k_i$ that appears in the right--hand sides of formulas (\ref{qcom}) for various $\alpha$ and $\beta$. $U_\mathcal{A}^{s}(\g)$ is the subalgebra in $U_q^s(\g)$ generated over $\mathcal{A}$ by the elements $t_i^{\pm 1},~{K_i -K_i^{-1} \over q_i -q_i^{-1}},~e_i,~f_i,~i=1,\ldots ,l$.

The most important for us is the specialization $U_\varepsilon^s(\g)$ of $U_\mathcal{A}^{s}(\g)$, $U_\varepsilon^s(\g)=U_\mathcal{A}^{s}(\g)/(q^{\frac{1}{2d}}-\varepsilon^{\frac{1}{2d}})U_\mathcal{A}^{s}(\g)$, $\varepsilon\in \mathbb{C}^*$, $[r]_{\varepsilon_i}!\neq 0$, $i=1, \ldots ,l$. Note that $[r]_{1}!\neq 0$, and hence one can define the specialization $U_1^s(\g)$.

$U_q^s(\g)$, $U_\mathcal{A}^{s}(\g)$ and $U_\varepsilon^s(\g)$ are Hopf algebras with the comultiplication induced from $U_h^s(\g)$.

If in addition $\varepsilon^{2d_i}\neq 1$, $i=1,\ldots, l$, then $U_\varepsilon^s(\g)$ is generated over $\mathbb{C}$ by $t_i^{\pm 1},~e_i,~f_i,~i=1,\ldots ,l$ subject to relations (\ref{sqgr1}) where $q=\varepsilon$.

The algebra $U_\mathcal{A}^{s}(\g)$ has a similar description.

The elements $t_i$ are central in the algebra $U_1^s(\g)$, and the quotient of $U_1^s(\g)$ by the two--sided ideal generated by $t_i-1$ is isomorphic to $U(\g)$.

If $V$ is a $U_q^s(\g)$--module then its weight spaces are all the nonzero $\mathbb{C}(q^{\frac{1}{2d}})$--linear subspaces of the form
$$
V_{\bf c}=\{v\in V, t_iv=c_iv,~c_i\in \mathbb{C}(q{^\frac{1}{2d}})^*,~i=1,\ldots ,l\}.
$$
The l-tuple ${\bf c}=(c_1,\ldots, c_l)\in (\mathbb{C}(q^{\frac{1}{2d}})^*)^l$ is called a weight.

If ${\bf c}'=(c_1',\ldots, c_l')$ is another weight one says that ${\bf c}'\leq {\bf c}$ if $c_i'c_i=q^{\frac{1}{2d}\beta(H_i)}$ for some $\beta \in Q^+=\bigoplus_{i=1}^l\mathbb{N}\alpha_i$ and  all $i=1,\ldots ,l$.

A highest weight $U_q^s(\g)$--module is a $U_q^s(\g)$--module $V$ which contains a weight vector $v\in V_{\bf c}$ annihilated by the action of all elements $e_i$  and such that $V=U_q^s(\g)v$. In that case we also have a weight space decomposition
$$
V=\bigoplus_{{\bf c}'\leq {\bf c}}V_{{\bf c}'},
$$
and ${\rm dim}_{\mathbb{C}(q^{\frac{1}{2d}})}V_{\bf c}=1$. In particular, ${\bf c}$ is uniquely defined by $V$. It is called the highest weight of $V$, and $v$ is called the highest weight vector.

Verma and finite--dimensional irreducible $U_q^s(\g)$--modules are defined in the usual way. For instance, the Verma module $M_q(q)$ corresponding to a highest weight $q \in \h^*$ is the quotient of $U_q^s(\g)$ by the right ideal generated by $e_i$ and $t_i-q^{\frac{1}{2d}q(H_i)}$, where $i=1,\ldots ,l$.

The image of $1$ in $M_q(q)$ is the highest weight vector $v_{q}$ in $M_q(q)$. For $q \in P_+=\{\mu \in \h^*, \mu(H_i)\in \mathbb{N}~{\rm for~all}~i\}$ the unique irreducible quotient $V_q(q)$ of $M_q(q)$ is a finite--dimensional irreducible representation of $U_q^s(\g)$.

If $V$ is a highest weight $U_q(\g)$--module with highest weight vector $v$ then $V_\mathcal{A}=U_\mathcal{A}^{s}(\g)v$ is a $U_\mathcal{A}^{s}$--submodule of $V$ which has weight decomposition induced by that of $V$.

Moreover, $V_\mathcal{A}\otimes_{\mathcal{A}}\mathbb{C}(q^{\frac{1}{2d}})\simeq V$, $V_\mathcal{A}$ is the direct sum of its intersections with the weight spaces of $V$, each such intersection is a free $\mathcal{A}$--module of finite rank, and $\overline{V}=V_\mathcal{A}/(q^{\frac{1}{2d}}-1)V_\mathcal{A}$ is naturally a $U(\g)$-module. In particular for $q\in P_+$, $M(q)=\overline{M}_q(q)$ and $V(q)=\overline{V}_q(q)$ are the Verma and the finite--dimensional irreducible $U(\g)$--modules with highest weight $q$.

For Verma and finite--dimensional representations every nonzero weight subspace has weight of the form
$(q^{\frac{1}{2d}q(H_1)},\ldots,q^{\frac{1}{2d}q(H_l)})$, where $q \in P=\{\mu \in \h^*, \mu(H_i)\in \mathbb{Z}~{\rm for~all}~i\}$. One simply calls such a subspace a subspace of weight $q$.

Similarly one can define highest weight, Verma and finite--dimensional $U_\varepsilon^s(\g)$--modules in case when $\varepsilon$ is transcendental; one should just replace $q$ with $\varepsilon$ in the definitions above for the algebra $U_q^s(\g)$.

For the solution $n_{ij}=\frac{1}{2d_j}c_{ij}$ to equations (\ref{eqpi}) the root vectors $e_{\beta},f_\beta$ belong to all the above introduced subalgebras of $U_h({\frak g})$, and one can define Poincar\'{e}--Birkhoff--Witt bases for them in a similar way. From now on we shall assume that the solution to equations (\ref{eqpi}) is fixed as above, $n_{ij}=\frac{1}{2d_j}c_{ij}$.

If we define for ${\bf s}=(s_1,\ldots s_l)\in {\Bbb Z}^{l}$
$$
t^{\bf s}=t_1^{s_1}\ldots t_l^{s_l}
$$
and denote by $U_q^s({\frak n}_+),U_q^s({\frak n}_-)$ and $U_q^s({\frak h})$ the subalgebras of $U_q^s({\frak g})$
generated by the
$e_i$, $f_i$ and by the $t_i$, respectively, then the elements $e^{\bf r}=e_{\beta_1}^{r_1}\ldots e_{\beta_D}^{r_D},~~f^{\bf t}=f_{\beta_D}^{t_D}\ldots f_{\beta_1}^{t_1}$
and $t^{\bf s}$, for ${\bf r},~{\bf t}\in {\Bbb N}^D$,
${\bf s}\in {\Bbb Z}^l$, form bases of $U_q^s({\frak n}_+),U_q^s({\frak n}_-)$ and $U_q^s({\frak h})$,
respectively, and the products $e^{\bf r}t^{\bf s}f^{\bf t}$ form a basis of
$U_q^s({\frak g})$. In particular, multiplication defines an isomorphism:
$$
U_q^s({\frak n}_-)\otimes U_q^s({\frak h}) \otimes U_q^s({\frak n}_+)\rightarrow U_q^s({\frak g}).
$$

By specializing the above constructed basis for $q=\varepsilon$ we obtain a similar basis for $U_\varepsilon^s(\g)$.

Let $U_\mathcal{A}^{s}({\frak n}_+),U_\mathcal{A}^{s}({\frak n}_-)$ be the subalgebras of $U_\mathcal{A}^{s}({\frak g})$
generated by the
$e_i$ and by the $f_i$, $i=1,\ldots,l$, respectively.
Using the root vectors $e_{\beta}$ and $f_{\beta}$ we can construct a basis of
$U_\mathcal{A}^{s}({\frak g})$.
Namely, the elements $e^{\bf r}$, $f^{\bf t}$ for ${\bf r},~{\bf t}\in {\Bbb N}^N$
form bases of $U_\mathcal{A}^{s}({\frak n}_+),U_\mathcal{A}^{s}({\frak n}_-)$,
respectively.

The elements
$$
\left[ \begin{array}{l}
K_i;c \\
r
\end{array} \right]_{q_i}=\prod_{s=1}^r \frac{K_i q_i^{c+1-s}-K_i^{-1}q_i^{s-1-c}}{q_i^s-q_i^{-s}}~,~i=1,\ldots,l,~c\in \mathbb{Z},~r\in \mathbb{N}
$$
belong to $U_\mathcal{A}^{s}(\g)$. Denote by $U_\mathcal{A}^{s}(\h)$ the subalgebra of $U_\mathcal{A}^{s}(\g)$ generated by those elements and by $t_i^{\pm 1}$, $i=1,\ldots,l$.

Then multiplication defines an isomorphism of $\mathcal{A}$ modules:
$$
U_\mathcal{A}^{s}({\frak n}_-)\otimes U_\mathcal{A}^{s}({\frak h}) \otimes U_\mathcal{A}^{s}({\frak n}_+)\rightarrow U_\mathcal{A}^{s}({\frak g}).
$$

A basis for $U_\mathcal{A}^{s}(\h)$ is a little bit more difficult to describe. We do not need its explicit description (see \cite{ChP}, Proposition 9.3.3 for details).

None of the subalgebras of $U_h^s(\g)$ introduced above is quasitriangular. However, one can define an action of R--matrix (\ref{rmatrspi}) in the finite--dimensional representations of $U_q^s(\g)$,  $U_\mathcal{A}^{s}(\g)$ and $U_\varepsilon^s(\g)$. Indeed, observe that one can write R--matrix (\ref{rmatrspi}) in the factorized form
\begin{equation}\label{}
\mathcal{R}^s=\mathcal{E}\widetilde{\mathcal{R}},
\end{equation}
where
$$
\mathcal{E}=exp\left[ h(\sum_{i=1}^l(Y_i\otimes H_i)+
\sum_{i=1}^l {1+s \over 1-s }P_{{\h'}}H_i\otimes Y_i) \right]
$$
and
$$
\widetilde{\mathcal{R}}=
\sum_{u_1,\ldots,u_D=0}^\infty \prod_{r=1}^D
q_{\beta_r}^{\frac{1}{2}u_r(u_r+1)}(1-q_{\beta_r}^{-2})^{u_r}e_{\beta_r}^{u_r}\otimes e^{u_rh{1+s \over 1-s}P_{{\h'}} \beta^\vee}(f_{\beta_r})^{(u_r)},
$$
where the order of the factors in the product is such that the $\beta_r$--term appears to the right of the $\beta_s$--term if $r>s$.

Using the fact that the numbers $p_{ij}$ defined by (\ref{rat}) are of the form $p_{ij}=\frac{2v_{ij}}{d}$, $v_{ij}\in \mathbb{Z}$ one can check that actually $e^{u_rh{1+s \over 1-s}P_{{\h'}} \beta^\vee}\in U_\mathcal{A}^{s}(\g)$. Therefore $e_{\beta}^{u_r}\otimes e^{u_rh{1+s \over 1-s}P_{{\h'}} \beta^\vee}(f_{\beta})^{(u_r)}\in U_\mathcal{A}^{s}(\g)\otimes U_\mathcal{A}^{s}(\g)$.

For every two finite--dimensional $U_\mathcal{A}^{s}(\g)$-modules $V$ and $W$ only finitely many terms in the expression for $\widetilde{\mathcal{R}}$ act nontrivially on $V\otimes W$ since the action of root vectors on $V$ and $W$ is nilpotent. Therefore the action of the element $\widetilde{\mathcal{R}}$ in the space $V\otimes W$ is well defined.

Moreover, if $V_\mu$ and $W_q$ are two weight subspaces of $V$ and $W$ of weights $\mu,q\in P$ then one can define an action of $\mathcal{E}$ in $V_\mu\otimes W_q$ as multiplication by the scalar $q^{(q, \mu)+
({1+s \over 1-s }P_{{\h'}^*}q,\mu)}$. Since the numbers $p_{ij}$ defined by (\ref{rat}) are of the form $p_{ij}=\frac{2v_{ij}}{d}$ this scalar is an element of $\mathcal{A}=\mathbb{C}[q^{\frac{1}{2d}},q^{-\frac{1}{2d}}]$.

If we define an action of the element $\mathcal{R}^s$ in $V\otimes W$ as the composition of the above defined action of the operators  $\mathcal{E}$ and $\widetilde{\mathcal{R}}$ in $V\otimes W$ and denote the obtained  operator by ${{R}}^{V,W}$ then
one can check that
$$
{{R}}^{V,W}(\pi_V\otimes \pi_W)\Delta_s(x){{{R}}^{V,W}}^{-1}=(\pi_W\otimes \pi_V)\Delta_s^{opp}(x),
$$
where $\pi_V$, $\pi_W$ are the representations $V$ and $W$ and $\Delta_s$ is the comultiplication on $U_\mathcal{A}^{s}(\g)$.
Moreover, ${{R}}^{V,W}$ satisfies the quantum Yang-Baxter equation.

By specializing $q$ to a particular value $q=\varepsilon$ one can obtain an operator with similar properties acting in the tensor product of any two finite--dimensional $U_\varepsilon^{s}(\g)$-modules.
Obviously, the above construction can be applied in case of the algebra $U_q^{s}(\g)$ as well.

Finally we discuss an obvious analogue of the subalgebra $U_h^{s}({\frak m}_+)\subset U_h^{s}({\frak g})$ for $U_\mathcal{A}^{s}(\g)$.

Let $U_\mathcal{A}^{s}({\frak m}_+)\subset U_\mathcal{A}^{s}({\frak g})$ be the subalgebra generated by elements $e_\beta\in U_\mathcal{A}^{s}({\n_+})$, $\beta \in \Delta_{\m_+}$, where $\Delta_{\m_+}\subset \Delta$ is the ordered segment (\ref{dn}). The defining relations in the subalgebra $U_\mathcal{A}^{s}({\frak m}_+)$ are given by formula (\ref{erel}),
\begin{equation}\label{erel1}
e_{\alpha}e_{\beta} - q^{(\alpha,\beta)+({1+s \over 1-s}P_{{\h'}^*}\alpha,\beta)}e_{\beta}e_{\alpha}= \sum_{\alpha<\gamma_1<\ldots<\gamma_n<\beta}C'(k_1,\ldots,k_n)
e_{\gamma_1}^{k_1}e_{\gamma_2}^{k_2}\ldots e_{\gamma_n}^{k_n},~\alpha<\beta,
\end{equation}
where $C'(k_1,\ldots,k_n)\in \mathcal{A}$.

The elements
$e^{\bf r}=e_{\beta_1}^{r_1}\ldots e_{\beta_D}^{r_D}$, $r_i\in \mathbb{N}$, $i=1,\ldots, D$, and $r_i$ can be strictly positive only if $\beta_i\in \Delta_{\m_+}$, form a
basis of $U_\mathcal{A}^{s}({\frak m}_+)$.

Obviously $U_\mathcal{A}^{s}({\frak m}_+)/(q^{\frac{1}{2d}}-1)U_\mathcal{A}^{s}({\frak m}_+)\simeq U({\m_+})$, where ${\m_+}$ is the Lie subalgebra of $\g$ generated by the root vectors $X_\alpha$, $\alpha\in \Delta_{\m_+}$.

Moreover, the map $\chi_\mathcal{A}^s:U_\mathcal{A}^{s}({\frak m}_+)\rightarrow \mathcal{A}$ defined on generators by
$$
\chi_\mathcal{A}^s(e_\beta)=\left\{ \begin{array}{ll} 0 & \beta \not \in \{\gamma_1, \ldots, \gamma_{l'}\} \\ c_i & \beta=\gamma_i, c_i\in \mathcal{A}
\end{array}
\right  .
$$
is a character of $U_\mathcal{A}^{s}({\frak m}_+)$.

By specializing $q$ to a particular value $q=\varepsilon$ one can obtain a subalgebra $U_\varepsilon^{s}({\frak m}_+)\subset U_\varepsilon^{s}(\g)$ with similar properties.


\section{Poisson--Lie groups}\label{plgroupss}

\setcounter{equation}{0}
\setcounter{theorem}{0}

In this section we recall some notions concerned with Poisson--Lie groups (see \cite{ChP}, \cite{Dm},
\cite{fact}, \cite{dual}). These facts will  be used in Section \ref{qplproups} to define q-W algebras.

Let $G$ be a finite--dimensional Lie group equipped with a Poisson bracket,
$\frak g$ its Lie algebra. $G$ is called
a Poisson--Lie group if the multiplication $G\times G \rightarrow G$ is a
Poisson map.
A Poisson bracket satisfying this axiom is degenerate and, in particular, is
identically zero
at the unit element of the group. Linearizing this bracket at the unit element
defines the
structure of a Lie algebra in the space $T^*_eG\simeq {\frak g}^*$.
The pair (${\frak g},{\frak g}^{*})$ is called the tangent bialgebra of $G$.

Lie brackets in $\frak{g}$ and $\frak{g}^{*}$ satisfy the following
compatibility condition:

{\em Let }$\delta: {\frak g}\rightarrow {\frak g}\wedge {\frak g}$ {\em be
the dual  of the commutator map } $[,]_{*}: {\frak g}^{*}\wedge
{\frak g}^{*}\rightarrow {\frak g}^{*}$. {\em Then } $\delta$ {\em is a
1-cocycle on} $  {\frak g}$ {\em (with respect to the adjoint action
of } $\frak g$ {\em on} ${\frak g}\wedge{\frak g}$).

Let $c_{ij}^{k}, f^{ab}_{c}$ be the structure constants of
${\frak g}, {\frak g}^{*}$ with respect to the dual bases $\{e_{i}\},
\{e^{i}\}$ in ${\frak g},{\frak g}^{*}$. The compatibility condition
means that

$$
c_{ab}^{s} f^{ik}_{s} ~-~ c_{as}^{i} f^{sk}_{b} ~+~ c_{as}^{k}
f^{si}_{b} ~-~ c_{bs}^{k} f^{si}_{a} ~+~ c_{bs}^{i} f^{sk}_{a} ~~=
~~0.
$$
This condition is symmetric with respect to exchange of $c$ and
$f$. Thus if $({\frak g},{\frak g}^{*})$ is a Lie bialgebra, then
$({\frak g}^{*}, {\frak g})$ is also a Lie bialgebra.

The following proposition shows that the category of finite--dimensional Lie
bialgebras is isomorphic to
the category of finite--dimensional connected simply connected Poisson--Lie
groups.
\begin{proposition}{\bf (\cite{ChP}, Theorem 1.3.2)}
If $G$ is a connected simply connected finite--dimensional Lie group, every
bialgebra structure on $\frak g$
is the tangent bialgebra of a unique Poisson structure on $G$ which makes $G$
into a Poisson--Lie group.
\end{proposition}

Let $G$ be a finite--dimensional Poisson--Lie group, $({\frak g},{\frak g}^{*})$
the tangent bialgebra of $G$.
The connected simply connected finite--dimensional
Poisson--Lie group corresponding to the Lie bialgebra $({\frak g}^{*}, {\frak
g})$ is called the dual
Poisson--Lie group and denoted by $G^*$.

$({\frak g},{\frak g}^{*})$ is called a factorizable Lie
bialgebra if the following conditions are satisfied (see \cite{Dm}, \cite{fact}):
\begin{enumerate}
\item
${\frak g}${\em \ is equipped with a non--degenerate invariant
scalar product} $\left( \cdot ,\cdot \right)$.

We shall always identify ${\frak g}^{*}$ and ${\frak g}$ by means of this
scalar product.

\item  {\em The dual Lie bracket on }${\frak g}^{*}\simeq {\frak g}${\em \
is given by}
\begin{equation}
\left[ X,Y\right] _{*}=\frac 12\left( \left[ rX,Y\right] +\left[ X,rY\right]
\right) ,X,Y\in {\frak g},  \label{rbr}
\end{equation}
{\em where }$r\in {\rm End}\ {\frak g}${\em \ is a skew symmetric linear
operator
(classical r-matrix).}

\item  $r${\em \ satisfies} {\em the} {\em modified classical Yang-Baxter
identity:}
\begin{equation}
\left[ rX,rY\right] -r\left( \left[ rX,Y\right] +\left[ X,rY\right] \right)
=-\left[ X,Y\right] ,\;X,Y\in {\frak g}{\bf .}  \label{cybe}
\end{equation}
\end{enumerate}

Define operators $r_\pm \in {\rm End}\ {\frak g}$ by
\[
r_{\pm }=\frac 12\left( r\pm id\right) .
\]
We shall need some properties of the operators $r_{\pm }$.
Denote by ${\frak b}_\pm$ and ${\frak n}_\mp$ the image and the kernel of the
operator
$r_\pm $:
\begin{equation}\label{bnpm}
{\frak b}_\pm = Im~r_\pm,~~{\frak n}_\mp = Ker~r_\pm.
\end{equation}
\begin{proposition}{\bf (\cite{BD}, \cite{rmatr})}\label{bpm}
Let $({\frak g}, {\frak g}^*)$ be a factorizable Lie bialgebra. Then

(i) ${\frak b}_\pm \subset {\frak g}$ is a Lie subalgebra, the subspace ${\frak
n}_\pm$ is a Lie ideal in
${\frak b}_\pm,~{\frak b}_\pm^\perp ={\frak n}_\pm$.

(ii) ${\frak n}_\pm$ is an ideal in ${\frak {g}}^{*}$.

(iii) ${\frak b}_\pm$ is a Lie subalgebra in ${\frak {g}}^{*}$. Moreover ${\frak
b}_\pm ={\frak {g}}^{*}/ {\frak n}_\pm$.

(iv) $({\frak b}_\pm,{\frak b}_\pm ^*)$ is a subbialgebra of $({\frak
{g}},{\frak {g}}^{*})$ and
$({\frak b}_\pm,{\frak b}_\pm ^*)\simeq ({\frak b}_\pm,{\frak b}_\mp)$. The
canonical paring
between ${\frak b}_\mp$ and ${\frak b}_\pm$is given by
\begin{equation}
(X_\mp ,Y_\pm )_\pm=(X_\mp,r_\pm^{-1}Y_\pm ) ,~ X_\mp \in {\frak b}_\mp ;~ Y_\pm
\in {\frak b}_\pm .
\end{equation}
\end{proposition}
The classical Yang--Baxter equation implies that $r_{\pm }$ , regarded as a
mapping from
${\frak g}^{*}$ into ${\frak g}$, is a Lie algebra homomorphism.
Moreover, $r_{+}^{*}=-r_{-},$\ and $r_{+}-r_{-}=id.$

Put ${\frak {d}}={\frak g\oplus {\g}}$ (direct sum of two
copies). The mapping
\begin{eqnarray}\label{imbd}
{\frak {g}}^{*}\rightarrow {\frak {d}}~~~:X\mapsto (X_{+},~X_{-}),~~~X_{\pm
}~=~r_{\pm }X
\end{eqnarray}
is a Lie algebra embedding. Thus we may identify ${\frak g^{*}}$ with a Lie
subalgebra in ${\frak {d}}$.

Naturally, embedding (\ref{imbd}) extends to a homomorphism
$$
G^*\rightarrow G\times G,~~L\mapsto (L_+,L_-).
$$
We shall identify $G^*$ with the corresponding subgroup in $G\times G$.


\section{Poisson reduction}\label{poisred}

\setcounter{equation}{0}
\setcounter{theorem}{0}

In this section we recall basic facts on Poisson reduction (see \cite{W},
\cite{RIMS}).

Let $M,~B,~B'$ be Poisson manifolds. Two Poisson surjections
$$
\begin{array}{ccccc}
&  & M &  &  \\
& \stackrel{\pi^{\prime } }{\swarrow } &  & \stackrel{\pi }{\searrow } &  \\
B^{\prime } &  &  &  & B
\end{array}
$$
form a dual pair if the pullback $\pi^{^{\prime
}*}C^\infty(B^{\prime })$ is the centralizer of $\pi^* C^\infty (B)$ in the
Poisson algebra
$C^\infty (M)$. In that case the sets $B^{\prime }_b=\pi^{\prime } \left( \pi
^{-1}(b) \right),~b\in
B$ are Poisson submanifolds in $B^{\prime }$ (see \cite{W}) called reduced
Poisson manifolds.

Fix an element $b\in B$. Then the algebra of functions $C^\infty (B^{\prime
}_b)$ may be described as
follows. Let $I_b$ be the ideal in $C^\infty (M)$ generated by elements
${\pi}^*(f),~f\in C^\infty (B),~f(b)=0$. Denote $M_b=\pi^{-1}(b)$. Then the
algebra $C^\infty (M_b)$
is simply the quotient of $C^\infty (M)$ by $I_b$.  Denote by
$P_b:C^\infty (M)\rightarrow C^\infty (M)/I_b=C^\infty (M_b)$
the canonical projection onto the quotient.
\begin{lemma}\label{redspace}
Suppose that the map $f\mapsto f(b)$ is
a character of the Poisson algebra $C^\infty (B)$. Then one can define an action
of
the Poisson algebra $C^\infty (B)$ on the space $C^\infty (M_b)$ by
\begin{equation}\label{redact}
f\cdot \varphi=P_b(\{ {\pi}^*(f), \tilde \varphi \}),
\end{equation}
where $f\in C^\infty (B)$, $\varphi \in C^\infty (M_b)$ and $\tilde \varphi \in
C^\infty (M)$ is a
representative of $\varphi$ in $C^\infty (M)$ such that $P_b(\tilde
\varphi)=\varphi$.
Moreover, $C^\infty (B^{\prime }_b)$ is the subspace of invariants in $C^\infty
(M_b)$
with respect to this action.
\end{lemma}
\begin{proof}
Let $\varphi \in C^\infty (M_b)$. Choose a representative
$\tilde \varphi \in C^\infty (M)$ such that $P_b(\tilde \varphi)=\varphi$.
Since the map $f\mapsto f(b)$ is
a character of the Poisson algebra $C^\infty (B)$, Hamiltonian vector fields of
functions
${\pi}^*(f),~f\in C^\infty (B)$ are tangent to the surface $M_b$. Therefore the r.h.s. of (\ref{redact}) only depends on $\varphi$ but not on
the representative
$\tilde \varphi$, and hence
formula (\ref{redact})
defines an action of the Poisson algebra $C^\infty (B)$ on the space $C^\infty
(M_b)$.

Using the definition of the dual
pair we obtain that
$\varphi={\pi^{\prime }}^*(\psi)$ for some $\psi \in C^\infty(B^{\prime }_b)$ if
and only if
$P_b(\{ {\pi}^*(f), \tilde \varphi\})=0$ for every $f\in C^\infty (B)$.

Finally we obtain that $C^\infty (B^{\prime }_b)$ is exactly the subspace of
invariants in $C^\infty (M_b)$
with respect to action (\ref{redact}).
\end{proof}
\begin{definition}
The algebra $C^\infty (B^{\prime }_b)$ is called a reduced Poisson algebra.
We also denote it by $C^\infty (M_b)^{C^\infty (B)}$.
\end{definition}
\begin{remark}\label{redpoisalg}
Note that the description of the algebra $C^\infty (M_b)^{C^\infty (B)}$
obtained in Lemma \ref{redspace}
is independent
of both the manifold $B^{\prime }$ and the projection $\pi^{\prime }$.
Observe also that the reduced space $B^{\prime }_b$ may be identified with a
cross--section
of the action of the Poisson algebra $C^\infty (B)$ on $M_b$ by Hamiltonian
vector fields in case when this action is free.
In particular, in that case $B^{\prime }_b$ may be regarded as a submanifold in $M_b$.
\end{remark}

An important example of dual pairs is provided by Poisson group actions.
Recall that a (local) Poisson group action of a Poisson--Lie group $A$ on a Poisson
manifold $M$
is a (local) group action $A\times M\rightarrow M$ which is also a Poisson map (as
usual, we suppose that $A\times M$ is equipped with the product Poisson
structure).

In \cite{RIMS} it is proved that if the space $M/A$ is a smooth
manifold,
there exists a unique Poisson structure on $M/A$
such that the canonical projection $M\rightarrow M/A$ is a Poisson map.

Let $\frak a$ be the Lie algebra of $A$. Denote by $\langle\cdot,\cdot\rangle$
the
canonical paring between ${\frak a}^*$ and $\frak a$.
A map $\mu :M\rightarrow A^*$ is called a moment map for a (local) right Poisson group
action
$A\times M\rightarrow M$ if (see \cite{Lu})
\begin{equation}
L_{\widehat X} \varphi =\langle \mu^*(\theta_{A^*}) , X \rangle (\xi_\varphi ) ,
\end{equation}
where $\theta_{A^*}$ is the universal right--invariant Maurer--Cartan form on
$A^*$, $X \in {\frak a}$,
$\widehat X$ is the corresponding vector field on $M$ and
$\xi_\varphi $ is the Hamiltonian vector field of $\varphi \in C^\infty (M)$.

By Theorem 4.9 in \cite{Lu} one can always equip $A^*$ with a Poisson structure in
such a way
that $\mu$ becomes a Poisson mapping. From the definition of the moment map it follows that if $M/A$ is a smooth
manifold then the canonical projection $M\rightarrow M/A$ and the moment map
$\mu:M\rightarrow A^*$
form a dual pair (see \cite{Lu} for details).

The main example of Poisson group actions is the so--called dressing action.
The dressing action may be described as follows (see \cite{Lu}, \cite{RIMS}).
\begin{proposition}\label{dressingact}
Let $G$ be a connected simply connected Poisson--Lie group with factorizable
tangent Lie bialgebra,
$G^*$ the dual group. Then there exists a unique right local Poisson group action
$$
G^*\times G\rightarrow G^*,~~((L_+,L_-),g)\mapsto g\circ (L_+,L_-),
$$
such that the identity mapping $\mu: G^* \rightarrow G^*$ is the moment map for
this action.

Moreover, let $q:G^* \rightarrow G$ be the map defined by
$$
q(L_+,L_-)=L_-L_+^{-1}.
$$
Then
$$
q(g\circ (L_+,L_-))=g^{-1}L_-L_+^{-1}g.
$$
\end{proposition}
The notion of Poisson group actions may be generalized as follows.
Let $A\times M \rightarrow M$ be a Poisson group action of a Poisson--Lie
group $A$ on a Poisson manifold $M$.
A subgroup $K\subset A$ is called  admissible if the set
$C^\infty \left( M\right) ^K$ of $K$-invariants is a Poisson subalgebra in $
C^\infty \left( M\right)$. If space $M/K$ is a smooth manifold, we may identify
the algebras
$C^\infty(M/K)$ and $C^\infty \left( M\right) ^K$.  Hence there exists a Poisson
structure on $M/K$
such that the canonical projection $M\rightarrow M/K$ is a Poisson map.
\begin{proposition}\label{admiss}{\bf (\cite{RIMS}, Theorem 6; \cite{Lu}, \S 2)}
Let $\left( {\frak a},{\frak a}^{*}\right) $ be the tangent
Lie bialgebra of a Poisson--Lie group $A$. A connected Lie subgroup $K\subset A$ with Lie algebra
${\frak k}\subset {\frak a}$ is admissible if the annihilator ${\frak k}^{\perp }$ of $\mathfrak{k}$ in ${\frak a}^{*}$ is a Lie subalgebra ${\frak k}^{\perp }\subset
{\frak a}^{*}$.
\end{proposition}
We shall need the following particular example of dual pairs arising from
Poisson group actions.

Let $A\times M \rightarrow M$ be a right (local) Poisson group action of a Poisson--Lie
group $A$ on a manifold $M$.
Suppose that this action possesses a moment map $\mu : M\rightarrow A^*$.
Let $K$ be an admissible subgroup in $A$. Denote by $\frak k$ the Lie algebra of
$K$.
Assume that ${\frak k}^\perp \subset {\frak a}^*$ is a Lie subalgebra in ${\frak
a}^*$.
Suppose also that there is a splitting ${\frak a}^*={\frak t}\oplus {\frak
k}^\perp$, and that
$\frak t$ is a Lie subalgebra in ${\frak a}^*$. Then the linear space ${\frak
k}^*$ is naturally
identified with $\frak t$.
Assume that $A^*=K^\perp T$ as a manifold, where $K^\perp , T$ are the Lie subgroups of $A^*$ corresponding to the Lie subalgebras
${\frak k}^\perp , {\frak t}\subset {\frak a}^*$, respectively. Denote by $\pi_{K^\perp} , \pi_{T}$ the projections onto
$K^\perp$ and $T$ in this decomposition. Suppose that $K^\perp$ is a
connected subgroup in $A^*$ and that for any $k^\perp \in K^\perp$ the transformation
\begin{eqnarray}
{\frak t}\rightarrow {\frak t},\\
t\mapsto ({\rm Ad}(k^\perp)t)_{{\frak t}}, \nonumber
\end{eqnarray}
where the subscript ${{\frak t}}$ stands for the ${{\frak t}}$--component with respect to the decomposition ${\frak a}^*={\frak t}\oplus {\frak
k}^\perp$, is invertible.
The following proposition is a slight generalization of Theorem 14 in \cite{S3}. The proof given in \cite{S3} still applies under the conditions imposed on $K, K^\perp$ and $T$ above.
\begin{proposition}\label{QPmoment}
Define a map $\overline{\mu}:M\rightarrow T$ by
$$
\overline{\mu}=\pi_{T}\mu.
$$
Then

(i)
$\overline{\mu}^*\left( C^\infty \left( T\right)\right)$ is a Poisson subalgebra
in $C^\infty \left( M\right)$,
and hence one can equip $T$ with a Poisson structure such that
$\overline{\mu}:M\rightarrow T$ is
a Poisson map.

(ii)Moreover, the algebra $C^\infty \left( M\right) ^K$ is the centralizer of
$\overline{\mu}^*\left( C^\infty \left( T\right)\right)$ in the Poisson algebra
$C^\infty \left( M\right)$.
In particular, if $M/K$ is a smooth manifold the maps
\begin{equation}\label{dp}
\begin{array}{ccccc}
&  & M &  &  \\
& \stackrel{\pi }{\swarrow } &  & \stackrel{\overline{\mu}}{\searrow } &  ,\\
M/K &  &  &  & T
\end{array}
\end{equation}
form a dual pair.
\end{proposition}

\begin{remark}\label{remred}
Let $t\in T$ be as in Lemma \ref{redspace}.
Assume that $\pi(\overline{\mu}^{-1}(t))$ is a smooth manifold ($M/K$ does not
need to be smooth). Then
the algebra $C^\infty(\pi(\overline{\mu}^{-1}(t)))$ is isomorphic to
the reduced Poisson algebra $C^\infty(\overline{\mu}^{-1}(t))^{C^\infty(T)}$.
\end{remark}

\begin{remark}
In the proof of Theorem 14 in \cite{S3} we obtained a formula which relates the action of the Poisson algebra $\overline{\mu}^*\left( C^\infty \left( T\right)\right)$ and the action of $K$ on $C^\infty \left( M\right)$.
Let $X \in {\frak k}$ and  $\widehat X$ be the corresponding vector field on $M$,
$\xi_\varphi $ the Hamiltonian vector field of $\varphi \in C^\infty (M)$. Then
\begin{equation}\label{actrelp}
\begin{array}{l}
L_{\widehat X} \varphi =
\langle {\rm Ad}(\pi_{K^\perp}\mu )({\overline{\mu}}^*\theta_{T}),X \rangle (\xi_\varphi )=\\
\\
\langle {\rm Ad}(\pi_{K^\perp}\mu )(\theta_{T}),X \rangle ({\overline{\mu}}_*(\xi_\varphi )) ,
\end{array}
\end{equation}
where $\theta_{T}$ is the universal right invariant Cartan form on $T$.
\end{remark}


\section{Quantization of Poisson--Lie groups and q-W algebras}\label{qplproups}

\setcounter{equation}{0}
\setcounter{theorem}{0}

Let $\frak g$ be a finite--dimensional complex simple Lie algebra, $\h\subset \g$ its Cartan subalgebra. Let $s\in W$ be an element of the Weyl group $W$ of the pair $(\g,\h)$ and $\Delta_+$ the system of positive roots associated to $s$.
Observe that
cocycle (\ref{cocycles}) equips
$\frak g$ with
the structure of a factorizable Lie bialgebra.
Using the identification
${\rm End}~{\frak g}\cong {\frak g}\otimes {\frak g}$ the corresponding
r--matrix may be represented as
$$
r^{s}=P_+-P_-+{1+s \over 1-s}P_{{\h'}},
$$
where $P_+,P_-$ and $P_{{\h'}}$ are the projection operators onto ${\frak n}_+,{\frak
n}_-$ and ${\frak h}'$ in
the direct sum
$$
{\frak g}={\frak n}_+ +{\frak h}'+{\h'}^\perp + {\frak n}_-,
$$
where ${\h'}^\perp$ is the orthogonal complement to $\h'$ in $\h$ with respect to the Killing form.

Let $G$ be the connected simply connected simple Poisson--Lie group with the
tangent Lie bialgebra $({\frak g},{\frak g}^*)$,
$G^*$ the dual group. Note that by Proposition 17 in \cite{S3} Poisson--Lie groups $G^*$ corresponding to different Weyl group elements $s\in W$ are isomorphic as Poisson manifolds, and as Poisson manifolds all Poisson--Lie groups $G^*$ are isomorphic to the Poisson--Lie group $G^*_0$ associated to the standard bialgebra structure on $\g$ with $r=P_+-P_-$. This is the quasiclassical version of Theorem \ref{newreal}.

Observe that $G$ is an algebraic group (see \S 104,
Theorem 12 in \cite{Z}).

Note also that
$$
r^{s}_+=P_+ + {1 \over 1-s}P_{{\h'}}+\frac{1}{2}P_{{\h'}^\perp},~~r^{s}_-=-P_- + {s \over 1-s}P_{{\h'}}-\frac{1}{2}P_{{\h'}^\perp},
$$
and hence the subspaces ${\frak b}_\pm$ and ${\frak n}_\pm$ defined by
(\ref{bnpm}) coincide with
the Borel subalgebras in $\frak g$ and their nilradicals, respectively.
Therefore every element $(L_+,L_-)\in G^*$ may be uniquely written as
\begin{equation}\label{fact}
(L_+,L_-)=(h_+,h_-)(n_+,n_-),
\end{equation}
where $n_\pm \in N_\pm$, $h_+=exp(({1 \over 1-s}P_{{\h'}}+\frac{1}{2}P_{{\h'}^\perp})x),~h_-=exp(({s \over 1-s}P_{{\h'}}-\frac{1}{2}P_{{\h'}^\perp})x),~x\in
{\frak h}$.
In particular, $G^*$ is a solvable algebraic subgroup in $G\times G$.

In terms of factorization (\ref{fact}) and a similar factorization in the Poisson--Lie group $G^*_0$,
$$
(L_+',L_-')=(h_+',h_-')(n_+',n_-'),~n_\pm \in N_\pm,~h_+'=exp(\frac{1}{2}x'),~h_-'=exp(-\frac{1}{2}x'),~x'\in
{\frak h},
$$
the Poisson manifold isomorphism $G^*_0\rightarrow G^*$ established in Proposition 17 in \cite{S3} takes the form
$$
L'\mapsto tL't^{-1}=L,~L'\in G^*_0, L\in G^*,~L'=L_-'(L_+')^{-1}, L=L_-L_+^{-1},~t=e^{Ax'},
$$
where $A\in {\rm End}~{\h}$ is an arbitrary endomorphism of $\h$ commuting with $s$ and satisfying the equation
$$
A-A^*=\frac{1}{2}\frac{1+s}{1-s}P_{{\h'}}.
$$

For every algebraic variety $V$ we denote by ${\mathbb{C}}[V]$ the algebra of
regular functions on $V$.
Our main object will be the algebra of regular functions on $G^*$, ${\mathbb{C}}[G^*]$.
This algebra may be explicitly described as follows.
Let $\pi_V$ be a finite--dimensional representation of $G$. Then matrix
elements of $\pi_V(L_\pm)$ are well--defined functions on $G^*$, and ${\mathbb{C}}[G^*]$ is the subspace
in $C^\infty(G^*)$ generated by matrix elements of $\pi_V(L_\pm)$, where $V$
runs through all finite--dimensional
representations of $G$.

The elements $L^{\pm,V}=\pi_V(L_\pm)$ may be viewed as elements of the space
${\mathbb{C}}[G^*]\otimes {\rm End}V$. For every two finite--dimensional ${\frak g}$
modules $V$ and $W$
we denote ${r^s_+}^{VW}=(\pi_V\otimes \pi_W)r^s_+$, where $r^s_+$ is regarded as
an
element of ${\frak g}\otimes {\frak g}$.
\begin{proposition}{\bf (\cite{dual}, Section 2)}\label{pbff}
${\mathbb{C}}[G^*]$ is a Poisson subalgebra in the Poisson algebra $C^\infty(G^*)$,
the Poisson brackets
of the elements $L^{\pm,V}$ are given by
\begin{equation}\label{pbf}
\begin{array}{l}
\{L^{\pm,W}_{1},L^{\pm,V}_{2}\}~=~
2[{r_+^s}^{VW},L^{\pm,W}_{1}L^{\pm,V}_{2}],\\
\\
\{L^{-,W}_{1},L^{+,V}_{2}\}~=2[{r_{+}^s}^{VW},L^{-,W}_{1}L^{+,V}_{2}],
\end{array}
\end{equation}
where
$$
L^{\pm,W}_1=L^{\pm,W}\otimes I_V,~~L^{\pm,V}_2=I_W\otimes L^{\pm,V},
$$
and $I_X$ is the unit matrix in $X$.

Moreover, the map $\Delta:{\mathbb{C}}[G^*]\rightarrow {\mathbb{C}}[G^*]\otimes {\mathbb{C}}[G^*]$ dual to the multiplication in $G^*$,
\begin{equation}\label{comultcl}
\Delta(L^{\pm,V}_{ij})=\sum_k L^{\pm,V}_{ik}\otimes L^{\pm,V}_{kj},
\end{equation}
is a homomorphism of Poisson algebras, and the map $S:{\mathbb{C}}[G^*]\rightarrow {\mathbb{C}}[G^*]$,
$$
S(L^{\pm,V}_{ij})=(L^{\pm,V})^{-1}_{ij}
$$
is an antihomomorphism of Poisson algebras.
\end{proposition}
\begin{remark}
Recall that a Poisson--Hopf algebra is a Poisson algebra which is also a Hopf
algebra such that the
comultiplication is a homomorphism of Poisson algebras and the antipode is an antihomomorphism of Poisson algebras. According to Proposition
\ref{pbff}
${\mathbb{C}}[G^*]$ is a Poisson--Hopf algebra.
\end{remark}

Now we construct a quantization of the Poisson--Hopf algebra ${\mathbb{C}}[G^*]$.
For technical reasons we shall need an extension of the algebra $U_\mathcal{A}^{s}({\frak g})$ to an algebra $U_{\mathcal{A}'}^{s}({\frak g})=U_\mathcal{A}^{s}({\frak g})\otimes_\mathcal{A}\mathcal{A}'$, where
$\mathcal{A}'=\mathbb{C}[q^{\frac{1}{2d}}, q^{-\frac{1}{2d}}, \frac{1}{[2]_{q_i}},\ldots ,\frac{1}{[r]_{q_i}}, \frac{1-q^{\frac{1}{2d}}}{1-q_i^{-2}}]_{i=1,\ldots, l}$. Note that the ratios $\frac{1-q^{\frac{1}{2d}}}{1-q_i^{-2}}$ have no singularities when $q=1$, and we can define a localization,
$\mathcal{A}'/(1-q^{\frac{1}{2d}})\mathcal{A}'=\mathbb{C}$ as well as similar localizations for other generic values of $\varepsilon$, $\mathcal{A}'/(\varepsilon^{\frac{1}{2d}}-q^{\frac{1}{2d}})\mathcal{A}'=\mathbb{C}$ and the corresponding localizations of algebras over $\mathcal{A}'$. $U_{\mathcal{A}'}^{s}({\frak g})$ is naturally a Hopf algebra with the comultiplication and the antipode induced from $U_{\mathcal{A}}^{s}({\frak g})$.

First, using arguments similar to those applied in the end of Section \ref{forms} where we defined the action of the element $\mathcal{R}^s$ in tensor products of finite--dimensional representations, one can show that for any finite--dimensional $U_\mathcal{A}^{s}({\frak g})$ module $V$ the invertible elements
${^q{L^{\pm,V}}}$ given by
$$
{^q{L^{+,V}}}=(id\otimes \pi_V){{\mathcal R}_{21}^{s}}^{-1}=(id\otimes
\pi_VS^{s}){\mathcal R}_{21}^{s}
,~~ {^q{L^{-,V}}}=(id\otimes \pi_V){\mathcal R}^{s}.
$$
are well--defined elements of $U_\mathcal{A}^{s}({\frak g})\otimes {\rm End}V$ (compare with \cite{FRT}).
If we fix a basis in $V$, ${^q{L^{\pm,V}}}$ may be regarded as matrices
with matrix
elements $({^q{L^{\pm,V}}})_{ij}$ being elements of $U_\mathcal{A}^{s}({\frak g})$.
We also recall that one can define an operator $R^{VW}=(\pi_V\otimes \pi_W){\mathcal R}^{s}$ (see Section \ref{forms}).

From the Yang--Baxter equation for $\mathcal R$ we get
relations between ${^q{L^{\pm,V}}}$:
\begin{equation}\label{ppcomm}
\begin{array}{l}
R^{VW}{^q{L^{\pm,W}_1}}{^q{L^{\pm,V}_2}}={^q{L^{\pm,V}_2}}{^q{L^{\pm,W}_1}}R^{VW
},
\end{array}
\end{equation}
\begin{equation}\label{pmcomm}
R^{VW}{^q{L^{-,W}_1}}{^q{L^{+,V}_2}}={^q{L^{+,V}_2}}{^q{L^{-,W}_1}}R^{VW}.
\end{equation}
By ${^q{L^{\pm,W}_1}},~{^q{L^{\pm,V}_2}}$ we understand the following matrices
in $V\otimes W$ with entries being elements of $U_\mathcal{A}^{s}({\frak g})$:
$$
{^q{L^{\pm,W}_1}}={^q{L^{\pm,W}}}\otimes I_V,~~{^q{L^{\pm,V}_2}}=I_W\otimes
{^q{L^{\pm,V}}},
$$
where $I_X$ is the unit matrix in $X$.

From (\ref{rmprop}) we can obtain the action of the comultiplication on the
matrices ${^q{L^{\pm,V}}}$:
\begin{equation}\label{comult}
\Delta_s({^q{L^{\pm,V}_{ij}}})=\sum_k {^q{L^{\pm,V}_{ik}}}\otimes
{^q{L^{\pm,V}_{kj}}}
\end{equation}
and the antipode,
\begin{equation}\label{ants}
S_s({^q {L^{\pm,V}_{ij}}})=({^q{L^{\pm,V}}})^{-1}_{ij}.
\end{equation}

We denote by ${\mathbb{C}}_{\mathcal{A}'}[G^*]$ the Hopf subalgebra in $U_{\mathcal{A}'}^{s}({\frak
g})$ generated by matrix elements
of $({^q{L^{\pm,V}}})^{\pm 1}$, where $V$ runs through all finite--dimensional
representations of $U_\mathcal{A}^{s}({\frak
g})$.

Since ${\mathcal R}^{s}=1\otimes 1$ (mod $h$) relations (\ref{ppcomm}) and (\ref{pmcomm}) imply that the quotient algebra ${\mathbb{C}}_{\mathcal{A}'}[G^*]/(q^{\frac{1}{2d}}-1){\mathbb{C}}_{\mathcal{A}'}[G^*]$ is commutative, and one can equip it with a Poisson structure given by
\begin{equation}\label{quasipb}
\{x_1,x_2\}=\frac{1}{2d}{[a_1,a_2] \over q^{\frac{1}{2d}}-1}~(\mbox{mod }(q^{\frac{1}{2d}}-1)),
\end{equation}
where $a_1,a_2\in {\mathbb{C}}_{\mathcal{A}'}[G^*]$ reduce to
$x_1,x_2\in {\mathbb{C}}_{\mathcal{A}'}[G^*]/(q^{\frac{1}{2d}}-1){\mathbb{C}}_{\mathcal{A}'}[G^*]~(\mbox{mod }(q^{\frac{1}{2d}}-1))$.
Obviously, the maps (\ref{comult}) and (\ref{ants}) induce a comultiplication and an antipode on ${\mathbb{C}}_{\mathcal{A}'}[G^*]/(q^{\frac{1}{2d}}-1){\mathbb{C}}_{\mathcal{A}'}[G^*]$ compatible with the introduced Poisson structure, and the quotient ${\mathbb{C}}_{\mathcal{A}'}[G^*]/(q^{\frac{1}{2d}}-1){\mathbb{C}}_{\mathcal{A}'}[G^*]$ becomes a Poisson--Hopf algebra.

\begin{proposition}\label{quantreg}
The Poisson--Hopf algebra ${\mathbb{C}}_{\mathcal{A}'}[G^*]/(q^{\frac{1}{2d}}-1){\mathbb{C}}_{\mathcal{A}'}[G^*]$ is isomorphic to ${\mathbb{C}}[G^*]$ as a Poisson--Hopf
algebra.
\end{proposition}
\begin{proof}
Denote by $p:{\mathbb{C}}_{\mathcal{A}'}[G^*] \rightarrow {\mathbb{C}}_{\mathcal{A}'}[G^*]/(q^{\frac{1}{2d}}-1){\mathbb{C}}_{\mathcal{A}'}[G^*]={\mathbb{C}}[G^*]'$ the
canonical projection, and let ${\tilde
L^{\pm,V}}=(p\otimes p_V)({^q{L^{\pm,V}}})\in
{\mathbb{C}}[G^*]'\otimes {\rm End}V$, where $p_V:V\rightarrow \overline{V}=V/(q^{\frac{1}{2d}}-1)V$ be the projection of finite--dimensional $U_\mathcal{A}^{s}({\frak g})$--module $V$ onto the corresponding $\g$--module $\overline{V}$.

First observe that the map
$$
\imath :{\mathbb{C}}[G^*]'\rightarrow {\mathbb{C}}[G^*],~~(\imath \otimes id){\tilde
L^{\pm,V}}={L^{\pm,\overline{V}}}
$$
is a well--defined linear isomorphism.
Indeed, consider, for instance, element ${\tilde L^{-,V}}$.
From (\ref{rmatrspi}) it follows that
\begin{equation}\label{lv}
\begin{array}{l}
{\tilde L^{-,V}}_{ij}=\{ (p\otimes id)exp\left[ \sum_{i=1}^lhH_i\otimes \pi_{\overline{V}}((-{2s \over 1-s}P_{{\h'}}+P_{{\h'}^\perp})Y_i)\right]\times \\
\prod_{\beta}
exp[p((1-q_\beta^{-2})e_{\beta}) \otimes
\pi_{\overline{V}}(X_{-\beta})]\}_{ij}.
\end{array}
\end{equation}

On the other hand (\ref{fact}) implies that every element $L_-$ may be
represented in the form
\begin{equation}
\begin{array}{l}
L_- = exp\left[ \sum_{i=1}^lb_i({s \over 1-s}P_{{\h'}}-\frac{1}{2}P_{{\h'}^\perp})Y_i\right]\times \\
\prod_{\beta}
exp[b_{\beta}X_{-\beta}],~b_i,b_\beta\in {\Bbb C},
\end{array}
\end{equation}
and hence
\begin{equation}
\begin{array}{l}
L^{-,V}_{ij}=\{ exp\left[ \sum_{i=1}^lb_i\otimes \pi_V(({s \over 1-s}P_{{\h'}}-\frac{1}{2}P_{{\h'}^\perp})Y_i)\right]\times \\
\prod_{\beta}
exp[b_{\beta} \otimes
\pi_V(X_{-\beta})]\}_{ij}.
\end{array}
\end{equation}
Therefore $\imath$ is a linear isomorphism. We have to prove that $\imath$ is an
isomorphism of
Poisson--Hopf algebras.

Recall that ${\mathcal R}^{s}=1\otimes 1 +2hr_+^{s}$ (mod $h^2$). Therefore from
commutation relations (\ref{ppcomm}), (\ref{pmcomm}) it follows that ${\mathbb{C}}[G^*]'$ is a commutative
algebra, and the Poisson brackets of matrix elements ${\tilde L^{\pm,V}}_{ij}$
(see (\ref{quasipb}))
are given by (\ref{pbf}), where $L^{\pm,V}$ are replaced by ${\tilde
L^{\pm,V}}$. The factor $\frac{1}{2d}$ in formula (\ref{quasipb}) normalizes the Poisson bracket in such a way that bracket (\ref{quasipb}) is in agreement with (\ref{pbf}).

From (\ref{comult})
we also obtain that the action of the comultiplication on the matrices ${\tilde
L^{\pm,V}}$ is given by
(\ref{comultcl}), where $L^{\pm,V}$ are replaced by ${\tilde L^{\pm,V}}$.
This completes the proof.
\end{proof}

We shall call the map $p:{\mathbb{C}}_{\mathcal{A}'}[G^*] \rightarrow {\mathbb{C}}[G^*]$ the
quasiclassical limit.

From the definition of the elements ${^q{L^{\pm,V}}}$ it follows that ${\mathbb{C}}_{\mathcal{A}'}[G^*]$ is the subalgebra in $U_{\mathcal{A}'}^{s}({\frak g})$ generated by the elements $\prod_{j=1}^lt_j^{\pm 2dp_{ij}},~\prod_{j=1}^lt_j^{\pm 2dp_{ji}},~i=1,\ldots l,~\tilde e_{\beta}=(1-q_\beta^{-2})e_{\beta},~\tilde f_{\beta}=(1-q_\beta^{-2})e^{h\beta^\vee}f_{\beta},~\beta\in \Delta_+$.

Now using the Hopf algebra ${\mathbb{C}}_{\mathcal{A}'}[G^*]$ we shall define quantum
versions of W--algebras.
From the definition of the elements ${^q{L^{\pm,V}}}$ it follows that ${\mathbb{C}}_{\mathcal{A}'}[G^*]$ contains the subalgebra ${\mathbb{C}}_{\mathcal{A}'}[N_-]$
generated by elements
$\tilde e_{\beta}=(1-q_\beta^{-2})e_{\beta}$, $\beta\in \Delta_+$.

Suppose that the ordering of the root system $\Delta_+$ is fixed as in formula (\ref{NO}).
Denote by ${\mathbb{C}}_{\mathcal{A}'}[M_-]$ the subalgebra in ${\mathbb{C}}_{\mathcal{A}'}[N_-]$
generated by elements
$\tilde e_{\beta}$, $\beta\in {\Delta_{\m_+}}$.

By construction ${\mathbb{C}}_{\mathcal{A}'}[N_-]$ is a quantization of the algebra of regular functions on the algebraic subgroup $N_-\subset G^*$ corresponding to the Lie subalgebra $\n_- \subset \g^*$, and ${\mathbb{C}}_{\mathcal{A}'}[M_-]$ is a quantization of the algebra of regular functions on the algebraic subgroup $M_-\subset G^*$ corresponding to the Lie subalgebra $\m_- \subset \g^*$ in the sense that $p({\mathbb{C}}_{\mathcal{A}'}[N_-])={\mathbb{C}}[N_-]$ and $p({\mathbb{C}}_{\mathcal{A}'}[M_-])={\mathbb{C}}[M_-]$. We also denote by $M_+$ the algebraic subgroup $M_+\subset G^*$ corresponding to the Lie subalgebra $\m_+ \subset \g^*$.

We claim that the defining relations in the subalgebra ${\mathbb{C}}_{\mathcal{A}'}[M_-]$ are given by a formula similar to (\ref{erel1}).
Indeed, consider the defining relations in the subalgebra $U_\mathcal{A}^{s}({\frak m}_+)$,
$$
e_{\alpha}e_{\beta} - q^{(\alpha,\beta)+({1+s \over 1-s}P_{{\h'}^*}\alpha,\beta)}e_{\beta}e_{\alpha}= \sum_{\alpha<\delta_1<\ldots<\delta_n<\beta}C'(k_1,\ldots,k_n)
e_{\delta_1}^{k_1}e_{\delta_2}^{k_2}\ldots e_{\delta_n}^{k_n},~\alpha<\beta,
$$
where $C'(k_1,\ldots,k_n)\in \mathcal{A}$. Commutation relations between quantum analogues of root vectors obtained in Proposition 4.2 in \cite{kh-t} imply that each function $C'(k_1,\ldots,k_n)$ has a zero of order $k_1+\ldots+k_n-1$ at point $q=1$. Therefore one can write the following defining relations for the generators $\tilde{e}_\beta=(1-q_\beta^{-2})e_\beta$, $\beta\in \Delta_{\m_+}$ in the algebra ${\mathbb{C}}_{\mathcal{A}'}[M_-]$,
$$
\tilde{e}_{\alpha}\tilde{e}_{\beta} - q^{(\alpha,\beta)+({1+s \over 1-s}P_{{\h'}^*}\alpha,\beta)}\tilde{e}_{\beta}\tilde{e}_{\alpha}= \sum_{\alpha<\delta_1<\ldots<\delta_n<\beta}C''(k_1,\ldots,k_n)
{\tilde{e}_{\delta_1}}^{k_1}{\tilde{e}_{\delta_2}}^{k_2}\ldots {\tilde{e}_{\delta_n}}^{k_n},~\alpha<\beta,
$$
where $C''(k_1,\ldots,k_n)\in \mathcal{A}'$, and each function $C''(k_1,\ldots,k_n)$ has a zero of order 1 at point $q=1$.

Now arguments similar to those used in the proof of Theorem \ref{qnil} show that the map $\chi_q^{s}:{\mathbb{C}}_{\mathcal{A}'}[M_-]\rightarrow \mathcal{A}'$,
\begin{equation}\label{charq}
\chi_q^s(\tilde e_\beta)=\left\{ \begin{array}{ll} 0 & \beta \not \in \{\gamma_1, \ldots, \gamma_{l'}\} \\ k_i & \beta=\gamma_i, k_i\in \mathcal{A}'
\end{array}
\right  .,
\end{equation}
is a character of ${\mathbb{C}}_{\mathcal{A}'}[M_-]$. Denote by $\mathbb{C}_{\chi_q^{s}}$ the rank one representation of the algebra ${\mathbb{C}}_{\mathcal{A}'}[M_-]$ defined by the character $\chi_q^{s}$.

We call the algebra
\begin{equation}\label{QW}
W_q^s(G)={\rm End}_{{\mathbb{C}}_{\mathcal{A}'}[G^*]}({\mathbb{C}}_{\mathcal{A}'}[G^*]\otimes_{{\mathbb{C}}_{\mathcal{A}'}[M_-]}\mathbb{C}_{\chi_q^{s}})^{opp}
\end{equation}

the q-W algebra associated to (the conjugacy class of) the Weyl group element $s\in W$.

Observe that by Frobenius reciprocity we also have $$W_q^s(G)={\rm Hom}_{{\mathbb{C}}_{\mathcal{A}'}[M_-]}(\mathbb{C}_{\chi_q^{s}},{\mathbb{C}}_{\mathcal{A}'}[G^*]\otimes_{{\mathbb{C}}_{\mathcal{A}'}[M_-]}\mathbb{C}_{\chi_q^{s}}).$$
Therefore if we denote by $I_q$ the left ideal in ${\mathbb{C}}_{\mathcal{A}'}[G^*]$ generated by the kernel of $\chi_q^s$ then $W_q^s(G)$ can be defined as the subspace of all $x+I_q\in Q_{\chi_q^s}$, $Q_{\chi_q^s}={\mathbb{C}}_{\mathcal{A}'}[G^*]\otimes_{{\mathbb{C}}_{\mathcal{A}'}[M_-]}\mathbb{C}_{\chi_q^{s}}={\mathbb{C}}_{\mathcal{A}'}[G^*]/I_q$, such that $[m,x]=mx-xm\in I_q$ for any $m\in {\mathbb{C}}_{\mathcal{A}'}[M_-]$.

Now consider the Lie algebra $\mathfrak{L}_{\mathcal{A}'}$ associated to the associative algebra ${\mathbb{C}}_{\mathcal{A}'}[M_-]$, i.e. $\mathfrak{L}_{\mathcal{A}'}$ is the Lie algebra which is isomorphic to ${\mathbb{C}}_{\mathcal{A}'}[M_-]$ as a linear space, and the Lie bracket in $\mathfrak{L}_{\mathcal{A}'}$ is given by the usual commutator of elements in ${\mathbb{C}}_{\mathcal{A}'}[M_-]$.

If we denote by $\rho_{\chi^{s}_q}$ the canonical projection ${\mathbb{C}}_{\mathcal{A}'}[G^*]\rightarrow {\mathbb{C}}_{\mathcal{A}'}[G^*]/I_q$ then
the algebra $W_q^s(G)$ can be regarded as the algebra of invariants with respect to the following action of the Lie algebra $\mathfrak{L}_{\mathcal{A}'}$ on the space ${\mathbb{C}}_{\mathcal{A}'}[G^*]/I_q$:
\begin{equation}\label{qmainactcl}
m\cdot (x+I_q) =\rho_{\chi^{s}_q}([m,x] ).
\end{equation}
where $x\in {\mathbb{C}}_{\mathcal{A}'}[G^*]$ is any representative of $x+I_q\in {\mathbb{C}}_{\mathcal{A}'}[G^*]/I_q$ and $m\in {\mathbb{C}}_{\mathcal{A}'}[M_-]$.

In terms of this description the multiplication in $W_s^q(G)$ takes the form $(x+I_q)(y+I_q)=xy+I_q$, $x+I_q,y+I_q\in W_q^s(G)$.
Note also that since ${\chi_q^s}$ is a character of ${\mathbb{C}}_{\mathcal{A}'}[M_-]$ the ideal $I_q$ is stable under that action of ${\mathbb{C}}_{\mathcal{A}'}[M_-]$ on ${\mathbb{C}}_{\mathcal{A}'}[G^*]$ by commutators.

In conclusion we remark that by specializing $q$ to a particular value $\varepsilon\in \mathbb{C}$ such that $[r]_{\varepsilon_i}!\neq 0$, $\varepsilon^{d_i}\neq 0$, $i=1,\ldots ,l$, one can define a complex associative algebra
${\mathbb{C}}_\varepsilon[G^*]={\mathbb{C}}_{\mathcal{A}'}[G^*]/(q^{\frac{1}{2d}}-\varepsilon^{\frac{1}{2d}})
{\mathbb{C}}_{\mathcal{A}'}[G^*]$, its subalgebra ${\mathbb{C}}_\varepsilon[M_-]$ with a nontrivial character $\chi_\varepsilon^{s}$ and the corresponding W--algebra
\begin{equation}\label{eW}
W_\varepsilon^s(G)={\rm End}_{{\mathbb{C}}_\varepsilon[G^*]}({\mathbb{C}}_\varepsilon[G^*]\otimes_{{\mathbb{C}}_\varepsilon[M_-]}\mathbb{C}_{\chi_\varepsilon^{s}})^{opp}
\end{equation}

Obviously, for generic $\varepsilon$ we have $W_\varepsilon^s(G)=W_q^s(G)/(q^{\frac{1}{2d}}-\varepsilon^{\frac{1}{2d}})W_q^s(G)$.


\section{Poisson reduction and q-W algebras}\label{wpsred}

\setcounter{equation}{0}
\setcounter{theorem}{0}

In this section we shall analyze the
quasiclassical limit of
the algebra $W_q^s(G)$. Using results of Section
\ref{poisred}
we realize this limit algebra as the algebra of functions on a reduced
Poisson manifold.

Denote by $\chi^{s}$ the character of the Poisson subalgebra ${\mathbb{C}}[M_-]$ such that
$\chi^{s}(p(x))=\chi_q^{s}(x)~(\mbox{mod }~(q^{\frac{1}{2d}}-1))$ for every $x\in {\mathbb{C}}_{\mathcal{A}'}[M_-]$.

Let $I=p(I_q)$ be the ideal in ${\mathbb{C}}[G^*]$ generated by the kernel of $\chi^{s}$.
Then the Poisson algebra $W^s(G)=W^s_q(G)/(q^{\frac{1}{2d}}-1)W^s_q(G)$ is the subspace of all $x+I\in Q_{\chi^s}$, $Q_{\chi^s}={\mathbb{C}}[G^*]/I$, such that $\{m,x\}\in I$ for any $m\in {\mathbb{C}}[M_-]$, and the Poisson bracket in $W^s(G)$ takes the form $\{(x+I),(y+I)\}=\{x,y\}+I$, $x+I,y+I\in W^s(G)$. We shall also write $W^s(G)=({\mathbb{C}}[G^*]/I)^{{\mathbb{C}}[M_-]}=(Q_{\chi^s})^{{\mathbb{C}}[M_-]}$.

Denote by $\rho_{\chi^{s}}$ the canonical projection ${\mathbb{C}}[G^*]\rightarrow {\mathbb{C}}[G^*]/I$.
The discussion above implies that $W^s(G)$ is the algebra of invariants with respect to the following action of the Poisson algebra ${{\mathbb{C}}[M_-]}$ on the space ${\mathbb{C}}[G^*]/I$:
\begin{equation}\label{mainactcl}
x\cdot (v+I) =\rho_{\chi^{s}}(\{x,v\} ),
\end{equation}
where $v\in {\mathbb{C}}[G^*]$ is any representative of $v+I\in {\mathbb{C}}[G^*]/I$ and $x\in {\mathbb{C}}[M_-]$.

We shall describe the space of invariants $({\mathbb{C}}[G^*]/I)^{{\mathbb{C}}[M_-]}$  with respect to this action by analyzing ``dual geometric objects''. First observe that algebra $({\mathbb{C}}[G^*]/I)^{{\mathbb{C}}[M_-]}$ is a particular
example of the
reduced Poisson algebra introduced in Lemma \ref{redspace}.

Indeed, recall that according to (\ref{fact}) any element $(L_+,L_-)\in G^*$ may be uniquely written as
\begin{equation}\label{fact1}
(L_+,L_-)=(h_+,h_-)(n_+,n_-),
\end{equation}
where $n_\pm \in N_\pm$, $h_+=exp(({1 \over 1-s}P_{{\h'}}+\frac{1}{2}P_{{\h'}^\perp})x),~h_-=exp(({s \over 1-s}P_{{\h'}}-\frac{1}{2}P_{{\h'}^\perp})x),~x\in
{\frak h}$.

Formula (\ref{fact}) and decomposition of $N_-$ into products of one--dimensional subgroups corresponding to roots also imply that every element $L_-$ may be
represented in the form
\begin{equation}\label{lm}
\begin{array}{l}
L_- = exp\left[ \sum_{i=1}^lb_i({s \over 1-s}P_{{\h'}}-\frac{1}{2}P_{{\h'}^\perp})H_i\right]\times \\
\prod_{\beta}
exp[b_{\beta}X_{-\beta}],~b_i,b_\beta\in {\Bbb C},
\end{array}
\end{equation}
where the product over roots is taken in the same order as in normal ordering (\ref{NO}).

Now define a map $\mu_{M_+}:G^* \rightarrow M_-$ by
\begin{equation}\label{mun}
\mu_{M_+}(L_+,L_-)=m_-,
\end{equation}
where for $L_-$ given by (\ref{lm}) $m_-$ is defined as follows
$$
m_-=\prod_{\beta\in \Delta_{\m_+}}
exp[b_{\beta}X_{-\beta}],
$$
and the product over roots is taken in the same order as in the normally ordered segment $\Delta_{\m_+}$.

By definition $\mu_{M_+}$ is a morphism of algebraic
varieties.
We also note that by definition ${\mathbb{C}}[M_-]=\{ \varphi\in {\mathbb{C}}[G^*]:\varphi=
\varphi(m_-)\}$. Therefore ${\mathbb{C}}[M_-]$ is generated by the pullbacks of
regular functions on $M_-$ with respect to the map $\mu_{M_+}$.
Since ${\mathbb{C}}[M_-]$ is a Poisson subalgebra in ${\mathbb{C}}[G^*]$, and  regular
functions
on $M_-$ are dense in $C^\infty(M_-)$ on every compact subset, we can equip the
manifold $M_-$ with
the Poisson structure in such a way that $\mu_{M_+}$ becomes a Poisson mapping.

Let $u$ be the element defined by
\begin{equation}\label{defu}
u=\prod_{i=1}^{l'}exp[t_{i} X_{-\gamma_i}]~ \in M_-,t_{i}=k_{i}~({\rm mod}~(q^{\frac{1}{2d}}-1)),
\end{equation}
where the product over roots is taken in the same order as in the normally ordered segment $\Delta_{\m_+}$.
From (\ref{lv}) and the definition of $\chi^s$ it follows that $\chi^s(\varphi)=\varphi (u)$
for every $\varphi \in {\mathbb{C}}[M_-]$. $\chi^s$ naturally extends to a character
of the
Poisson algebra $C^\infty(M_-)$.

Now applying Lemma \ref{redspace} for $M=G^*,~B=M_-,~\pi=\mu_{M_+},~b=u$ we can
define the
reduced Poisson algebra $C^\infty(\mu_{M_+}^{-1}(u))^{C^\infty(M_-)}$ (see also
Remark \ref{redpoisalg}).
Denote by $I_u$ the ideal in $C^\infty(G^*)$ generated by elements
$\mu_{M_+}^*\psi,~\psi \in C^\infty(M_-),
~\psi(u)=0$. Let $P_u:C^\infty(G^*)\rightarrow
C^\infty(G^*)/I_u=C^\infty(\mu_{M_+}^{-1}(u))$ be the
canonical projection. Then the action (\ref{redact}) of $C^\infty(M_-)$ on
$C^\infty(\mu_{N_+}^{-1}(u))$
takes the form:
\begin{equation}\label{actred}
\psi\cdot \varphi=P_u(\{ \mu_{M_+}^*\psi, \tilde \varphi\}),
\end{equation}
where $\psi \in C^\infty(M_-),~\varphi \in C^\infty(\mu_{M_+}^{-1}(u))$ and
$\tilde \varphi \in C^\infty(G^*)$
is a representative of $\varphi$ such that $P_u\tilde \varphi=\varphi$.
\begin{lemma}\label{redreg}
$\mu_{M_+}^{-1}(u)$ is a subvariety in $G^*$. Moreover, the algebra ${\mathbb{C}}[G^*]/I$ is isomorphic to the algebra of regular functions on $\mu_{M_+}^{-1}(u)$, ${\mathbb{C}}[G^*]/I={\mathbb{C}}[\mu_{M_+}^{-1}(u)]$, and
the algebra
$W^s(G)=({\mathbb{C}}[G^*]/I)^{{\mathbb{C}}[M_-]}$ is isomorphic to the algebra of regular
functions on
$\mu_{M_+}^{-1}(u)$ which are invariant with respect to the action
(\ref{actred}) of
$C^\infty(M_-)$ on $C^\infty(\mu_{M_+}^{-1}(u))$, i.e.
$$
W^s(G)={\mathbb{C}}[\mu_{M_+}^{-1}(u)]^{{\mathbb{C}}[M_-]}={\mathbb{C}}[\mu_{M_+}^{-1}(u)]\cap
C^\infty(\mu_{M_+}^{-1}(u))^{C^\infty(M_-)}.
$$
\end{lemma}
\begin{proof}
By definition $\mu_{M_+}^{-1}(u)$ is a subvariety in $G^*$. Next observe that
$I=
{\mathbb{C}}[G^*]\cap I_u$. Therefore the algebra ${\mathbb{C}}[G^*]/I$ is
identified with the algebra of regular functions on $\mu_{M_+}^{-1}(u)$.

Since ${\mathbb{C}}[M_-]$ is dense in $C^\infty(M_-)$ on every compact subset in
$M_-$ we have:
$$
C^\infty(\mu_{M_+}^{-1}(u))^{C^\infty(M_-)}\cong
C^\infty(\mu_{M_+}^{-1}(u))^{{\mathbb{C}}[M_-]}.
$$

Finally observe that action (\ref{actred}) coincides with action
(\ref{mainactcl}) when restricted to
regular functions.
\end{proof}

In case when the roots $\gamma_1,\ldots, \gamma_n$ are simple or the set $\{\gamma_1,\ldots, \gamma_n\}$ is empty, and hence the segment $\Delta_{s^1}$ is of the form $\Delta_{s^1}=\{\gamma_1,\ldots, \gamma_n\}$, we shall realize the
algebra $C^\infty(\mu_{M_+}^{-1}(u))^{C^\infty(M_-)}$ as the algebra of
functions on a reduced
Poisson manifold. In the spirit of Lemma \ref{redspace} we shall construct a map
that forms
a dual pair together with the mapping $\mu_{M_+}$. In this construction we use
the dressing
action of the Poisson--Lie group $G$ on $G^*$ (see Proposition
\ref{dressingact}).

Consider the restriction of the dressing action $G^*\times G \rightarrow G^*$ to
the subgroup $M_+\subset G$.
According to part (iv) of Proposition \ref{bpm} $({\frak b}_+,{\frak b}_-)$ is
a subbialgebra of
$({\frak g},{\frak g}^*)$. Therefore $B_+$ is a Poisson--Lie subgroup in $G$. We claim that $M_+\subset B_+$ is an admissible subgroup.

Indeed, observe that if the roots $\gamma_1,\ldots, \gamma_n$ are simple or the set $\{\gamma_1,\ldots, \gamma_n\}$ is empty then $\Delta_{s^1}=\{\gamma_1,\ldots, \gamma_n\}$, and the complementary subset to $\Delta_{\m_+}$ in $\Delta_+$ is a minimal segment $\Delta_{\m_+}^0$ with respect to normal ordering (\ref{NO}). Now using Proposition \ref{bpm} (iv) the subspace $\m_+^\perp$ in $\b_-$ can be identified with the linear subspace in $\b_-$ spanned by the Cartan subalgebra $\h$ and the root subspaces corresponding to the roots from the minimal segment $-\Delta_{\m_+}^0$. Using the fact that the adjoint action of $\h$ normalizes root subspaces and Lemma \ref{minsegm} we deduce that $\m_+^\perp\subset \b_-$ is a Lie subalgebra, and hence $M_+\subset B_+$ is an admissible subgroup.
Therefore $C^\infty (G^*)^{M_+}$ is a Poisson subalgebra in the Poisson algebra
$C^\infty (G^*)$.
\begin{proposition}\label{centralz}
Assume that the roots $\gamma_1,\ldots, \gamma_n$ are simple or the set $\{\gamma_1,\ldots, \gamma_n\}$ is empty.
Then the algebra $C^\infty (G^*)^{M_+}$ is the centralizer of
$\mu_{M_+}^*\left( C^\infty \left( M_-\right)\right)$ in the Poisson algebra
$C^\infty (G^*)$.
\end{proposition}
\begin{proof}
First recall that, as we observed above, $B_+$ is a Poisson--Lie subgroup in $G$.
By Proposition
\ref{dressingact} for $X \in {\frak b}_+$ we have:
\begin{equation}
L_{\widehat X} \varphi(L_+,L_-) =( \theta_{G^*}(L_+,L_-) , X ) (\xi_\varphi )=
(r_-^{-1}\mu_{B_+}^*(\theta_{B_-}) , X) (\xi_\varphi ),
\end{equation}
where $\widehat X$ is the corresponding vector field on $G^*$,
$\xi_\varphi $ is the Hamiltonian vector field of $\varphi \in C^\infty (G^*)$,
and the map
$\mu_{B_+}:G^*\rightarrow B_-$ is defined by $\mu_{B_+}(L_+,L_-)=L_-$.
Now from Proposition \ref{bpm} (iv) and the definition of the moment map it
follows that
$\mu_{B_+}$ is a moment map for the dressing action of the subgroup $B_+$ on
$G^*$.

We also proved above that $M_+$ is an admissible subgroup
in the Lie--Poisson group $B_+$. Moreover the dual group $B_-$ can be uniquely factorized as $B_-=M_+^\perp M_-$, where $M_+^\perp\subset B_-$ is the Lie subgroup
corresponding to the Lie subalgebra ${\frak m}_+^\perp \subset\b_-$.

We conclude that all the conditions of Proposition \ref{QPmoment}
are satisfied with $A=B_+ , K=M_+ , A^*=B_-,
T=M_- , K^\perp = M_+^\perp, \mu =\mu_{B_+}$.
It follows that the algebra $C^\infty (G^*)^{M_+}$ is the centralizer of
$\mu_{M_+}^*\left( C^\infty \left( M_-\right)\right)$ in the Poisson algebra
$C^\infty (G^*)$.
This completes the proof.
\end{proof}

Let $G^*/M_+$ be the quotient of $G^*$ with respect to the dressing action of
$M_+$,
$\pi:G^* \rightarrow G^*/M_+$
the canonical projection. Note that the space $G^*/M_+$ is not a smooth
manifold. However,
in the next section we will see that the subspace $\pi(\mu_{M_+}^{-1}(u))\subset
G^*/M_+$ is
a smooth manifold. Therefore by Remark \ref{remred} the algebra
$C^\infty(\pi(\mu_{M_+}^{-1}(u)))$
is isomorphic to $C^\infty(\mu_{M_+}^{-1}(u))^{C^\infty(M_-)}$.


\section{Cross--section theorem}\label{crossect}

\setcounter{equation}{0}
\setcounter{theorem}{0}

In this section we prove that the reduced space $\pi(\mu_{M_+}^{-1}(u))\subset
G^*/M_+$ is smooth.

Using the map $q:G^* \rightarrow G$ one can
reduce the study of the dressing action to the study of the action of $G$ on
itself by conjugations.
This simplifies many geometric problems. Consider the restriction
of this action to
the subgroup $M_+$. Denote by $\pi_q:G\rightarrow G/M_+$
the canonical projection onto the quotient with respect to this action. Then the reduced
space $\pi(\mu_{M_+}^{-1}(u))$ is closely related to the subspace $\pi_q(q(\mu_{M_+}^{-1}(u)))$
in $G/M_+$ which has a geometric description.

Let $X_\alpha(t)=\exp(tX_\alpha)\in G$, $t\in \mathbb{C}$ be the one--parametric subgroup in the algebraic group $G$ corresponding to root $\alpha\in \Delta$. Recall that for any $\alpha \in \Delta_+$ and any $t\neq 0$ the element  $s_\alpha(t)=X_{-\alpha}(t)X_{\alpha}(-t^{-1})X_{-\alpha}(t)\in G$ is a representative for the reflection $s_\alpha$ corresponding to the root $\alpha$. Denote by $s\in G$ the following representative of the Weyl group element $s\in W$,
\begin{equation}\label{defrep}
s=s_{\gamma_1}(t_1)\ldots s_{\gamma_{l'}}(t_{l'}),
\end{equation}
where the numbers $t_{i}$ are defined in (\ref{defu}), and we assume that $t_i\neq 0$ for any $i$.

We shall also use the following representatives for $s^1$ and $s^2$
$$
s^1=s_{\gamma_1}(t_1)\ldots s_{\gamma_{n}}(t_{n}),~s^2=s_{\gamma_{n+1}}(t_{n+1})\ldots s_{\gamma_{l'}}(t_{l'}).
$$

Let $Z$ be the subgroup of $G$ generaled by the semi--simple part of the Levi factor $L$ and by the centralizer of $s$ in $H$. Denote by $N$ the subgroup of $G$ corresponding to the Lie subalgebra $\n$ and by $\overline{N}$ the opposite unipotent subgroup in $G$ with the Lie algebra $\overline{\n}=\bigoplus_{m<0}(\g)_m$. By definition we have that $N_+\subset ZN$.

\begin{proposition}\label{crosssect}
{\bf (\cite{S6}, Propositions 2.1 and 2.2)}
Let $N_s=\{ v \in N|svs^{-1}\in \overline{N} \}$.
Then the conjugation map
\begin{equation}\label{cross}
N\times sZN_s\rightarrow NsZN
\end{equation}
is an isomorphism of varieties. Moreover, the variety $sZN_s$ is a transversal slice to the set of conjugacy classes in $G$.
\end{proposition}

\begin{proposition}\label{constrt}
Let $q:G^*\rightarrow G$ be the map introduced in Proposition \ref{dressingact},
$$
q(L_+,L_-)=L_-L_+^{-1}.
$$
Suppose that the numbers $t_{i}$ defined in (\ref{defu}) are not equal to zero for all $i$. Then $q(\mu_{M_+}^{-1}(u))$ is a subvariety in $NsZN$ and the closure $\overline{q(\mu_{M_+}^{-1}(u))}$ of $q(\mu_{M_+}^{-1}(u))$ with respect to Zariski topology is also contained in $NsZN$.
\end{proposition}
\begin{proof}
Let $\Delta_{\m_+}^1=\{\alpha\in \Delta_+: \alpha< \gamma_1\}$, $\Delta_{\m_+}^2=\{\alpha\in \Delta_+: \alpha> \gamma_{l'}\}$.
Using definition (\ref{mun}) of the map $\mu_{M_+}$  we can
describe the space
$\mu_{M_+}^{-1}(u)$ as
follows:
\begin{equation}\label{mun1}
\mu_{M_+}^{-1}(u)=\{(h_+n_+,h_-x_1ux_2) | n_+ \in N_+ , h_\pm=e^{r_\pm^s x}, x \in \h, x_1\in M_-^1, x_2\in M_-^2 \},
\end{equation}
where $M_-^{1,2}$ is the subgroup of $G$ generated by the one--parametric subgroups corresponding to the roots from the segment $-\Delta_{\m_+}^{1,2}$. Therefore
\begin{equation}\label{dva}
q(\mu_{M_+}^{-1}(u))=
\{ h_-x_1ux_2n_+^{-1}h_+^{-1}| n_+ \in N_+ , h_\pm=e^{r_\pm^s x}, x \in \h, x_{1,2}\in M_-^{1,2} \}.
\end{equation}

First we show that $x_1ux_2n_+^{-1}$ belongs to $NsZN$. Fix the circular normal ordering on $\Delta$ associated to normal ordering (\ref{NO}) of $\Delta_+$. Observe that the segment which consists of $\alpha\in \Delta$ such that $\gamma_{1}\leq \alpha< -\gamma_{1}$ is minimal with respect to the circular normal ordering, and its intersection with $\Delta_-$ is $-\Delta_{\m_+}^1$. Therefore by Lemma \ref{minsegm} we have

$$
x_1X_{-\gamma_1}(t_1)\ldots X_{-\gamma_{n}}(t_{n})=x_1X_{\gamma_1}(t_1^{-1})\ldots X_{\gamma_{n}}(t_{n}^{-1})X_{\gamma_n}(-t_n^{-1})\ldots X_{\gamma_{1}}(-t_{1}^{-1})X_{-\gamma_1}(t_1)\ldots X_{-\gamma_{n}}(t_{n})=
$$
$$
=n_1x_1'X_{\gamma_n}(-t_n^{-1})\ldots X_{\gamma_{1}}(-t_{1}^{-1})X_{-\gamma_1}(t_1)\ldots X_{-\gamma_{n}}(t_{n}),~~n_1\in N_+,x_1'\in M_-^1.
$$

Using the relations $X_{\gamma_{1}}(-t_{1}^{-1})X_{-\gamma_1}(t_1)=X_{-\gamma_1}(-t_1){s}_{\gamma_1}$ one can rewrite the last identity as follows
$$
x_1X_{-\gamma_1}(t_1)\ldots X_{-\gamma_{n}}(t_{n})=n_1x_1'X_{\gamma_n}(-t_n^{-1})\ldots X_{\gamma_{2}}(-t_{2}^{-1})X_{-\gamma_1}(-t_1){s}_{\gamma_1}X_{-\gamma_2}(t_2)\ldots X_{-\gamma_{n}}(t_{n}).
$$

Now we move $X_{-\gamma_1}(-t_1)$ to the left from the product $X_{\gamma_n}(-t_n^{-1})\ldots X_{\gamma_{2}}(-t_{2}^{-1})$ in the last formula.
Since the segment which consists of $\alpha\in \Delta$ such that $\gamma_{2}\leq \alpha\leq -\gamma_{1}$ is minimal with respect to the circular normal ordering, and its intersection with $\Delta_-$ is $\Delta_-^1=\{\alpha\in \Delta_-: \alpha\leq -\gamma_1\}$, one has by Lemma \ref{minsegm}, using commutation relations between one--parametric subgroups corresponding to roots and the orthogonality of roots $\gamma_1$ and $\gamma_2$
$$
x_1X_{-\gamma_1}(t_1)\ldots X_{-\gamma_{n}}(t_{n})=n_2x_1''X_{\gamma_n}(-t_n^{-1})\ldots X_{\gamma_{3}}(-t_{3}^{-1}){s}_{\gamma_1}X_{\gamma_{2}}(-t_{2}^{-1})X_{-\gamma_2}(t_2)\ldots X_{-\gamma_{n}}(t_{n}),
$$
where $n_2\in N_+,x_1''\in M^1$, and $M^1$ is the subgroup of $G$ generated by the one--parametric subgroups corresponding to roots from $\Delta_-^1$.

Now we can use the relation $X_{\gamma_{2}}(-t_{2}^{-1})X_{-\gamma_2}(t_2)=X_{-\gamma_1}(-t_2){s}_{\gamma_2}$ and apply similar arguments to get
$$
x_1X_{-\gamma_1}(t_1)\ldots X_{-\gamma_{n}}(t_{n})=n_3x_1'''X_{\gamma_n}(-t_n^{-1})\ldots X_{\gamma_{4}}(-t_{4}^{-1}){s}_{\gamma_1}{s}_{\gamma_2}X_{\gamma_{3}}(-t_{3}^{-1})X_{-\gamma_3}(t_3)\ldots X_{-\gamma_{n}}(t_{n}),
$$
where $n_3\in N_+,x_1'''\in M^2$, and $M^2$ is the subgroup of $G$ generated by the one--parametric subgroups corresponding to roots from $\Delta_-^2=\{\alpha\in \Delta_-: \alpha\leq -\gamma_2\}$.

We can proceed in a similar way to obtain the following representation
\begin{equation}\label{1?}
x_1X_{-\gamma_1}(t_1)\ldots X_{-\gamma_{n}}(t_{n})=n\widetilde{x}{s}_{\gamma_1}\ldots{s}_{\gamma_n},~~n\in N_+,\widetilde{x}\in M^n,
\end{equation}
where $M^n$ is the subgroup of $G$ generated by the one--parametric subgroups corresponding to roots from $\Delta_-^n=\{\alpha\in \Delta_-: \alpha\leq -\gamma_n\}$.

By the definition of normal ordering (\ref{NO}) one also has ${s}_{\gamma_n}^{-1}\ldots {s}_{\gamma_1}^{-1}M^n{s}_{\gamma_1}\ldots {s}_{\gamma_n}\subset N$, and hence (\ref{1?}) can be rewritten in the following form
\begin{equation}\label{2?}
x_1X_{-\gamma_1}(t_1)\ldots X_{-\gamma_{n}}(t_{n})=n{s}_{\gamma_1}\ldots{s}_{\gamma_n}n'=ns^1n',~~n\in N_+, n'\in N.
\end{equation}

Similarly, taking into account that $Z$ normalizes $N$, one has
\begin{equation}\label{3?}
X_{-\gamma_{n+1}}(t_{n+1})\ldots X_{-\gamma_{l'}}(t_{l'})x_2=n''{s}_{\gamma_{n+1}}\ldots{s}_{\gamma_{l'}}n'=n''s^2n''',~~n''\in N,n'''\in Z_-Z_+N,
\end{equation}
where $Z_-=Z\bigcap N_-$, $Z_+=Z\bigcap N_+$.

Combining (\ref{2?}) and (\ref{3?}) we obtain
\begin{equation}\label{6?}
x_1ux_2n_+^{-1}=n{s}^1g{s}^2k,g\in N,k\in Z_-Z_+N.
\end{equation}

Let $M_+^{1,2}$ be the subgroups of $G$ generated by the one--parametric subgroups corresponding to the roots from the segments $\Delta_{s^1}$ and $\Delta_{s^2}$, respectively, $M_+'$ the subgroup of $G$ generated by the one--parametric subgroups corresponding to the roots from the segment $\Delta_+ \setminus (\Delta_{s^1}\bigcup \Delta_{s^2}\bigcup(\overline{\Delta}_{0})_+)$. By the definition of these subgroups every element $g\in N$ has a unique factorization $g=g_2g'g_1$, where $g_{1,2}\in M_+^{1,2}$, and $g'\in M_+'$. Applying this factorization to the element $g$ in (\ref{6?}) and recalling the properties of normal ordering (\ref{NO}), we have $({s}^2)^{-1}M_+^1{s}^2\subset N$,
${s}^1M_+^2({s}^1)^{-1}\subset N$, $({s}^2)^{-1}M_+'{s}^2\subset N$. Now we derive from (\ref{6?}) that
\begin{equation}\label{8?}
x_1ux_2n_+^{-1}=\widehat{n}{s}^1{s}^2k'=\widehat{n}sk',\widehat{n}\in N_+,k'\in Z_-Z_+N.
\end{equation}

Finally factorizing $\widehat{n}$ as $\widehat{n}=n_s\widetilde{n}$, where $n_s\in N_s'=\{ v \in N_+|s^{-1}vs\in N_- \}\subset N$ and $\widetilde{n}\in \widetilde{N}=\{ v \in N_+|s^{-1}vs\in N_+ \}$ we arrive at
\begin{equation}\label{7?}
x_1ux_2n_+^{-1}=n_ssk'',n_s\in N_s',k''\in ZN.
\end{equation}
Hence $x_1ux_2n_+^{-1}\in NsZN$.

Let $H'\subset H$ be the subgroup corresponding to the Lie subalgebra $\h'\subset \h$, and $H_0\subset H$ the subgroup corresponding to the orthogonal complement $\h_0$ of $\h'$ in $\h$ with respect to the Killing form. Note that $\h_0$ is the space of fixed points for the action of $s$ on $\h$. We obviously have $H=H'H_0$ (direct product of subgroups). From the definition of $r_\pm^s$ it follows that for any $h_0\in H_0$ and $h'\in H'$ elements $h_+=h_0h'$ and $h_-=h_0^{-1}s(h')$ are of the form  $h_\pm=e^{r_\pm^s x}$ for some $x\in \h$ and all elements  $h_\pm=e^{r_\pm^s x}, x\in \h$ are obtained in this way.

Next observe that the space $NsZN$ is invariant with respect to
the following action of $H$:
\begin{equation}\label{tri}
h\circ L= h_-Lh_+^{-1}, h= h_+=h_0h', h_-=h_0^{-1}s(h').
\end{equation}

Indeed, let $L=vszw,~ v,w \in N,z\in Z$ be an element of $NsZN$. Then
\begin{equation}\label{hact}
h\circ L=h_-vh_-^{-1}h_-sh_+^{-1}h_+zwh_+^{-1}=h_-vh_-^{-1}sh_0^{-2}h_+zwh_+^{-1}
\end{equation}
Since $s^{-1}h_-s=h_0^{-1}h'$. The r.h.s. of the last equality belongs to $NsZN$ because $H$ normalizes
$N$ and $Z$.

Comparing action (\ref{tri}) with (\ref{dva}) and recalling that $x_1ux_2n_+^{-1}\in NsZN$ we deduce
$q(\mu_{M_+}^{-1}(u)) \subset NsZN$.

The variety $\lambda(\mu_{M_+}^{-1}(u))$ is not closed in $G$. But following Corollary 2.5 and Proposition 2.10 in \cite{FZ} we shall show that $NsZN$ is closed in $G$.

Observe that an element $g\in G$ belongs to $NsZN=N_s'sZN$, $N_s'=\{ v \in N_+|s^{-1}vs\in N_- \}\subset N$, if and only if $s^{-1}g\in s^{-1}N_s'sZN$. The variety $s^{-1}N_s'sZN$ is a subvariety of $\overline{N}ZN$. First we prove that $\overline{N}ZN$ is closed in $G$. 

Let $V_{\gamma_i}$, $i=1,\ldots ,l'$ be the irreducible finite-dimensional representation of $G$ with highest weight $\gamma_i$. Denote by $v_{\gamma_i}$ a nonzero highest weight vector in $V_{\gamma_i}$ and by $<\cdot ,\cdot >$ the Hermitian form on $V_{\gamma_i}$ normalized in such a way that $<v_{\gamma_i},v_{\gamma_i}>=1$.

We claim that an element $g\in G$ belongs to $\overline{N}ZN$ iff $<v_{\gamma_i},gv_{\gamma_i}>=1$, $i=1,\ldots ,l'$. Indeed, according to the Bruhat decomposition $g\in B_-wB_+$ for some $w\in W$. In this case $g$ can be written uniquely in the form $g=n_-whn_+$ for some $n_\pm \in N_\pm$, $h\in H$. Now $<v_{\gamma_i},gv_{\gamma_i}>=\gamma_i(h)<v_{\gamma_i},wv_{\gamma_i}>$. As different weight spaces of $V_{\gamma_i}$ are orthogonal with respect to the Hermitian form, the right hand side of the last identity is not zero for all $i=1,\ldots ,l'$ iff $w$ fixes all roots $\gamma_i$, i.e. iff $w$ belongs to the Weyl group of the root subsystem $\overline{\Delta}_0$. Since $\overline{\Delta}_0$ is the root system of the Levi factor $L=ZH'$, and $<v_{\gamma_i},v_{\gamma_i}>=1$, one has $<v_{\gamma_i},wv_{\gamma_i}>\neq 0$ iff $g\in \overline{N}ZH'N$, and in that case $<v_{\gamma_i},gv_{\gamma_i}>=\gamma_i(h)$. Finally the conditions $<v_{\gamma_i},gv_{\gamma_i}>=\gamma_i(h)=1$, $i=1,\ldots ,l'$ are equivalent to the fact that $h$ is fixed by $s_{\gamma_i}$, $i=1,\ldots ,l'$ and by Lemma 1 in \cite{C} this is true iff $h$ is fixed by $s$. In this case $g\in \overline{N}ZN$.
 
Thus $\overline{N}ZN$ is closed in $G$. The variety $s^{-1}N_s'sZN$ is a closed subvariety of $\overline{N}ZN$ as $s^{-1}N_s's$ is the closed algebraic subgroup in $\overline{N}$ generated by the one--parametric subgroups corresponding to the roots from the set $\{\alpha \in  -\Delta_+: s(\alpha)\in \Delta_+\}$. So finally $s^{-1}N_s'sZN$ is closed in $G$, and hence $NsZN=N_s'sZN$ is also closed. 

Therefore the closure $\overline{\lambda(\mu_{M_+}^{-1}(u))}$ is contained in $NsZN$. This completes the proof.

\end{proof}

Now we can prove that the reduced space $\pi(\mu_{M_+}^{-1}(u))$ is smooth.
\begin{theorem}\label{var}
Suppose that the numbers $t_{i}$ defined in (\ref{defu}) are not equal to zero for all $i$. Then the (locally defined) action of $M_+$ on $\mu_{M_+}^{-1}(u)$ is (locally) free, the quotient $\pi(\mu_{M_+}^{-1}(u))$ is a smooth manifold of dimension ${\rm dim}~sZN_s$ and is a finite covering of $\pi_q(q(\mu_{M_+}^{-1}(u)))\subset sZN_s$. Assume also that the roots $\gamma_1, \ldots , \gamma_n$ are simple or the set $\gamma_1, \ldots , \gamma_n$ is empty. Then the algebra $W^s(G)$ is isomorphic to the algebra of $M_+$--invariant regular functions on $\mu_{M_+}^{-1}(u)\subset G^*$, $W^s(G)=\mathbb{C}[\mu_{M_+}^{-1}(u)]^{M_+}$.
\end{theorem}
\begin{proof}
First observe that by construction the map $q$ is a finite covering over its image and is equivariant with respect to the dressing action of $M_+$ on $G^*$ and to the conjugation action of $M_+$ on $G$. The image $q(\mu_{M_+}^{-1}(u))$ is contained in $NsZN$, and the conjugation action of $M_+\subset N$ is free on $NsZN=N\times sZN_s$. Therefore $\pi_q(q(\mu_{M_+}^{-1}(u)))$ is smooth, and hence  $\pi(\mu_{M_+}^{-1}(u))$ is smooth. By Proposition \ref{crosssect} $\pi_q(q(\mu_{M_+}^{-1}(u)))$ can be regarded as a subvariety of $sZN_s$. Hence $\pi(\mu_{M_+}^{-1}(u))$ is a finite covering of $\pi_q(q(\mu_{M_+}^{-1}(u)))\subset sZN_s$.

From formula (\ref{dimm}) for the cardinality $\sharp \Delta_{\m_+}$ of the set $\Delta_{\m_+}$ and from the definition of $\mu_{M_+}^{-1}(u)$ we deduce that the dimension of the quotient $\pi(\mu_{M_+}^{-1}(u))$ is equal to the dimension of the variety $sZN_s$,
\begin{eqnarray*}
{\rm dim}~\pi(\mu_{M_+}^{-1}(u))={\rm dim}~G-2{\rm dim}~M_+=2D+l-2\sharp \Delta_{\m_+} =2D+l- \\ -2(D-\frac{l(s)-l'}{2}-D_0)=l(s)+2D_0+l-l'={\rm dim}~N_s+{\rm dim}~Z={\rm dim}~sZN_s.
\end{eqnarray*}

Now observe that by Remark
\ref{remred} and by Proposition \ref{centralz}
the map
$$
C^\infty(\mu_{M_+}^{-1}(u))^{M_+}=C^\infty(\pi(\mu_{M_+}^{-1}(u)))\rightarrow
C^\infty(\mu_{M_+}^{-1}(u))^{C^\infty(M_-)},~~\psi \mapsto \pi^*\psi
$$
is an isomorphism.

Therefore the induced map
$$
\mathbb{C}[\mu_{M_+}^{-1}(u)]^{M_+}\rightarrow {\mathbb{C}}[\mu_{M_+}^{-1}(u)]\cap
C^\infty(\mu_{M_+}^{-1}(u))^{C^\infty(M_-)}
$$
is an isomorphism.

Finally observe that by Lemma \ref{redreg} the algebra ${\mathbb{C}}[\mu_{M_+}^{-1}(u)]\cap
C^\infty(\mu_{M_+}^{-1}(u))^{C^\infty(M_-)}$ is isomorphic to $W^s(G)$.
This completes the proof.
\end{proof}

\begin{remark}
A similar theorem can be proved in case when the roots  $\gamma_{n+1}, \ldots , \gamma_{l'}$ are simple or the set $\gamma_{n+1}, \ldots , \gamma_{l'}$ is empty. In that case instead of the map $q:G^*\rightarrow G$ one should use another map $q':G^*\rightarrow G$, $q'(L_+,L_-)=L_-^{-1}L_+$ which has the same properties as $q$, see \cite{dual}, Section 2.
\end{remark}

Observe that each central element $z\in Z({\mathbb{C}}_\varepsilon[G^*])$ obviously gives rise to an element $\rho_{\chi_\varepsilon^s}(z)\in {\mathbb{C}}_\varepsilon[G^*]\otimes_{{\mathbb{C}}_\varepsilon[M_-]}\mathbb{C}_{\chi_\varepsilon^{s}}$, and since $z$ is central
\begin{eqnarray*}
\rho_{\chi_\varepsilon^s}(z)\in {\rm Hom}_{{\mathbb{C}}_\varepsilon[M_-]}(\mathbb{C}_{\chi_\varepsilon^{s}},{\mathbb{C}}_\varepsilon[G^*]\otimes_{{\mathbb{C}}_\varepsilon[M_-]}
\mathbb{C}_{\chi_\varepsilon^{s}})= \\ ={\rm End}_{{\mathbb{C}}_\varepsilon[G^*]}({\mathbb{C}}_\varepsilon[G^*]\otimes_{{\mathbb{C}}_\varepsilon[M_-]}
\mathbb{C}_{\chi_\varepsilon^{s}})^{opp} =W_\varepsilon^s(G).
\end{eqnarray*}

The proof of the following proposition is similar to that of Theorem $\rm A_h$ in \cite{S7}.

\begin{proposition}
Let $\varepsilon\in \mathbb{C}$ be generic. Then
the restriction of the linear map $\rho_{\chi_\varepsilon^s}:{\mathbb{C}}_\varepsilon[G^*]\rightarrow {\mathbb{C}}_\varepsilon[G^*]\otimes_{{\mathbb{C}}_\varepsilon[M_-]}
\mathbb{C}_{\chi_\varepsilon^{s}}$ to the center $Z({\mathbb{C}}_\varepsilon[G^*])$ of ${\mathbb{C}}_\varepsilon[G^*]$ gives rise to an injective homomorphism of algebras,
$$
\rho_{\chi_\varepsilon^s}:Z({\mathbb{C}}_\varepsilon[G^*])\rightarrow W_\varepsilon^s(G).
$$
\end{proposition}


\section{Homological realization of q-W algebras}\label{hreal}

\setcounter{equation}{0}
\setcounter{theorem}{0}

In this section we realize the algebra $W_q^s(G)$ in the spirit of homological BRST reduction.

Let $A$ be an associative algebra over a unital ring $\bf k$,
$A_0\subset A$ a subalgebra with augmentation $\e : A_0 \ra \bf k$. We denote this rank one $A_0$--module by $\ke$.

Let $\Alm$ ($A_0-{\rm mod}$) be the category of left $A$ ($A_0$) modules. Denote by ${\rm Ind}_{A_0}^A$ the functor of induction,
$$
{\rm Ind}_{A_0}^A: A_0-{\rm mod}\ra \Alm$$
defined on objects by
$$
{\rm Ind}_{A_0}^A(V)=A\otimes_{A_0}V,~~V\in A_0-{\rm mod}.
$$

Let $\DAl$ ($D^-(A_0)$) be the derived category of the abelian category $\Alm$ ($A_0-{\rm mod}$) whose objects are bounded from above complexes of left $A$ ($A_0$) modules.
Let $({\rm Ind}_{A_0}^A)^L: D^-(A_0) \ra \DAl$ be the left derived functor of the functor of induction. Recall that if $V^{\gr}\in D^-(A_0)$ then $({\rm Ind}_{A_0}^A)^L(V^{\gr})=A\otimes_{A_0}X^{\gr}$, where
$X^{\gr}$ is a projective resolution of the complex $V^{\gr}$.

The $\mathbb Z$--graded algebra
\begin{equation}\label{heckedefo}
\Hk = \bigoplus_{n\in {\mathbb Z}}{\rm Hom}_{D^-(A)}(({\rm Ind}_{A_0}^A)^L(\ke),T^n(({\rm Ind}_{A_0}^A)^L(\ke))),
\end{equation}
where $T$ is the grading shift functor, is called in \cite{S2,S4} the Hecke algebra of the triple $(A,A_0,\e)$.

For every left $A$--module $V$ and right $A$--module $W$ the algebra $\Hk$ naturally acts in the spaces ${\rm Ext}^{\gr}_{A_0}(\ke,V)$ and ${\rm Tor}^{\gr}_{A_0}(W,\ke)$, from the right and from the left, respectively (see \cite{S2,S4} for details).

Note that if $H^{\gr}(({\rm Ind}_{A_0}^A)^L(\ke))={\rm Tor}_{A_0}^{\gr}(A,\ke )=A\otimes_{A_0}\ke$ the object \\ $({\rm Ind}_{A_0}^A)^L(\ke)\in \DAl$ is isomorphic in $\DAl$ to the complex $\ldots \ra 0 \ra A\otimes_{A_0}\ke \ra 0 \ra \ldots$ (with $A\otimes_{A_0}\ke$ at the 0-th place) and hence
\begin{equation}\label{hk}
{\rm Hk}^\gr(A,A_0,\e)={\rm Ext}^{\gr}_A(A\otimes_{A_0}\ke,A\otimes_{A_0}\ke).
\end{equation}
In particular,
\begin{equation}\label{hk<}
{\rm Hk}^n(A,A_0,\e)=0,~n<0,
\end{equation}
and the 0-th graded component of the algebra $\Hk$ takes the form
\begin{equation}\label{hkoo}
{\rm Hk}^0(A,A_0,\e)=\HA(A\otimes_{A_0}\ke,A\otimes_{A_0}\ke).
\end{equation}

In view of definition (\ref{QW}) and the last formula it is natural to consider the Hecke algebra associated to the triple $({\mathbb{C}}_\mathcal{A}[G^*], {\mathbb{C}}_\mathcal{A}[M_-], \chi_q^{s})$ and its specialization $({\mathbb{C}}_\varepsilon[G^*], {\mathbb{C}}_\varepsilon[M_-], \chi_\varepsilon^{s})$.
\begin{theorem}\label{brst}
Assume that the roots $\gamma_1, \ldots , \gamma_n$ (or $\gamma_{n+1},\ldots, \gamma_{l'}$) are simple.
Then
$$
{\rm Hk}^n(({\mathbb{C}}_\varepsilon[G^*], {\mathbb{C}}_\varepsilon[M_-], \chi_\varepsilon^{s}))=0,~n< 0,
$$
and
$$
{\rm Hk}^0(({\mathbb{C}}_\varepsilon[G^*], {\mathbb{C}}_\varepsilon[M_-], \chi_\varepsilon^{s}))={\rm End}_{{\mathbb{C}}_\varepsilon[G^*]}({\mathbb{C}}_\varepsilon[G^*]\otimes_{{\mathbb{C}}_\varepsilon[M_-]}\mathbb{C}_{\chi_\varepsilon^{s}})=W_\varepsilon^s(G)^{opp}.
$$
\end{theorem}

\begin{proof}
First observe that the definition of the algebra ${\mathbb{C}}_\varepsilon[G^*]$ implies that it has a Poincar\'{e}--Birkhoff--Witt basis similar to that for $U_\varepsilon(\g)$. Since the roots $\gamma_1, \ldots , \gamma_n$ are simple the segment $\Delta_{\m_+}$ has a complementary segment $\Delta_{\m_+}^0$ in $\Delta_+$. This fact and  the existence of the Poincar\'{e}--Birkhoff--Witt basis for  ${\mathbb{C}}_\varepsilon[G^*]$ imply that ${\mathbb{C}}_\varepsilon[G^*]$ is a free ${\mathbb{C}}_\varepsilon[M_-]$--module with respect to multiplication by elements from ${\mathbb{C}}_\varepsilon[M_-]$ on ${\mathbb{C}}_\varepsilon[G^*]$ from the right. Therefore $${\rm Tor}_{{\mathbb{C}}_\varepsilon[M_-]}^{\gr}({\mathbb{C}}_\varepsilon[G^*], \mathbb{C}_{\chi_\varepsilon^{s}})={\mathbb{C}}_\varepsilon[G^*]\otimes_{{\mathbb{C}}_\varepsilon[M_-]}\mathbb{C}_{\chi_\varepsilon^{s}},$$ and the discussion of the general properties of Hecke algebras in the beginning of this section gives
$$
{\rm Hk}^n({\mathbb{C}}_\varepsilon[G^*], {\mathbb{C}}_\varepsilon[M_-], \chi_\varepsilon^{s})=0,~n<0,
$$
and
$$
{\rm Hk}^0({\mathbb{C}}_\varepsilon[G^*], {\mathbb{C}}_\varepsilon[M_-], \chi_\varepsilon^{s})={\rm End}_{{\mathbb{C}}_\varepsilon[G^*]}({\mathbb{C}}_\varepsilon[G^*]\otimes_{{\mathbb{C}}_\varepsilon[M_-]}\mathbb{C}_{\chi_\varepsilon^{s}})=W_\varepsilon^s(G)^{opp}.
$$
(see formulas (\ref{hk<}) and (\ref{hkoo})).
This completes the proof.

\end{proof}

Note that the higher components
$$
{\rm Hk}^n(({\mathbb{C}}_\varepsilon[G^*], {\mathbb{C}}_\varepsilon[M_-], \chi_\varepsilon^{s})), n>0$$ of the algebra $
{\rm Hk}^\gr(({\mathbb{C}}_\varepsilon[G^*], {\mathbb{C}}_\varepsilon[M_-], \chi_\varepsilon^{s}))
$ are not equal to zero.

Consider, for instance, the case when $\g=\mathfrak{sl}_2$. In that case the only nontrivial element of the Weyl group is $s=-1$, and hence $\frac{1+s}{1-s}=0$. Assume that $\varepsilon$ is generic. Then  the algebra $\mathbb{C}_\varepsilon[SL_2^*]$ is
the complex associative algebra with generators $e,f,t^{\pm1}$ subject to the relations
\begin{equation}
tt^{-1}=t^{-1}t=1,~~ tet^{-1}=\varepsilon e_j, ~~tft^{-1}=\varepsilon^{-1}f,~
e f - f e = {t^2 -t^{-2} \over \varepsilon -\varepsilon^{-1}}.
\end{equation}
The subalgebra $\mathbb{C}_\varepsilon[M_-]$ is generated by the element $e$. Fix a character $\chi_\varepsilon^{s}$ of $\mathbb{C}_\varepsilon[M_-]$ defined by $\chi_\varepsilon^{s}(e)=1$.

The center of the algebra $\mathbb{C}_\varepsilon[SL_2^*]$ is generated by the element
$$
\Omega=\frac{\varepsilon t^2+\varepsilon^{-1}t^{-2}}{(\varepsilon-\varepsilon^{-1})^2}+fe.
$$

Denote by $v$ the image of $1\in {\mathbb{C}}_\varepsilon[SL_2^*]$ in the left ${\mathbb{C}}_\varepsilon[SL_2^*]$--module
$Q={\mathbb{C}}_\varepsilon[SL_2^*]\otimes_{{\mathbb{C}}_\varepsilon[M_-]}\mathbb{C}_{\chi_\varepsilon^{s}}$.
One checks straightforwardly that the elements $v_{mk}=t^m\Omega^kv$, $m \in \mathbb{Z},k\in \mathbb{N}$ form a linear basis of $Q$. It follows that as a ${\mathbb{C}}_\varepsilon[M_-]$--module $Q$ is the direct sum of one--dimensional modules $\mathbb{C}_{mk}=\mathbb{C}v_{mk}$, and the element $e$ acts on $\mathbb{C}_{mk}$ by multiplication by $\varepsilon^{-m}$. Now by Frobenius reciprocity
$$
{\rm Hk}^n({\mathbb{C}}_\varepsilon[SL_2^*], {\mathbb{C}}_\varepsilon[M_-], \chi_\varepsilon^{s})={\rm Ext}^{n}_{{\mathbb{C}}_\varepsilon[M_-]}(\mathbb{C}_{\chi_\varepsilon^{s}},{\mathbb{C}}_\varepsilon[SL_2^*]\otimes_{{\mathbb{C}}_\varepsilon[M_-]}\mathbb{C}_{\chi_\varepsilon^{s}}),
$$
and from the description of $Q$ as a left ${\mathbb{C}}_\varepsilon[M_-]$--module given above we obtain linear space isomorphisms
$$
{\rm Hk}^0({\mathbb{C}}_\varepsilon[SL_2^*], {\mathbb{C}}_\varepsilon[M_-], \chi_\varepsilon^{s})\simeq \bigoplus_{m\in \mathbb{Z},\varepsilon^m=1,k\in \mathbb{N}}\mathbb{C}_{mk}=\bigoplus_{k\in \mathbb{N}}\mathbb{C}_{0k},
$$
$$
{\rm Hk}^1({\mathbb{C}}_\varepsilon[SL_2^*], {\mathbb{C}}_\varepsilon[M_-], \chi_\varepsilon^{s})\simeq
\bigoplus_{m\in \mathbb{Z},k\in \mathbb{N}}\mathbb{C}_{mk}/(\varepsilon^m-1)\mathbb{C}_{mk}=\bigoplus_{m\in \mathbb{Z},\varepsilon^m=1,k\in \mathbb{N}}\mathbb{C}_{mk}=\bigoplus_{k\in \mathbb{N}}\mathbb{C}_{0k}.
$$
The other components of the algebra ${\rm Hk}^\gr({\mathbb{C}}_\varepsilon[SL_2^*], {\mathbb{C}}_\varepsilon[M_-], \chi_\varepsilon^{s})$ vanish. But as we see $${\rm Hk}^1({\mathbb{C}}_\varepsilon[SL_2^*], {\mathbb{C}}_\varepsilon[M_-], \chi_\varepsilon^{s})\neq 0.$$ This is related to the fact that the action of ${\mathbb{C}}_\varepsilon[M_-]$ on $Q$ is semisimple while in the Lie algebra case a similar action is nilpotent.


\section*{Appendix A. Normal orderings of root systems compatible with involutions in Weyl groups}

\setcounter{equation}{0}
\setcounter{theorem}{0}

By Theorem A in \cite{Ric} every involution $w$ in the Weyl group $W$ of the pair $(\g,\h)$ is the longest element of the Weyl group of a Levi subalgebra in $\g$ with respect to some system of positive roots, and $w$ acts by multiplication by $-1$ in the Cartan subalgebra $\h_w\subset \h$ of the semisimple part $\m_w$ of that Levi subalgebra. By Lemma 5 in \cite{C} the involution $w$ can also be expressed as a product of ${\rm dim}~\h_w$ reflections from the Weyl group of the pair $(\m_w,\h_w)$, with respect to mutually orthogonal roots, $w=s_{\gamma_1}\ldots s_{\gamma_n}$, and the roots ${\gamma_1},\ldots ,{\gamma_n}$ span the subalgebra $\h_{w}$.

If $w$ is the longest element in the Weyl group of the pair $(\m_w,\h_w)$ with respect to some system of positive roots, where $\m_w$ is a simple Lie algebra and $\h_w$ is a Cartan subalgebra of $\m_w$, then $w$ is an involution acting by multiplication by $-1$ in $\h_w$ if and only if $\m_w$ is of one of the following types: $A_1,B_l,C_l,D_{2n},E_7,E_8,F_4,G_2$.

Fix a system of positive roots $\Delta_+(\m_w,\h_w)$ of the pair $(\m_w,\h_w)$. Let $w=s_{\gamma_1}\ldots s_{\gamma_n}$ be a representation of $w$ as a product of ${\rm dim}~\h_w$ reflections from the Weyl group of the pair $(\m_w,\h_w)$, with respect to mutually orthogonal positive roots. A normal ordering of $\Delta_+(\m_w,\h_w)$ is called compatible with the decomposition $w=s_{\gamma_1}\ldots s_{\gamma_n}$ if it is of the following form

$$
\beta_{1}, \ldots,\beta_{\frac{p-n}{2}}, \gamma_1,\beta_{\frac{p-n}{2}+2}^1, \ldots , \beta_{\frac{p-n}{2}+n_1}^1, \gamma_2,
\beta_{\frac{p-n}{2}+n_1+2}^1 \ldots , \beta_{\frac{p-n}{2}+n_2}^1, \gamma_3,\ldots, \gamma_n,
$$
where $p$ is the number of positive roots, and for any two positive roots $\alpha, \beta\in \Delta_+(\m_w,\h_w)$ such that $\gamma_1\leq \alpha<\beta$ the sum $\alpha+\beta$ cannot be represented as a linear combination $\sum_{k=1}^qc_k\gamma_{i_k}$, where $c_k\in \mathbb{N}$ and $\alpha<\gamma_{i_1}<\ldots <\gamma_{i_k}<\beta$.

Existence of such compatible normal orderings is checked straightforwardly for all simple Lie algebras of types $A_1,B_l,C_l,D_{2n},E_7,E_8,F_4$ and $G_2$. In case $A_1$ this is obvious since there is only one positive root. In the other cases normal orderings defined by the properties described below for each of the types $B_l,C_l,D_{2n},E_7,E_8,F_4,G_2$ exist and are compatible with decompositions of nontrivial involutions in Weyl group.
We use Bourbaki notation for the systems of positive and simple roots (see \cite{Bur}).

\begin{itemize}

\vskip 0.5cm

\item

$B_l$

\vskip 0.2cm

Dynkin diagram:

$$
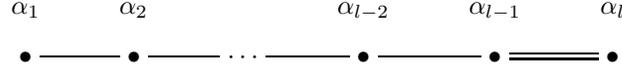
\xymatrix@R=.25cm{
\alpha_1&\alpha_2&&\alpha_{l-2}&\alpha_{l-1}&\alpha_l\\
{\bullet}\ar@{-}[r]&{\bullet}\ar@{-}[r]&\cdots
\ar@{-}[r]& {\bullet}\ar@{-}[r]&{\bullet}\ar@2{-}[r] &{\bullet}
}$$

Simple roots: $\alpha_1=\varepsilon_1-\varepsilon_2,\alpha_2=\varepsilon_2-\varepsilon_3,\ldots, \alpha_{l-1}=\varepsilon_{l-1}-\varepsilon_l,\alpha_l=\varepsilon_l$.

Positive roots: $\varepsilon_i$ $(1\leq i\leq l)$, $\varepsilon_i-\varepsilon_j,\varepsilon_i+\varepsilon_j$ $(1\leq i<j\leq l)$.

The longest element of the Weyl group expressed as a product of ${\rm dim}~\h_w$ reflections with respect to mutually orthogonal roots: $w=s_{\varepsilon_1}\ldots s_{\varepsilon_l}$.

Normal ordering of $\Delta_+(\m_w,\h_w)$ compatible with expression $w=s_{\varepsilon_1}\ldots s_{\varepsilon_l}$:

$$
\varepsilon_1-\varepsilon_2,\ldots,\varepsilon_{l-1}-\varepsilon_l,\varepsilon_1,\ldots,\varepsilon_2,\ldots,
\varepsilon_l,
$$
where the roots $\varepsilon_i-\varepsilon_j$ $(1\leq i<j\leq l)$ forming the subsystem $\Delta_+(A_{l-1})\subset \Delta_+(B_l)$ are situated to the left from $\varepsilon_1$, and the roots $\varepsilon_i+\varepsilon_j$ $(1\leq i<j\leq l)$ are situated to the right from $\varepsilon_1$.

\vskip 0.5cm

\item

$C_l$

\vskip 0.2cm

Dynkin diagram:

$$
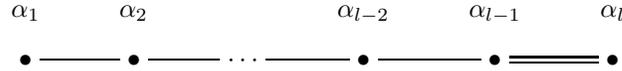
\xymatrix@R=.25cm{
\alpha_1&\alpha_2&&\alpha_{l-2}&\alpha_{l-1}&\alpha_l\\
{\bullet}\ar@{-}[r]&{\bullet}\ar@{-}[r]&\cdots
\ar@{-}[r]& {\bullet}\ar@{-}[r]&{\bullet}\ar@2{-}[r] &{\bullet}
}$$

Simple roots: $\alpha_1=\varepsilon_1-\varepsilon_2,\alpha_2=\varepsilon_2-\varepsilon_3,\ldots, \alpha_{l-1}=\varepsilon_{l-1}-\varepsilon_l,\alpha_l=2\varepsilon_l$.

Positive roots: $2\varepsilon_i$ $(1\leq i\leq l)$, $\varepsilon_i-\varepsilon_j,\varepsilon_i+\varepsilon_j$ $(1\leq i<j\leq l)$.

The longest element of the Weyl group expressed as a product of ${\rm dim}~\h_w$ reflections with respect to mutually orthogonal roots: $w=s_{2\varepsilon_1}\ldots s_{2\varepsilon_l}$.

Normal ordering of $\Delta_+(\m_w,\h_w)$ compatible with expression $w=s_{2\varepsilon_1}\ldots s_{2\varepsilon_l}$:

$$
\varepsilon_1-\varepsilon_2,\ldots,\varepsilon_{l-1}-\varepsilon_l,2\varepsilon_1,\ldots,2\varepsilon_2,\ldots,
2\varepsilon_l,
$$
where the roots $\varepsilon_i-\varepsilon_j$ $(1\leq i<j\leq l)$ forming the subsystem $\Delta_+(A_{l-1})\subset \Delta_+(C_l)$ are situated to the left from $2\varepsilon_1$, and the roots $\varepsilon_i+\varepsilon_j$ $(1\leq i<j\leq l)$ are situated to the right from $2\varepsilon_1$.

\vskip 0.5cm

\item

$D_{2n}$

\vskip 0.2cm

Dynkin diagram:

$$
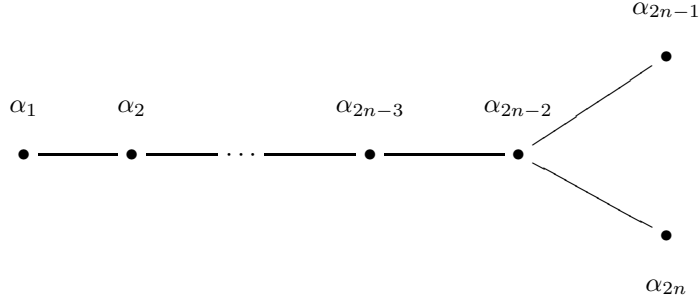
\xymatrix@R=.25cm{
&&&&&\alpha_{2n-1}\\
&&&&&{\bullet}\\
\alpha_1&\alpha_2&&\alpha_{2n-3}& \alpha_{2n-2}&\\
{\bullet}\ar@{-}[r]&{\bullet}\ar@{-}[r]&\cdots
\ar@{-}[r]&{\bullet}\ar@{-}[r]& {\bullet}\ar@{-}[uur]\ar@{-}[ddr]& \\
&&&&&\\
&&&&&{\bullet}\\
&&&&&\alpha_{2n}
}$$

Simple roots: $\alpha_1=\varepsilon_1-\varepsilon_2,\alpha_2=\varepsilon_2-\varepsilon_3,\ldots, \alpha_{2n-1}=\varepsilon_{2n-1}-\varepsilon_{2n},\alpha_{2n}=\varepsilon_{2n-1}+\varepsilon_{2n}$.

Positive roots: $\varepsilon_i-\varepsilon_j,\varepsilon_i+\varepsilon_j$ $(1\leq i<j\leq 2n)$.

The longest element of the Weyl group expressed as a product of ${\rm dim}~\h_w$ reflections with respect to mutually orthogonal roots: $$w=s_{\varepsilon_1-\varepsilon_2}s_{\varepsilon_1+\varepsilon_2}\ldots s_{\varepsilon_{2n-1}-\varepsilon_{2n}}s_{\varepsilon_{2n-1}+\varepsilon_{2n}}.$$

Normal ordering of $\Delta_+(\m_w,\h_w)$ compatible with expression $$w=s_{\varepsilon_1-\varepsilon_2}s_{\varepsilon_1+\varepsilon_2}\ldots s_{\varepsilon_{2n-1}-\varepsilon_{2n}}s_{\varepsilon_{2n-1}+\varepsilon_{2n}}:$$

\begin{eqnarray*}
\varepsilon_2-\varepsilon_3,\varepsilon_4-\varepsilon_5,\ldots,\varepsilon_{2n-2}-\varepsilon_{2n-1},\ldots,
\varepsilon_1-\varepsilon_2,\varepsilon_3-\varepsilon_4,\ldots,\varepsilon_{2n-1}-\varepsilon_{2n-2}, \\
\varepsilon_1+\varepsilon_2,\ldots,\varepsilon_3+\varepsilon_4,\ldots,
\varepsilon_{2n-1}+\varepsilon_{2n},
\end{eqnarray*}
where the roots $\varepsilon_i-\varepsilon_j$ $(1\leq i<j\leq l)$ forming the subsystem $\Delta_+(A_{l-1})\subset \Delta_+(C_l)$ are situated to the left from $\varepsilon_1+\varepsilon_2$, and the roots $\varepsilon_i+\varepsilon_j$ $(1\leq i<j\leq l)$ are situated to the right from $\varepsilon_1+\varepsilon_2$.

\vskip 0.5cm

\item

\vskip 0.2cm

$E_7$

Dynkin diagram:

$$
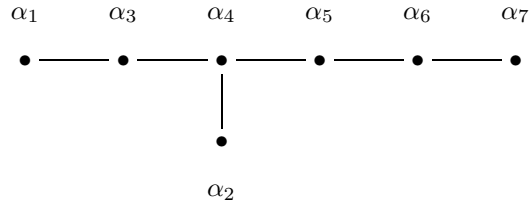
\xymatrix@R=.25cm@C=.25cm{
&\alpha_1&&\alpha_3&&\alpha_4&&\alpha_5&&\alpha_6&&\alpha_7\\
&{\bullet}\ar@{-}[rr]&&{\bullet}\ar@{-}[rr]&&{\bullet}
\ar@{-}[dd]\ar@{-}[rr]&& {\bullet}\ar@{-}[rr] &&{\bullet}\ar@{-}[rr]
&&{\bullet}\\ &&&&&&&&&&&\\
&&&&&{\bullet}&&&&&&\\
&&&&&\alpha_2&&&&&&}$$

Simple roots: $\alpha_1=\frac{1}{2}(\varepsilon_1+\varepsilon_8)-
\frac{1}{2}(\varepsilon_2+\varepsilon_3+\varepsilon_4+\varepsilon_5+\varepsilon_6+\varepsilon_7)$,
$\alpha_2=\varepsilon_1+\varepsilon_2,\alpha_3=\varepsilon_2-\varepsilon_1,$ $ \alpha_4=\varepsilon_3-\varepsilon_2,\alpha_5=\varepsilon_4-\varepsilon_3,
\alpha_6=\varepsilon_5-\varepsilon_4, \alpha_7=\varepsilon_6-\varepsilon_5$.

Positive roots: $\pm\varepsilon_i+\varepsilon_j$ $(1\leq i<j\leq 6)$, $\varepsilon_8-\varepsilon_7$,
$\frac{1}{2}(\varepsilon_8-\varepsilon_7+\sum_{i=1}^6(-1)^{\nu(i)}\varepsilon_i)$ with $\sum_{i=1}^6{\nu(i)}$ odd.

The longest element of the Weyl group expressed as a product of ${\rm dim}~\h_w$ reflections with respect to mutually orthogonal roots: $$w=s_{\varepsilon_2-\varepsilon_1}s_{\varepsilon_2+\varepsilon_1}
s_{\varepsilon_4-\varepsilon_3}s_{\varepsilon_4+\varepsilon_3} s_{\varepsilon_{6}-\varepsilon_{5}}s_{\varepsilon_{6}+\varepsilon_{5}}s_{\varepsilon_8-\varepsilon_7}.$$

Normal ordering of $\Delta_+(\m_w,\h_w)$ compatible with expression $$w=s_{\varepsilon_2-\varepsilon_1}s_{\varepsilon_2+\varepsilon_1}
s_{\varepsilon_4-\varepsilon_3}s_{\varepsilon_4+\varepsilon_3} s_{\varepsilon_{6}-\varepsilon_{5}}s_{\varepsilon_{6}+\varepsilon_{5}}s_{\varepsilon_8-\varepsilon_7}:$$

\begin{eqnarray*}
\alpha_1,\varepsilon_3-\varepsilon_2,\varepsilon_5-\varepsilon_4,\ldots,\varepsilon_{8}-\varepsilon_{7},\ldots,
\varepsilon_2-\varepsilon_1,\varepsilon_4-\varepsilon_3,\varepsilon_6-\varepsilon_5,\ldots, \\
\varepsilon_{6}+\varepsilon_{5},\ldots,\varepsilon_4+\varepsilon_3,\ldots,\varepsilon_2+\varepsilon_1,
\end{eqnarray*}
where the roots $\pm\varepsilon_i+\varepsilon_j$ $(1\leq i<j\leq 6)$ forming the subsystem $\Delta_+(D_6)\subset \Delta_+(E_7)$ are placed as in case of the compatible normal ordering of the system $\Delta_+(D_6)$, the only roots from the subsystem $\Delta_+(A_5)\subset \Delta_+(D_6)$ situated to the right from the maximal root $\varepsilon_{8}-\varepsilon_{7}$ are $\varepsilon_2-\varepsilon_1,\varepsilon_4-\varepsilon_3,\varepsilon_6-\varepsilon_5$, the roots $\varepsilon_i+\varepsilon_j$ $(1\leq i<j\leq 6)$ are situated to the right from $\varepsilon_6+\varepsilon_5$, and a half of the positive roots which do not belong to the subsystem $\Delta_+(D_6)\subset \Delta_+(E_7)$ are situated to the left from $\varepsilon_{8}-\varepsilon_{7}$ and the other half of those roots are situated to the right from $\varepsilon_{8}-\varepsilon_{7}$.

\vskip 0.5cm

\item

$E_8$

\vskip 0.2cm

Dynkin diagram:

$$
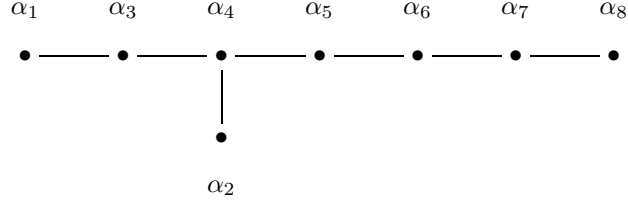
\xymatrix@R=.25cm@C=.25cm{
&\alpha_1&&\alpha_3&&\alpha_4&&\alpha_5&&\alpha_6&&\alpha_7&&\alpha_8\\
&{\bullet}\ar@{-}[rr]&&{\bullet}\ar@{-}[rr]&&{\bullet}
\ar@{-}[dd]\ar@{-}[rr]&& {\bullet}\ar@{-}[rr] &&{\bullet}\ar@{-}[rr]
&&{\bullet}\ar@{-}[rr]&& {\bullet}\\ &&&&&&&&&&&&&\\
&&&&&{\bullet}&&&&&&&& \\
&&&&&\alpha_2&&&&&&&&}$$

Simple roots: $\alpha_1=\frac{1}{2}(\varepsilon_1+\varepsilon_8)-
\frac{1}{2}(\varepsilon_2+\varepsilon_3+\varepsilon_4+\varepsilon_5+\varepsilon_6+\varepsilon_7)$,
$\alpha_2=\varepsilon_1+\varepsilon_2,\alpha_3=\varepsilon_2-\varepsilon_1,$ $ \alpha_4=\varepsilon_3-\varepsilon_2,\alpha_5=\varepsilon_4-\varepsilon_3,
\alpha_6=\varepsilon_5-\varepsilon_4, \alpha_7=\varepsilon_6-\varepsilon_5, \alpha_8=\varepsilon_7-\varepsilon_6$.

Positive roots: $\pm\varepsilon_i+\varepsilon_j$ $(1\leq i<j\leq 8)$,
$\frac{1}{2}(\varepsilon_8+\sum_{i=1}^7(-1)^{\nu(i)}\varepsilon_i)$ with $\sum_{i=1}^7{\nu(i)}$ even.

The longest element of the Weyl group expressed as a product of ${\rm dim}~\h_w$ reflections with respect to mutually orthogonal roots: $$w=s_{\varepsilon_2-\varepsilon_1}s_{\varepsilon_2+\varepsilon_1}
s_{\varepsilon_4-\varepsilon_3}s_{\varepsilon_4+\varepsilon_3} s_{\varepsilon_{6}-\varepsilon_{5}}s_{\varepsilon_{6}+\varepsilon_{5}}
s_{\varepsilon_8-\varepsilon_7}s_{\varepsilon_8+\varepsilon_7}.$$

Normal ordering of $\Delta_+(\m_w,\h_w)$ compatible with expression $$w=s_{\varepsilon_2-\varepsilon_1}s_{\varepsilon_2+\varepsilon_1}
s_{\varepsilon_4-\varepsilon_3}s_{\varepsilon_4+\varepsilon_3} s_{\varepsilon_{6}-\varepsilon_{5}}s_{\varepsilon_{6}+\varepsilon_{5}}
s_{\varepsilon_8-\varepsilon_7}s_{\varepsilon_8-\varepsilon_7}:$$

\begin{eqnarray*}
\alpha_1,\varepsilon_3-\varepsilon_2,\varepsilon_5-\varepsilon_4,\varepsilon_{7}-\varepsilon_{6},\ldots,
\varepsilon_2-\varepsilon_1,\varepsilon_4-\varepsilon_3,\varepsilon_6-\varepsilon_5,
\\
\varepsilon_{8}-\varepsilon_{7},\ldots,\varepsilon_{8}+\varepsilon_{7},\ldots,
\varepsilon_{6}+\varepsilon_{5},\ldots,\varepsilon_4+\varepsilon_3,\ldots,\varepsilon_2+\varepsilon_1,
\end{eqnarray*}
where the roots $\pm\varepsilon_i+\varepsilon_j$ $(1\leq i<j\leq 8)$ forming the subsystem $\Delta_+(D_8)\subset \Delta_+(E_8)$ are placed as in case of the compatible normal ordering of the system $\Delta_+(D_8)$, the roots $\varepsilon_i+\varepsilon_j$ $(1\leq i<j\leq 8)$ are situated to the right from $\varepsilon_8+\varepsilon_7$; the positive roots which do not belong to the subsystem $\Delta_+(D_8)\subset \Delta_+(E_8)$ can be split into two groups: the roots from the first group contain $\frac{1}{2}(\varepsilon_8+\varepsilon_7)$ in their decompositions with respect to the basis $\varepsilon_i,i=1,\ldots,8$, and the roots from the second group contain $\frac{1}{2}(\varepsilon_8-\varepsilon_7)$ in their decompositions with respect to the basis $\varepsilon_i,i=1,\ldots,8$; a half of the roots from the first group are situated to the left from $\varepsilon_{2}-\varepsilon_{1}$ and the other half of those roots are situated to the right from $\varepsilon_{8}+\varepsilon_{7}$; a half of the roots from the second group are situated to the left from $\varepsilon_{2}-\varepsilon_{1}$ and the other half of those roots are situated to the right from $\varepsilon_{8}-\varepsilon_{7}$.

\vskip 0.5cm

\item

$F_4$

\vskip 0.2cm

Dynkin diagram:

$$
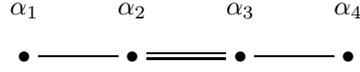
\xymatrix@R=.25cm{
\alpha_1&\alpha_2&\alpha_3&\alpha_4\\
{\bullet}\ar@{-}[r]&{\bullet}\ar@2{-}[r]
&{\bullet}\ar@{-}[r]&{\bullet} }$$

Simple roots: $\alpha_1=\varepsilon_2-\varepsilon_3,\alpha_2=\varepsilon_3-\varepsilon_4, \alpha_{3}=\varepsilon_{4},\alpha_4=\frac{1}{2}(\varepsilon_1-\varepsilon_2-\varepsilon_3-\varepsilon_4)$.

Positive roots: $\varepsilon_i$ $(1\leq i\leq 4)$, $\varepsilon_i-\varepsilon_j,\varepsilon_i+\varepsilon_j$ $(1\leq i<j\leq 4)$, $\frac{1}{2}(\varepsilon_1\pm\varepsilon_2\pm\varepsilon_3\pm\varepsilon_4)$.

The longest element of the Weyl group expressed as a product of ${\rm dim}~\h_w$ reflections with respect to mutually orthogonal roots: $w=s_{\varepsilon_1}s_{\varepsilon_2}s_{\varepsilon_3} s_{\varepsilon_4}$.

Normal ordering of $\Delta_+(\m_w,\h_w)$ compatible with expression $w=s_{\varepsilon_1}s_{\varepsilon_2}s_{\varepsilon_3} s_{\varepsilon_4}$:

$$
\alpha_4,\varepsilon_1-\varepsilon_2,\ldots,\varepsilon_{3}-\varepsilon_4,\ldots,\varepsilon_1,\ldots,\varepsilon_2,\ldots,
\varepsilon_4,
$$
where the roots $\varepsilon_i\pm \varepsilon_j$ $(1\leq i<j\leq l)$ forming the subsystem $\Delta_+(B_{4})\subset \Delta_+(F_4)$ are situated as in case of $B_4$, and a half of the positive roots which do not belong to the subsystem $\Delta_+(B_4)\subset \Delta_+(F_4)$ are situated to the left from $\varepsilon_{1}$ and the other half of those roots are situated to the right from $\varepsilon_{1}$.

\vskip 0.5cm

\item

$G_2$

\vskip 0.2cm

Dynkin diagram:

$$
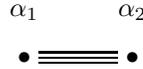
\xymatrix@R=.25cm{
\alpha_1&\alpha_2\\
{\bullet}\ar@3{-}[r] &{\bullet}
}$$

Simple roots: $\alpha_1=\varepsilon_1-\varepsilon_2,\alpha_2=-2\varepsilon_1+\varepsilon_2+\varepsilon_3$.

Positive roots: $\alpha_1,\alpha_1+\alpha_2, 2\alpha_1+\alpha_2, 3\alpha_1+\alpha_2,3\alpha_1+2\alpha_2,\alpha_2$.

The longest element of the Weyl group expressed as a product of ${\rm dim}~\h_w$ reflections with respect to mutually orthogonal roots: $w=s_{\alpha_1}s_{3\alpha_1+2\alpha_2}$.

Normal ordering of $\Delta_+(\m_w,\h_w)$ compatible with expression $w=s_{\alpha_1}s_{3\alpha_1+2\alpha_2}$:

$$
\alpha_2,\alpha_1+\alpha_2,3\alpha_1+2\alpha_2,2\alpha_1+\alpha_2,3\alpha_1+\alpha_2,\alpha_1.
$$

\end{itemize}

\end{document}